\documentclass[a4paper, 11pt]{amsart}
\usepackage[utf8]{inputenc}
\usepackage[T1]{fontenc}
\usepackage[english]{babel}
\usepackage{graphicx}
\usepackage{stmaryrd}
\usepackage{tikz}
\usepackage{tikz-cd} 
\usepackage[all]{xy}
\usepackage{comment}
\usepackage{bm}
\usepackage{amsmath,amsfonts,amssymb}
\usepackage[a4paper]{geometry}
\usepackage{amsthm}
\usepackage{float}
\usepackage{enumitem}
\usepackage{hyperref}
\usepackage{xcolor}
\hypersetup{
    colorlinks=true,
    linkcolor={blue},
    citecolor={blue},
    urlcolor={blue}}
\newtheorem{te}{Theorem}[section]

\newtheorem{prop}[te]{Proposition}
\newtheorem{pb}[te]{Problem}
\newtheorem{co}[te]{Corollary}
\newtheorem{conj}[te]{Conjecture}
\newtheorem{lemme}[te]{Lemma}
\theoremstyle{definition}

\newtheorem{de}[te]{Definition}
\newtheorem{ex}[te]{Example}

\theoremstyle{remark}
\newtheorem{rque}[te]{Remark}
\newenvironment{proofnash}{\noindent\textit{Proof of Proposition \ref{theorie2}.~}}{\hfill$\square$\bigbreak} 
\newenvironment{proofnashsac}{\noindent\textit{Proof of Proposition \ref{sacerdote1}.~}}{\hfill$\square$\bigbreak} 
\newenvironment{proofnashlemme}{\noindent\textit{Proof of Lemma \ref{abovelemma}.~}}{\hfill$\square$\bigbreak} 
\newlength{\plarg}
\usetikzlibrary{cd}
\calclayout
\title{On Tarski's problem for virtually free groups}
\author{Simon André}
\begin{document}
\clearpage
\begin{minipage}{\linewidth}
	\begin{abstract}
We give a complete classification of finitely generated virtually free groups up to $\forall\exists$-elementary equivalence. As a corollary, we give an algorithm that takes as input two finite presentations of virtually free groups, and decides whether these groups have the same $\forall\exists$-theory or not.
	\end{abstract}
	\maketitle
\end{minipage}

\section{Introduction}

\thispagestyle{empty}

The problem of classifying algebraic structures up to elementary equivalence emerged in the middle of the twentieth century. Around 1945, Tarski asked whether all non-abelian finitely generated free groups are elementarily equivalent. Two decades later, Merzlyakov made an important step forward by proving that free groups have the same \emph{positive theory}, i.e.\ satisfy the same first-order sentences without inequalities (see \cite{Mer66}). Sacerdote subsequently generalized Merzlyakov's result and proved in \cite{Sac73b} that all groups that split as a non-trivial free product have the same positive theory, except the infinite dihedral group $D_{\infty}=\mathbb{Z}/2\mathbb{Z}\ast\mathbb{Z}/2\mathbb{Z}$. In the same year, Sacerdote proved in \cite{Sac73} that free groups have the same $\forall\exists$\emph{-theory}, meaning that they satisfy the same sentences of the form $\forall\boldsymbol{x}\exists\boldsymbol{y} \psi(\boldsymbol{x},\boldsymbol{y})$, where $\boldsymbol{x}$ and $\boldsymbol{y}$ are two tuples of variables, and $\psi$ is a quantifier-free formula in these variables. Merzlyakov and Sacerdote's proofs rely heavily on small cancellation theory.

\smallskip 

Another major breakthrough towards the resolution of Tarski's problem was the study of systems of equations defined over a free group, due to Makanin and Razborov (see \cite{Mak82}, \cite{Mak84} and \cite{Raz84}).

\smallskip 

A positive answer to Tarski's question was eventually given by Sela in \cite{Sel06} and by Kharlampovich and Myasnikov in \cite{KM06}, as the culmination of two voluminous series of papers. 

\smallskip 

Sela further generalized his work and classified torsion-free hyperbolic groups up to elementary equivalence. His solution involves a study of the $\forall\exists$-theory of a given hyperbolic group (see \cite{Sel04} and \cite{Sel09}), combined with a quantifier elimination procedure down to $\forall\exists$-sentences (see \cite{Sel05a} and \cite{Sel05b}). The theory of group actions on real trees plays a crucial role in his approach (see for instance \cite{GLP94}, \cite{RS94}, \cite{BF95b}, \cite{Sel97}).

\smallskip 

In this paper, we give a complete classification of finitely generated virtually free groups up to $\forall\exists$-elementary equivalence, i.e.\ we give necessary and sufficient conditions for two finitely generated virtually free groups $G$ and $G'$ to have the same $\forall\exists$-theory, denoted by $\mathrm{Th}_{\forall\exists}(G)=\mathrm{Th}_{\forall\exists}(G')$. Recall that a group is said to be \emph{virtually free} if it has a free subgroup of finite index. For instance, it is well-known that $\mathrm{SL}_2(\mathbb{Z})$ has a subgroup of index 12 isomorphic to the free group $F_2$.


\smallskip 

Among virtually free groups, a wide variety of behaviours can be observed from the point of view of first-order logic. Here is an interesting illustration: on the one hand, all non-abelian free groups are elementarily equivalent (see \cite{Sel06} and \cite{KM06}), while at the other extreme, it can be proved that two co-Hopfian virtually free groups are elementarily equivalent if and only if they are isomorphic. Recall that a group is said to be \emph{co-Hopfian} if every monomorphism from this group into itself is bijective. One example of a co-Hopfian virtually free group is $\mathrm{GL}_2(\mathbb{Z})$. Between these two extremes, the picture is much more varied, and our goal in this paper is to give a description of it.

\smallskip 

In fact, in the class of virtually cyclic groups, we already have a glimpse of the unexpected influence of torsion on the first-order theory, as shown by the following example.

\begin{ex}\label{exemple3}Consider the following two $\mathbb{Z}/25\mathbb{Z}$-by-$\mathbb{Z}$ groups:\[N=\langle a,t \ \vert \ a^{25}=1, \ tat^{-1}=a^6\rangle\ \ \ \text{ and } \ \ \ N'=\langle a',t' \ \vert \ a'^{25}=1, \ t'a't'^{-1}=a'^{11}\rangle.\]These groups are non-isomorphic, but $N\times \mathbb{Z}$ and $N'\times \mathbb{Z}$ are isomorphic. It follows from a theorem of Oger (see \cite{Oge83}) that $N$ and $N'$ are elementarily equivalent.\end{ex}

This example is a particular manifestation of a more general phenomenon that plays an important role in our classification. Here below is an informal version of our main result; see Theorem \ref{principal0} for a precise statement.

\subsection*{Main result (see Theorem \ref{principal0})}\emph{Two finitely generated virtually free groups $G$ and $G'$ are $\forall\exists$-elementarily equivalent if and only if there exist two isomorphic groups $\Gamma \supset G$ and $\Gamma'\supset G'$ obtained respectively from $G$ and $G'$ by performing a finite sequence of specific HNN extensions over finite groups (called \emph{legal large extensions}) or replacements of virtually cyclic subgroups by virtually cyclic overgroups (called \emph{legal small extensions}).}

\medskip

As a corollary of this classification, we give an algorithm that takes as input two finite presentations of virtually free groups, and decides whether these groups have the same $\forall\exists$-theory or not. This algorithm relies on the main algorithm of \cite{DG11}, that takes as input two finite presentations of hyperbolic groups, and decides whether these groups are isomorphic or not.

Moreover, Theorem \ref{principal} gives three other characterizations of $\forall\exists$-elementary equivalence among virtually free groups. In fact, it is worth noting that some of our results are proved in the more general context of hyperbolic groups (see in particular Theorem \ref{legalte} and Theorem \ref{legal2te}), and we except that they will be useful in a future classification of hyperbolic groups (possibly with torsion) up to elementary equivalence. 

In addition, in some cases, we establish results stronger than $\forall\exists$-elementary equivalence, namely the existence of elementary embeddings (or rather $\exists\forall\exists$-elementary embeddings, see Defintion \ref{elememb}). We refer the reader to Theorem \ref{legalteplus}.

Before stating precise results, we need to introduce some definitions. Throughout the paper, all virtually free groups are assumed to be finitely generated, and we shall not repeat this assumption anymore. 

\subsection*{Legal large extensions}

By \cite{KPS73} (see also \cite{SW79} Theorem 7.3), a finitely generated group is virtually free if and only if it splits as a finite graph of finite groups, i.e.\ acts cocompactly by isometries on a simplicial tree with finite vertex stabilizers. Hence, every finitely generated virtually free group can be obtained from finite groups by iterating amalgamated free products and HNN extensions over finite groups. As a consequence, one of the basic questions we have to answer is the following: \emph{how amalgamated free products and HNN extensions over finite groups do affect the $\forall\exists$-theory of a virtually free group, or more generally of a hyperbolic group?}

\medskip

It can easily be seen that the number of conjugacy classes of finite subgroups of a given group is determined by its $\forall\exists$-theory. Thus, if a virtually free group $G$ splits as $G=A\ast_C B$ over a finite group $C$, then $G$ and $A,B$ have distinct $\forall\exists$-theories provided that $A$ or $B$ is not isomorphic to $C\rtimes F_n$, in which case the amalgamated product can be written as a multiple HNN extension. Hence, we can restrict our attention to the case where $G=A\ast_C$, with $C$ finite, which is more subtle: sometimes, the $\forall\exists$-theory is preserved when performing an HNN extension over finite groups, as shown by the following example.

\begin{ex}\label{provide}Let $G$ be a virtually free group (and more generally a hyperbolic group) without non-trivial normal finite subgroup, for instance $F_2$ or $\mathrm{PSL}_2(\mathbb{Z})=\mathbb{Z}/3\mathbb{Z}\ast\mathbb{Z}/2\mathbb{Z}$. Then, by Theorem \ref{legalte} below, we have $\mathrm{Th}_{\forall\exists}(G)=\mathrm{Th}_{\forall\exists}\left(G\ast_{\lbrace 1\rbrace}\right)$.\end{ex}


But sometimes, performing an HNN extension over finite groups modifies the $\forall\exists$-theory of a (non-elementary) virtually free group (and even its universal theory). 

\begin{ex}\label{instable}Let $G=F_2\times\mathbb{Z}/2\mathbb{Z}$. The universal sentence $\forall x\forall y \ (x^2=1)\Rightarrow (xy=yx)$ is satisfied by $G$, but not by $G\ast_{\lbrace 1\rbrace}=G\ast\mathbb{Z}$. \emph{A fortiori}, $G$ and $G\ast_{\lbrace 1\rbrace}$ do not have the same $\forall\exists$-theory. More generally, if $G$ is hyperbolic and if the normalizer $N_G(C)$ of a finite subgroup $C\subset G$ normalizes a finite subgroup $C'$ that contains $C$ strictly, then $G\ast_C$ and $G$ have different $\forall\exists$-theories.\end{ex}

This raises the following problem.

\begin{pb}Given a hyperbolic group $G$, characterize the HNN extensions $G\ast_{\alpha}$ over finite groups such that $\mathrm{Th}_{\forall\exists}(G)=\mathrm{Th}_{\forall\exists}(G\ast_{\alpha})$.
\end{pb}

In order to solve this problem (whose solution is given by Theorem \ref{legalte} below), let us consider an isomorphism $\alpha : C_1\rightarrow C_2$ between two finite subgroups of a hyperbolic group $G$, and suppose that $G$ and the HNN extension $G\ast_{\alpha}=\langle G,t \ \vert \ \alpha(c)=tct^{-1}, \ \forall c\in C\rangle$ have the same $\forall\exists$-theory. Let us derive some easy consequences from this assumption. 

First, note that $G$ must be non-elementary. Indeed, a hyperbolic group is finite if and only if it satisfies the first-order sentence $\forall x \ (x^N=1)$ for some integer $N\geq 1$, and virtually cyclic if and only if it satisfies $\forall x\forall y \ ([x^N,y^N]=1)$ for some integer $N\geq 1$. 

Then, observe that $C_1$ and $C_2$ are necessarily conjugate in $G$, because the number of conjugacy classes of finite subgroups is an invariant of the $\forall\exists$-theory. Therefore, one can assume without loss of generality that $C_1=C_2:=C$. 

In addition, denoting by $\mathrm{Aut}_G(C)$ the subgroup \[\lbrace\sigma\in\mathrm{Aut}(C) \ \vert \ \exists g\in N_G(C), \ \sigma=\mathrm{ad}(g)_{\vert C}\rbrace\] of $\mathrm{Aut}(C)$, where $\mathrm{ad}(g)$ denotes the inner automorphism $x\mapsto gxg^{-1}$ and $N_G(C)$ denotes the normalizer of $C$, it can be observed that we have $\vert \mathrm{Aut}_G(C)\vert =\vert \mathrm{Aut}_{G\ast_{\alpha}}(C)\vert$. We refer the reader to Proposition \ref{facile} for further details. This means that there exists an element $g\in G$ such that $\mathrm{ad}(g)_{\vert C}=\alpha$. 

Before giving two other consequences of the equality $\mathrm{Th}_{\forall\exists}(G)=\mathrm{Th}_{\forall\exists}(G\ast_{\alpha})$, let us recall the following result, proved by Olshanskiy in \cite{Ol93}.

\begin{prop}\label{Olshanskiy}Let $G$ be a non-elementary hyperbolic group, and let $H$ be a non-elementary subgroup of $G$. There exists a unique maximal finite subgroup of $G$ normalized by $H$. This group is denoted by $E_G(H)$.\end{prop}

One can prove (see Proposition \ref{facile}) that the equality $\mathrm{Th}_{\forall\exists}(G)=\mathrm{Th}_{\forall\exists}(G\ast_{\alpha})$ implies that the normalizer $N_G(C)$ of $C$ in $G$ is non-elementary, and that $C$ is the unique maximal finite subgroup of $G$ normalized by $N_G(C)$, i.e.\ that $E_G(N_G(C))=C$. The importance of this last condition is illustrated by Example \ref{instable} above. 

This leads us to the following definition.

\begin{de}[\emph{Legal large extension}]\label{legal}Let $G$ be a non-elementary hyperbolic group, and let $C_1,C_2$ be two finite subgroups of $G$. Suppose that $C_1$ and $C_2$ are isomorphic, and let $\alpha : C_1 \rightarrow C_2$ be an isomorphism. The HNN extension $G\ast_{\alpha}=\langle G,t \ \vert \ \mathrm{ad}(t)_{\vert C_1}=\alpha\rangle$ is said to be \emph{legal} if the following three conditions hold.
	\begin{enumerate}
		\item There exists an element $g\in G$ such that $gC_1g^{-1}=C_2$ and $\mathrm{ad}(g)_{\vert C_1}=\alpha$.
		\item $N_G(C_1)$ is non-elementary.
		\item $E_G(N_G(C_1))=C_1$.
	\end{enumerate}
A group $\Gamma$ is said to be a \emph{legal large extension} of $G$ if it splits as a legal HNN extension $\Gamma=G\ast_{\alpha}$. Sometimes we need to keep track of the order $m$ of the finite group over which the HNN extension is performed, and we say that $\Gamma$ is a \emph{$m$-legal large extension} of $G$
\end{de}

\begin{rque}\label{sansperte}Up to replacing $t$ by $g^{-1}t$ in the presentation above, one can assume without loss of generality that the presentation has the following form: $\langle G,t \ \vert \ \mathrm{ad}(t)_{\vert C_1}=\mathrm{id}_{C_1}\rangle$.\end{rque}

Example \ref{instable} is a typical illustration of a non-legal extension. Indeed, the third condition of the previous definition is clearly violated. By contrast, Example \ref{provide} (that is $\mathrm{PSL}_2(\mathbb{Z})\ast_{\lbrace 1\rbrace}$) is a legal large extension. Here is another example of a legal large extension (to be compared to Example \ref{instable}).

\begin{ex}Let $G=F_2\times\mathbb{Z}/2\mathbb{Z}$. The HNN extension $G\ast_{\mathbb{Z}/2\mathbb{Z}}=G\ast_{\mathbb{Z}/2\mathbb{Z}}(\mathbb{Z}/2\mathbb{Z}\times \mathbb{Z})$ is legal.\end{ex}


If $G$ and $G\ast_{\alpha}$ have the same $\forall\exists$-theories, the previous discussion shows that $G\ast_{\alpha}$ is a legal large extension of $G$ (see Proposition \ref{facile}). One of our main results is that the converse also holds: if $G\ast_{\alpha}$ is a legal large extension of $G$, then we have $\mathrm{Th}_{\forall\exists}(G\ast_{\alpha})=\mathrm{Th}_{\forall\exists}(G)$.

\begin{te}\label{legalte}Let $G$ be a non-elementary hyperbolic group, and let $G\ast_{\alpha}$ be an HNN extension over finite groups. Then, $\mathrm{Th}_{\forall\exists}(G\ast_{\alpha})=\mathrm{Th}_{\forall\exists}(G)$ if and only if $G\ast_{\alpha}$ is a legal large extension of $G$ in the sense of Definition \ref{legal}.\end{te}

The proof of this result relies on a generalization of the key lemma of \cite{Sac73}, using techniques introduced by Sela for torsion-free hyperbolic groups and extended by Reinfeldt and Weidmann to hyperbolic groups with torsion in \cite{RW14}, in particular the \emph{shortening argument}. We also refer the reader to \cite{Hei18} for some results about $\forall\exists$-sentences in hyperbolic groups (possibly with torsion), namely a generalization of Merzlyakov's formal solutions.

In fact, we shall prove the following result, which is stronger than Theorem \ref{legalte}. 

\begin{te}\label{legalteplus}Let $G$ be a non-elementary hyperbolic group, and let $G\ast_{\alpha}$ be an HNN extension over finite groups. Then, the inclusion of $G$ into $G\ast_{\alpha}$ is a $\exists\forall\exists$-elementary embedding (see Definition \ref{elememb}) if and only if $G\ast_{\alpha}$ is a legal large extension of $G$.\end{te}

\subsection*{Legal small extensions}\label{referreader}

Perhaps more surprisingly, another phenomenon of a different nature plays a crucial role in our classification of virtually free groups up to $\forall\exists$-elementary equivalence, as illustrated by Example \ref{exemple3}. This phenomenon is not limited to infinite virtually cyclic groups: more generally, if $G$ is a hyperbolic group, we will prove that one can replace a virtually cyclic subgroup $N\subset G$ by a virtually cyclic overbgroup $N'\supset N$ without modifying the $\forall\exists$-theory of $G$, as soon as certain additional technical conditions are satisfied (in particular, $N$ has to be the normalizer of a finite subgroup of $G$). Before giving a precise statement (Theorem \ref{legal2te} below), we need some definitions.

\begin{de}Given an infinite virtually cyclic group $N$ and an integer $p$, we denote by $D_p(N)$ the definable subset $D_p(N)=\lbrace n^{p} \ \vert \ n\in N\rbrace$.\end{de} 

Let $K_N$ be the maximal order of a finite subgroup of $N$. One can easily prove that for every integer $K\geq K_N$, the set $D_{2K!}(N)$ is a normal subgroup of $N$ (see Lemma \ref{claim}). Note that the quotient group $N/D_{2K!}(N)$ is finite. This finite group is determined by the $\forall\exists$-theory of $N$.

\begin{de}\label{special0}Let $N$ and $N'$ be two infinite virtually cyclic groups. Let $K_N$ and $K_{N'}$ denote the maximal order of a finite subgroup of $N$ and $N'$ respectively, and let $K\geq \mathrm{max}(K_N,K_{N'})$ be an integer. A homomorphism $\varphi : N \rightarrow N'$ is said to be $K$-\emph{nice} if it satisfies the following three properties.
	\begin{itemize}
		\item[$\bullet$]$\varphi$ is injective.
		\item[$\bullet$]If $C_1$ and $C_2$ are two non-conjugate finite subgroups of $N$, then $\varphi(C_1)$ and $\varphi(C_2)$ are non-conjugate in $N'$.
		\item[$\bullet$]The induced homomorphism $\overline{\varphi} : N/D_{2K!}(N) \rightarrow N'/D_{2K!}(N)$ is injective.
	\end{itemize}
\end{de}

\begin{de}[\emph{Legal small extension}]\label{legal2}Let $G$ be a hyperbolic group. Let $K_G$ denote the maximal order of a finite subgroup of $G$. Suppose that $G$ splits as $A\ast_C B$ or $A\ast_C$ over a finite subgroup $C$ whose normalizer $N$ is infinite virtually cyclic and non-elliptic in the splitting. Let $N'$ be a virtually cyclic group such that $K_{N'}\leq K_G$ and let $\iota : N\hookrightarrow N'$ be a $K_G$-nice embedding (in the sense of Definition \ref{special0} above). The amalgamated product \[G'=G\ast_N N'=\langle G, N' \ \vert \ g=\iota(g), \ \forall g\in N\rangle\] is called a \emph{legal small extension} of $G$ if there exists a $K_G$-nice embedding $\iota' : N'\hookrightarrow N$. Sometimes we need to keep track of the cardinality $m$ of the edge group $C$, and we say that $\Gamma$ is a $m$-legal extension of $G$.
\end{de}

For instance, the two virtually cyclic groups of Example \ref{exemple3} are legal small extensions of each other: in this example, $C$ is the cyclic group $\langle a\rangle\simeq \mathbb{Z}/25\mathbb{Z}$, and one can define $\iota : N \hookrightarrow N'$ by $\iota : a \mapsto a', t\mapsto t'^3$ and $\iota' : N' \hookrightarrow N$ by $\iota' : a' \mapsto a, t'\mapsto t^2$.

We will prove the following result.

\begin{te}\label{legal2te}Let $G$ be a hyperbolic group that splits as $A\ast_C B$ or $A\ast_C$ over a finite subgroup $C$ whose normalizer $N$ is infinite virtually cyclic and non-elliptic in the splitting. Let $K_G$ denote the maximal order of a finite subgroup of $G$. Let $N'$ be a virtually cyclic group such that $K_{N'}\leq K_G$, and let $\iota : N\hookrightarrow N'$ be a $K_G$-nice embedding. The amalgamated product \[G'=G\ast_N N'=\langle G, N' \ \vert \ g=\iota(g), \ \forall g\in N\rangle\] is a legal small extension in the sense of Definition \ref{legal2} if and only if $\mathrm{Th}_{\forall\exists}(G')=\mathrm{Th}_{\forall\exists}(G)$.\end{te}

\begin{rque}In general, the group $G$ is not $\exists\forall$-elementarily embedded into $G'$ (note the difference with Theorem \ref{legalteplus}, which generalizes Theorem \ref{legalte}). For instance, in Example \ref{exemple3}, the element $a\in N$ satisfies the following $\exists\forall$-formula $\theta(a)$: \[\theta(a):\exists t \forall u \ (tat^{-1}=a^6)\wedge (t\neq u^3),\] while it can be easily seen that any monomorphism $\iota : N \hookrightarrow N'$ maps $a$ to $a'^p$ for some integer $p$ satisfying $\gcd(p,25)=1$, and that $\theta(a'^p)$ is false is $N'$.\end{rque}

\subsection*{A remark about the terminology}Suppose that $G'$ is a legal large or small extension of a hyperbolic group $G$. Then, in both cases, $G'$ can be written as an amalgamated free product $G'=G\ast_N N'$, where $N$ is the normalizer of a finite subgroup $C$ of $G$, and $N'$ is an overgroup of $N$ in which $C$ is the maximal normal finite subgroup. The terminology "large" or "small" refers to the size of $N$ and $N'$: in the case where the legal extension is large, the groups $N$ and $N'$ are non-elementary, and in the case where the extension is small, $N$ and $N'$ are infinite virtually cyclic. 

\subsection*{Classification of virtually free groups up to $\forall\exists$-elementary equivalence}

Our main result, Theorem \ref{principal0}, asserts that the two kinds of extensions defined above are the only ones we need in order to classify virtually free groups up to $\forall\exists$-equivalence. 

\begin{de}Let $G$ be a hyperbolic group. A group $\Gamma$ is called a \emph{multiple legal extension} of $G$ if there exists a finite sequence of groups $G=G_0\subset G_1\subset \cdots \subset G_n\simeq \Gamma$ where $G_{i+1}$ is a legal (large or small) extension of $G_i$ in the sense of Definitions \ref{legal} or \ref{legal2}, for every integer $0\leq i\leq n-1$. 
\end{de}

Here is our main result (see also Theorem \ref{principal}).

\begin{te}\label{principal0}Two finitely generated virtually free groups $G$ and $G'$ have the same $\forall\exists$-theory if and only if there exist two multiple legal extensions $\Gamma$ and $\Gamma'$ of $G$ and $G'$ respectively, such that $\Gamma\simeq \Gamma'$.\end{te}

\begin{ex}One can deduce from Theorem \ref{principal0} that a virtually free group $G$ has the same $\forall\exists$-theory as $\mathrm{SL}_2(\mathbb{Z})=\mathbb{Z}/6\mathbb{Z}\ast_{\mathbb{Z}/2\mathbb{Z}}\mathbb{Z}/4\mathbb{Z}$ if and only if $G$ splits as \[(\mathbb{Z}/6\mathbb{Z}\ast_{\mathbb{Z}/2\mathbb{Z}}\mathbb{Z}/4\mathbb{Z})\ast_{\mathbb{Z}/2\mathbb{Z}}(\mathbb{Z}/2\mathbb{Z}\times F_n),\] where $F_n$ denotes the free group of rank $n\geq 0$.\end{ex}

Theorem \ref{principal}, which extends Theorem \ref{principal0} above, gives three other characterizations of $\forall\exists$-equivalence among virtually free groups. Before stating this result, we need to generalize the definition of a nice homomorphism (see Definition \ref{special0}).

\begin{de}\label{special1}Let $G$ and $G'$ be two hyperbolic groups. Let $K_G$ (resp.\ $K_{G'}$) denote the maximal order of a finite subgroup of $G$ (resp.\ $G'$). Suppose that $K_G\geq K_{G'}$. A homomorphism $\varphi : G \rightarrow G'$ is said to be \emph{special} if it satisfies the following three properties.
\begin{itemize}
\item[$\bullet$]It is injective on finite subgroups.
\item[$\bullet$]If $C_1$ and $C_2$ are two non-conjugate finite subgroups of $G$, then $\varphi(C_1)$ and $\varphi(C_2)$ are non-conjugate in $G'$.
\item[$\bullet$]If $C$ is a finite subgroup of $G$ whose normalizer is infinite virtually cyclic maximal, then $N_{G'}(\varphi(C))$ is infinite virtually cyclic, and the restriction \[\varphi_{\vert N_G(C)}:N_G(C)\rightarrow N_{G'}(\varphi(C))\] is $K_G$-nice in the sense of Definition \ref{special0} (in particular, $\varphi_{\vert N_G(C)}$ is injective).
\end{itemize}
\end{de}

\begin{rque}Note that if $G$ and $G'$ have the same universal theory, then $K_G=K_{G'}$. 
\end{rque}

\begin{rque}If $G$ and $G'$ are infinite virtually cyclic, then a homomorphism $\varphi : G \rightarrow G'$ is special if and only if it is $K_G$-nice.
\end{rque}



\begin{de}\label{strongmor}Let $G$ and $G'$ be two hyperbolic groups. A special homomorphism $\varphi : G \rightarrow G'$ is said to be \emph{strongly special} if the following holds: for every finite subgroup $C$ of $G$, if the normalizer $N_G(C)$ of $C$ in $G$ is not virtually cyclic, then $N_{G'}(\varphi(C))$ is not virtually cyclic and $\varphi(E_G(N_G(C)))=E_{G'}(N_{G'}(\varphi(C)))$.\end{de}

Recall that a sequence of homomorphisms $(\varphi_n : G \rightarrow G')_{n\in\mathbb{N}}$ is said to be \emph{discriminating} if the following holds: for every $g\in G\setminus \lbrace 1\rbrace$, $\varphi_n(g)$ is non-trivial for every integer $n$ sufficiently large.

We associate to every virtually free group $G$ a sentence $\zeta_G\in\mathrm{Th}_{\exists\forall}(G)$ (see Section \ref{zeta}) such that the following result holds. 

\begin{te}\label{principal}Let $G$ and $G'$ be two finitely generated virtually free groups. The following five assertions are equivalent.
\begin{enumerate}
\item $\mathrm{Th}_{\forall\exists}(G)=\mathrm{Th}_{\forall\exists}(G')$.
\item $G'\models \zeta_G$ and $G\models \zeta_{G'}$.
\item There exist two discriminating sequences $(\varphi_n : G \rightarrow G')_{n\in\mathbb{N}}$ and $(\varphi'_n : G' \rightarrow G)_{n\in\mathbb{N}}$ of special homomorphisms .
\item There exists two strongly special homomorphisms $\varphi : G \rightarrow G' $ and $\varphi' : G' \rightarrow G$.
\item There exist two multiple legal extensions $\Gamma$ and $\Gamma'$ of $G$ and $G'$ respectively, such that $\Gamma\simeq \Gamma'$.
\end{enumerate}
\end{te}

\begin{rque}A classical and easy result claims that two finitely presented groups $G$ and $G'$ have the same existential theory if and only if there exist two discriminating sequences of homomorphisms $(\varphi_n : G \rightarrow G')_{n\in\mathbb{N}}$ and $(\varphi'_n : G' \rightarrow G)_{n\in\mathbb{N}}$ (see for instance \cite{And18}, Proposition 2.1). This should be compared with the third assertion above: from this perspective, as a consequence of Theorem \ref{principal}, the only difference between the existential and $\forall\exists$ theories is that one cannot talk about the conjugacy classes of finite subgroups with only one quantifier, whereas it is possible with two quantifiers.\end{rque}

It seems reasonable to make the following conjecture, which generalizes the famous Tarski's problem about the elementary equivalence of non-abelian free groups (see \cite{Sel06} and \cite{KM06}).

\begin{conj}Two virtually free groups have the same $\forall\exists$-theory if and only if they are elementarily equivalent.\end{conj}

\begin{rque}As a consequence of Sela's work on the first-order theory of hyperbolic groups without torsion (see \cite{Sel09}), the above conjecture is known to be true if one replaces "virtually free" by "hyperbolic without torsion". Moreover, thanks to Sela's result on the first-order theory of free products, the conjecture is known to be true if the two virtually free groups in question are free products of finite groups with a free group.\end{rque}

\newpage

\subsection{Outline of the proof of Theorem \ref{principal}}

We shall prove the following series of implications.

\begin{figure}[h!]
\includegraphics[scale=0.45]{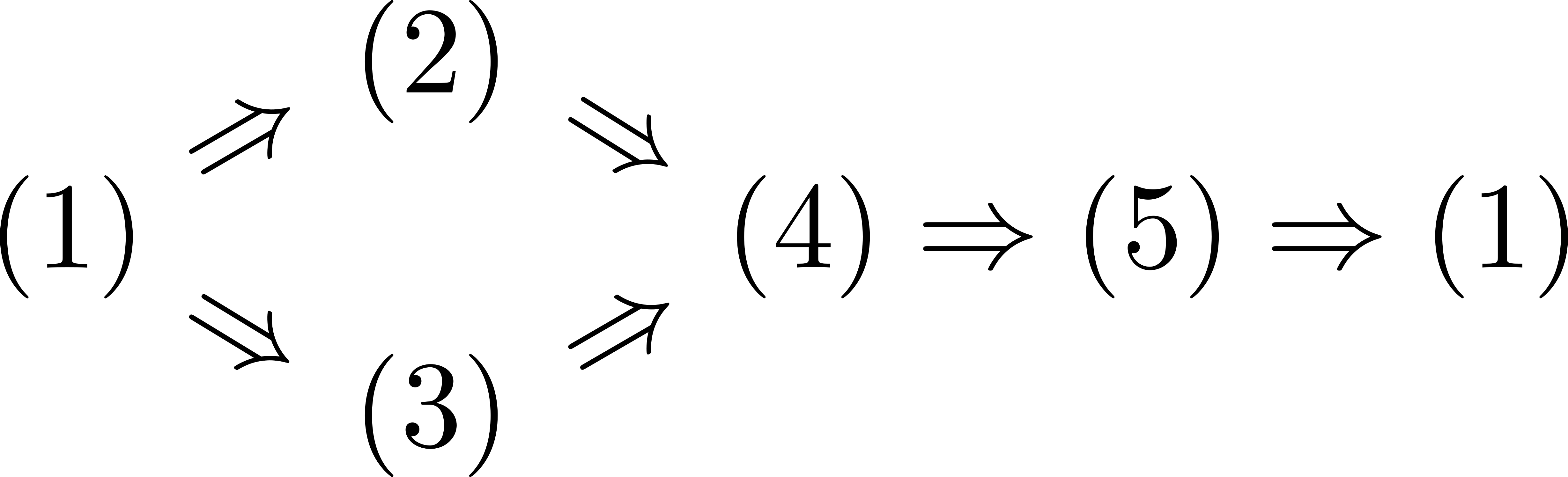}
\end{figure}

Note that $(1)\Rightarrow (2)$ is obvious. Implications $(1)\Rightarrow (3)$, $(2)\Rightarrow (4)$ and $(3)\Rightarrow (4)$ consist mainly in proving that our definitions are expressible by means of $\forall\exists$-sentences or $\exists\forall$-sentences.

The proof of $(5)\Rightarrow (1)$ is a consequence of Theorems \ref{legalte} and \ref{legal2te} (see Section \ref{5 implique 1}). The proof of Theorem \ref{legalte} (as well as of Theorem \ref{legalteplus}) consists in revisiting and generalizing the key lemma of Sacerdote's paper \cite{Sac73} dating from 1973, using some of the tools developed since then by Sela and others (in particular, the theory of group actions on real trees, the shortening argument and test sequences). The proof of Theorem \ref{legal2te} makes also important use of these techniques, but involves more technicalities.

We prove $(4)\Rightarrow (5)$ in three steps: first, we assume that all edge groups in reduced Stallings splittings of $G$ and $G'$ are equal. Then, we deal with the case where all edge groups have the same cardinality, by using a construction called the tree of cylinders, introduced by Guirardel and Levitt. In the general case, different cardinalities of edge groups may coexist in reduced Stallings splittings of $G$ and $G'$. The proof is by induction on the number of edges in these splittings. 


The existence of an algorithm that takes as input two finite presentations of virtually free groups and decides if these groups have the same $\forall\exists$-theory is established by proving that one can bound the number of legal (small or large) extensions involved in the construction of $\Gamma$ and $\Gamma'$ (with the notation of Theorem \ref{principal0}), and that this bound is computable from finite presentations of $G$ and $G'$.

\subsection*{Acknowledgements}

I am very grateful to Vincent Guirardel for his valuable help. I would also like to thank Frédéric Paulin for his careful reading of a previous version of this paper, which led to many improvements.

\hypersetup{colorlinks=true, linkcolor=black}
\tableofcontents

\newpage

\section{Preliminaries}\label{preliminaires}



\subsection{First-order logic}For detailed background, we refer the reader to \cite{Mar02}.

\begin{de}A \emph{first-order formula} in the language of groups is a finite formula using the following symbols: $\forall$, $\exists$, $=$, $\wedge$, $\vee$, $\Rightarrow$, $\neq$, $1$ (standing for the identity element), ${}^{-1}$ (standing for the inverse), $\cdot$ (standing for the group multiplication) and variables $x,y,g,z\ldots$ which are to be interpreted as elements of a group. A variable is \emph{free} if it is not bound by any quantifier $\forall$ or $\exists$. A \emph{sentence} is a formula without free variables.\end{de}

\begin{de}Given a formula $\psi(x_1,\ldots ,x_n)$ with $n\geq 0$ free variables, and $n$ elements $g_1,\ldots ,g_n$ of a group $G$, we say that $\psi(g_1,\ldots ,g_n)$ is \emph{satisfied} by $G$ if its interpretation is true in $G$. This is denoted by $G\models \psi(g_1,\ldots ,g_n)$. For brevity, we use the notation $\psi(\bm{x})$ where $\bm{x}$ denotes a tuple of variables.\end{de}

\begin{de}The \emph{elementary theory} of a group $G$, denoted by $\mathrm{Th}(G)$, is the collection of all sentences that are true in $G$. The \emph{universal-existential theory} of $G$, denoted by $\mathrm{Th}_{\forall\exists}(G)$, is the collection of sentences true in $G$ of the form \[\forall x_1\ldots\forall x_m\exists y_1\ldots\exists y_n\ \psi(x_1,\dots,x_m,y_1,\dots , y_n)\] where $m,n\geq 1$ and $\psi$ is a quantifier-free formula with $m+n$ free variables. In the same way, we define the \emph{universal theory} of $G$, denoted by $\mathrm{Th}_{\forall}(G)$, its \emph{existential theory} $\mathrm{Th}_{\exists}(G)$, etc. We say that two groups $G$ and $G'$ are \emph{elementarily equivalent} (resp.\ \emph{$\forall\exists$-elementarily equivalent}) if $\mathrm{Th}(G)=\mathrm{Th}(G')$ (resp.\ $\mathrm{Th}_{\forall\exists}(G)=\mathrm{Th}_{\forall\exists}(G')$).\end{de}

To keep track of quantifiers and make the first-order formulas more readable, we will often use notations such as $\mathrm{Formula}_1^{\forall}(\bm{x})$ for universal formulas, $\mathrm{Formula}_2^{\exists}(\bm{x})$ for existential formulas, and so on.

\begin{de}\label{elememb}Let $G$ and $\Gamma$ be two groups. An \emph{elementary embedding} of $G$ into $\Gamma$ is a map $i : G \rightarrow \Gamma$ such that, for every first-order formula $\theta(x_1,\ldots,x_n)$ with $n$ free variables and for every tuple $(g_1,\ldots,g_n)\in G^n$, the group $G$ satisfies $\theta(g_1,\ldots,g_n)$ if and only if $\Gamma$ satisfies $\theta(i(g_1),\ldots,i(g_n))$. We define \emph{$\exists\forall$-elementary embeddings} and \emph{$\exists\forall\exists$-elementary embeddings} in the same way, by considering only $\exists\forall$-formulas and $\exists\forall\exists$-formulas respectively, i.e.\ first-order formulas of the form $\exists\boldsymbol{x}\forall\boldsymbol{y} \psi(\boldsymbol{x},\boldsymbol{y},\boldsymbol{t})$ and $\exists\boldsymbol{x}\forall\boldsymbol{y}\exists\boldsymbol{z} \psi(\boldsymbol{x},\boldsymbol{y},\boldsymbol{z},\boldsymbol{t})$ respectively, where $\boldsymbol{x}$, $\boldsymbol{y}$, $\boldsymbol{z}$ and $\boldsymbol{t}$ are tuples of variables, and $\psi$ is a quantifier-free formula in these variables.\end{de}

\subsection{Virtually cyclic groups}

Recall that an infinite virtually cyclic group $G$ can be written as an extension of exactly one of the following two forms:\[1\rightarrow C\rightarrow G\rightarrow \mathbb{Z}\rightarrow 1 \ \ \ \text{or} \ \ \ 1\rightarrow C\rightarrow G\rightarrow D_{\infty}\rightarrow 1,\] where $C$ is finite and $D_{\infty}=\mathbb{Z}/2\mathbb{Z}\ast\mathbb{Z}/2\mathbb{Z}$ denotes the infinite dihedral group. In the first case, $G$ has infinite center and splits as $G=C\rtimes\mathbb{Z}$. We say that $G$ is of \emph{cyclic type}. In the second case, $G$ has finite center, and we say that $G$ is of \emph{dihedral type}. It splits as an amalgamated free product $A\ast_C B$ with $[A:C]=[B:C]=2$. 

We need to describe under which conditions the normalizer of a finite edge group in a splitting is a virtually cyclic group.

\begin{lemme}\label{diédral1}Let $G$ be a group. Suppose that $G$ splits as an amalgamated free product $G=A\ast_CB$ over a finite group $C$, and that $N_G(C)$ is not contained in a conjugate of $A$ or $B$. Then $N_G(C)$ is infinite virtually cyclic if and only if $C$ has index 2 in $N_A(C)$ and in $N_B(C)$. In this case, $N_G(C)$ is of dihedral type, equal to $N_A(C)\ast_C N_B(C)$.\end{lemme}

\begin{lemme}\label{diédral2}Let $G$ be a group. Suppose that $G$ splits as an HNN extension $G=A\ast_C$ over a finite group $C$. Let $C_1$ and $C_2$ denote the two copies of $C$ in $A$ and $t$ be the stable letter associated with the HNN extension. Suppose that $N_G(C)$ is not contained in a conjugate of $A$. Then $N_G(C)$ is infinite virtually cyclic if and only if one of the following two cases holds.
\begin{enumerate}
\item If $C_1$ and $C_2$ are conjugate in $A$ and $N_A(C_1)=C_1$, then the normalizer $N_G(C_1)$ is of cyclic type, equal to $C_1\rtimes \langle at\rangle$, where $a$ denotes an element of $A$ such that $aC_2a^{-1}=C_1$.
\item If $C_1$ and $C_2=t^{-1}C_1t$ are non-conjugate in $G$ and $C_1$ has index 2 in $N_A(C_1)$ and $N_{tAt^{-1}}(C_1)$, then the normalizer $N_G(C_1)$ is of dihedral type, equal to \[N_A(C_1)\ast_{C_1} N_{tAt^{-1}}(C_1).\]
\end{enumerate}\end{lemme}

\subsection{Maximal infinite virtually cyclic subgroups}\label{z}

Let $G$ be a hyperbolic group. If $g\in G$ has infinite order, we denote by $g^{+}$ and $g^{-}$ the attracting and repellings fixed points of $g$ on the boundary $\partial_{\infty} G$ of $G$. The stabilizer of the pair $\lbrace g^{+},g^{-}\rbrace$ is the unique maximal virtually cyclic subgroup of $G$ containing $g$. We denote this subgroup by $M(g)$. If $h$ and $g$ are two elements of infinite order, either $M(h)=M(g)$ or $M(h)\cap M(g)$ is finite; in the latter case, the subgroup $\langle h,g\rangle$ is non-elementary. The following easy lemma shows that $M(g)$ is definable by means of a quantifier-free formula.

\begin{lemme}\label{elementary}Let $g$ be an element of $G$ of infinite order. Let $K$ denote the maximum order of an element of $G$ of finite order. 
\begin{enumerate}
\item For every element $h\in G$, we have
\[h \ \text{belongs to} \ M(g) \ \Leftrightarrow \ [g^{K!},hg^{K!}h^{-1}]=1.\]
\item For every element $h\in G$ of infinite order, we have
\[h \ \text{belongs to} \ M(g) \ \Leftrightarrow \ [g^{K!},h^{K!}]=1.\]
\end{enumerate}
\end{lemme}

\begin{proof}We only prove the first point, the proof of the second point is similar. If $h$ belongs to $M(g)$, then $hgh^{-1}$ belongs to $M(g)$. Therefore, $g^{K!}$ and $(hgh^{-1})^{K!}$ commute, since $M(g)$ has a cyclic subgroup of index $\leq K$. Conversely, if $g^{K!}$ and $hg^{K!}h^{-1}$ commute, $hg^{K!}h^{-1}$ fixes the pair of points $\lbrace g^{+},g^{-}\rbrace$, so $h$ fixes $\lbrace g^{+},g^{-}\rbrace$ as well. Thus, $h$ belongs to $M(g)$.\end{proof}

\begin{co}\label{pour plus tard!!!}Let $g,h$ be two elements of $G$ of infinite order. The subgroup $\langle g,h\rangle$ is elementary if and only if $[g^{K!},h^{K!}]=1$.
\end{co}


Recall that if $H$ is a non-elementary subgroup of $G$, there exists a unique maximal finite subgroup $E_G(H)$ of $G$ normalized by $H$. The following fact is proved by Olshanskiy in \cite{Ol93}.

\begin{prop}[\cite{Ol93} Proposition 1]\label{Olshanskiy2}The finite subgroup $E_G(H)$ admits the following description:\[E_G(H)=\bigcap_{h\in H^0} M(h)\]where $H^0$ denotes the set of elements of $H$ of infinite order.\end{prop}

\subsection{Small cancellation condition}\label{Coulon}

Let $G$ be a hyperbolic group, let $(X,d)$ be a Cayley graph of $G$, and let $\delta$ be its hyperbolicity constant. Let $g$ be an element of $G$ of infinite order. We define the \emph{translation length} of $g$ as \[\vert\vert g\vert\vert=\inf_{x\in X}d(x,gx).\]

The \emph{quasi-axis} of $g$, denoted by $A(g)$, is the union of all geodesics joining $g^{-}$ and $g^{+}$. By Lemma 2.26 in \cite{Cou13} (see also Remark below Definition 2.8, \emph{loc.\ cit.}), the quasi-axis $A(g)$ is 11$\delta$-quasi-convex. If $g'$ is another element of $G$ of infinite order, $\Delta(g,g')$ is defined as follows: \[\Delta(g,g')=\mathrm{diam}\left(A(g)^{+100\delta}\cap A(g')^{+100\delta}\right)\in\mathbb{N}\cup\lbrace\infty\rbrace,\] where $A(g)^{+100\delta}$ is the $100\delta$-neighbourhood of $A(g)$ in $(X,d)$, and $A(g')^{+100\delta}$ is defined similarly. It is well-known that there exists a constant $ N(g) \geq 0 $ such that every element $h \in G $ satisfying $\Delta (g,hgh^{-1}) \geq N(g) $ belongs to $ M (g)$. 


The small-cancellation condition defined below will play a crucial role in the proof of the implication $(5)\Rightarrow (1)$ of Theorem \ref{principal} (for further details, see Section \ref{HNN}).

\begin{de}\label{SCC}Let $G$ be a hyperbolic group. Let $\varepsilon>0$. An element $g$ of infinite order satisfies the $\varepsilon$-small cancellation condition if the following holds: for every $h\in G$, if \[\Delta(g,hgh^{-1}) > \varepsilon \vert \vert g\vert \vert,\] then $h$ and $g$ commute (so $h$ belongs to $M(g)$). In particular, $g$ is central in $M(g)$.\end{de}


\subsection{Actions on real trees}

Recall that a \emph{real tree} is a geodesic metric space in which every triangle is a tripod. A group action by isometries on a real tree is \emph{minimal} if it has no proper invariant subtree. Note that if an action of a finitely generated group on a real tree has no global fixed point, then there is a unique invariant minimal subtree, which is the union of all translation axes. A subtree $T'$ of a real tree $T$ is said to be \emph{non-degenerate} if it contains more than one point.

\subsubsection{Stable and superstable actions}\label{stabletreenash}

General results about group actions on real trees involve hypotheses on infinite sequences of nested arc stabilizers (see for instance Theorem \ref{guilev} and Theorem \ref{reinfeldt} below). An action of a group on a real tree is said to be \emph{stable} in the sense of Bestvina and Feighn (see \cite{BF95b}) if the pointwise stabilizer of any arc eventually stabilizes when this arc gets smaller and smaller. Here is a formal definition.

\begin{de}Let $T$ be a real tree. A non-degenerate subtree $T'$ of $T$ is \emph{stable} if, for every non-degenerate subtree $T''\subset T'$, the pointwise stabilizers of $T''$ and $T'$ coincide. Otherwise, $T'$ is called \emph{unstable}. An action on a real tree is \emph{stable} if any non-degenerate arc contains a non-degenerate stable subarc.\end{de}

In \cite{Gui08}, Guirardel introduced the notion of a \emph{superstable} action on a real tree.

\begin{de}\label{Msuper}An action on a real tree is $M$\emph{-superstable} if every arc whose pointwise stabilizer has order $> M$ is stable.\end{de}

Let $T$ be a real tree, and let $x$ be a point of $T$. A \emph{direction} at $x$ is a connected component of $T\setminus\lbrace x\rbrace$. We say that $x$ is a \emph{branch point} if there are at least three directions at $x$. The following result is a work in preparation by Guirardel and Levitt (improving \cite{Gui01}).


\begin{te}\label{guilev}Let $L$ be a finitely generated group acting on a real tree $T$. Suppose that the action is $M$-superstable, with finitely generated arc stabilizers. Then every point stabilizer is finitely generated, the number of orbit of branch points in $T$ is finite, the number of orbit of directions at branch points in $T$ is finite.\end{te}

\subsubsection{The Bestvina-Paulin method}In the sequel, $\omega$ denotes a non-principal ultrafilter, i.e.\ a finitely additive probability measure $\omega : \mathcal{P}(\mathbb{N})\rightarrow \lbrace 0,1\rbrace$ such that $\omega(F)=0$ whenever $F\subset\mathbb{N}$ is finite.

Let $G$ be a hyperbolic group, and let $(X,d)$ be a Cayley graph of $G$. Let $G'$ be a finitely generated group, equipped with a finite generating set $S$. Let $(\varphi_n : G'\rightarrow G)_{n\in\mathbb{N}}$ be a sequence of homomorphisms. We define the displacement of $\varphi_n$ as \[\lambda=\underset{s\in S}{\mathrm{max}} \ d(1,\varphi_n(s)).\]Suppose that $(\lambda_n)_{n\in\mathbb{N}}\in\mathbb{R}^{\mathbb{N}}$ tends to infinity. Let $d_n$ denote the modified metric $d/\lambda_n$ on $X$. The following result is sometimes called the Bestvina-Paulin method in reference to \cite{Bes88} and \cite{Pau88}.


\begin{te}\label{existencearbre}The ultralimit $(X_{\omega},d_{\omega})$ of the metric spaces $(X,d_n)_{n\in\mathbb{N}}$ is a real tree endowed with an action of $G'$, and there exists a unique minimal $G'$-invariant non-degenerate subtree $T\subset X_{\omega}$. Moreover, some subsequence of the sequence $((X,d_n))_{n\in\mathbb{N}}$ converges to $T$ in the Gromov-Hausdorff topology.\end{te} 



\subsubsection{The Rips machine}

Under certain conditions, group actions on real trees can be analysed using the so-called Rips machine, which enables us to decompose the action into tractable building blocks. We shall use the following version of the Rips machine, proved by Guirardel in \cite{Gui08} (Theorem 5.1). See also \cite{GLP94}, \cite{RS94}, \cite{BF95b}, \cite{Sel97}. 

Given a group $G$ and a family $\mathcal{H}$ of subgroups of $G$, an action of the pair $(G,\mathcal{H})$ on a tree $T$ is an action of $G$ on $T$ such that each $H\in\mathcal{H}$ fixes a point.

\begin{te}\label{guirardel1}Let $G$ be a finitely generated group. Consider a minimal and non-trivial action of $(G,\mathcal{H})$ on an $\mathbb{R}$-tree $T$ by isometries. Assume that
\begin{enumerate}
\item[$(i)$] $T$ satisfies the ascending chain condition: for any decreasing sequence of non-degenerate arcs $I_1\supset I_2\supset\ldots$ whose lengths converge to $0$, the sequence of their pointwise stabilizers $\mathrm{Stab}(I_1)\subset\mathrm{Stab}(I_2)\subset\ldots$ stabilizes.
\item[$(ii)$] For any unstable arc $I\subset T$, 
\begin{enumerate}
\item $\mathrm{Stab}(I)$ is finitely generated,
\item $\forall g\in G$, $g\mathrm{Stab}(I)g^{-1}\subset \mathrm{Stab}(I)\Rightarrow g\mathrm{Stab}(I)g^{-1} = \mathrm{Stab}(I)$.
\end{enumerate}
\end{enumerate}
Then either $(G,\mathcal{H})$ splits over the pointwise stabilizer of an unstable arc, or over the pointwise stabilizer of a non-degenerate tripod whose normalizer contains $F_2$, or $T$ has a decomposition into a graph of actions where each vertex action $G_v\curvearrowright Y_v$ is either
\begin{enumerate}
\item simplicial: $G_v\curvearrowright Y_v$ is a simplicial action on a simplicial tree;
\item of Seifert type: the vertex action $G_v\curvearrowright Y_v$ has kernel $N_v$, and the
faithful action $G_v/N_v\curvearrowright Y_v$ is dual to an arational measured foliation on a compact conical 2-orbifold with boundary;
\item axial: $Y_v$ is a line, and the image of $G_v$ in $\mathrm{Isom}(Y_v)$ is a finitely generated group acting with dense orbits on $Y_v$.
\end{enumerate}
\end{te}

\subsubsection{Transverse covering}

We will use the following definitions (see \cite{Gui04}, Definitions 4.6 and 4.8).

\begin{de}\label{transverse}Let $T$ be a real tree endowed with an action of a group $G$, and let $(Y_j)_{i\in J}$ be a $G$-invariant family of non-degenerate closed subtrees of $T$. We say that $(Y_j)_{j\in J}$ is a transverse covering of $T$ if the following two conditions hold.
\begin{itemize}
\item[$\bullet$]\emph{Transverse intersection:} if $Y_i\cap Y_j$ contains more than one point, then $Y_i=Y_j$.
\item[$\bullet$]\emph{Finiteness condition:} every arc of $T$ is covered by finitely many $Y_j$.
\end{itemize}
\end{de}

\begin{de}\label{squelette}Let $T$ be a real tree, and let $(Y_j)_{j\in J}$ be a transverse covering of $T$. The \emph{skeleton} of this transverse covering is the bipartite simplicial tree $S$ defined as follows:
\begin{enumerate}
\item $V(S)=V_0(S)\sqcup V_1(S)$ where $V_1(S)=\lbrace Y_j \ \vert \ j\in J\rbrace$ and $V_0(S)$ is the set of points $x\in T$ that belong to at least two distinct subtrees $Y_i$ and $Y_j$. The stabilizer of a vertex of $S$ is the global stabilizer of the corresponding subtree of $T$.
\item There is an edge $\varepsilon=(Y_j,x)$ between $Y_j\in V_1(S)$ and $x\in V_0(S)$ if and only if $x$, viewed as a point of $T$, belongs to $Y_j$, viewed as a subtree of $T$. The stabilizer of $\varepsilon$ is $G_{Y_i}\cap G_x$.
\end{enumerate}
Moreover, the action of $G$ on $S$ is minimal provided that the action of $G$ on $T$ is minimal (see \cite{Gui04} Lemma 4.9).
\end{de}


\subsection{$G$-limit groups and the shortening argument}Let $G$ be a hyperbolic group, and let $G'$ be a finitely generated group. A sequence of homomorphisms $(\varphi_n:G'\rightarrow G)_{n\in\mathbb{N}}$ is termed stable if, for every element $x\in G'$, either $\varphi_n(x)$ is trivial for every $n$ large enough, or $\varphi_n(x)$ is non-trivial for every $n$ large enough. The stable kernel of the sequence is defined as follows:\[\ker((\varphi_n)_{n\in\mathbb{N}})=\lbrace x\in G' \ \vert \ \varphi_n(x)=1 \ \text{for every $n$ large enough}\rbrace.\]The quotient $L=G'/\ker((\varphi_n)_{n\in\mathbb{N}})$ is called the $G$-limit group associated with the sequence $(\varphi_n)$. This group acts on the tree $T$ given by Theorem \ref{existencearbre}. The class of $G$-limit groups admits several equivalent descriptions.

\begin{te}Let $G$ be a hyperbolic group, and let $L$ be a finitely generated group. The following three assertions are equivalent.
\begin{itemize}
\item[$\bullet$]$L$ is a $G$-limit group.
\item[$\bullet$]$L$ is fully residually $G$, meaning that there exists a sequence of homomorphisms $(\varphi_n)\in\mathrm{Hom}(L,G)^{\mathbb{N}}$ such that, for every non-trivial element $x\in L$, $\varphi_n(x)$ is non-trivial for every $n$ large enough. Such a sequence is said to be discriminating.
\item[$\bullet$]$\mathrm{Th}_{\exists}(L)\subset\mathrm{Th}_{\exists}(G)$.
\end{itemize}
\end{te}

The third point justifies why $G$-limit groups play a crucial role in Sela's resolution of the Tarski's problem \cite{Sel06}, and more generally in his classification of torsion-free hyperbolic groups up to elementary equivalence \cite{Sel09}. 

Given $\omega$ and $(\varphi_n)_{n\in\mathbb{N}}$, the action of the limit group on the real tree given by Theorem \ref{existencearbre} has nice properties. The following result generalizes a theorem proved by Sela for torsion-free hyperbolic groups.

\begin{te}[\cite{RW14}, Theorem 1.16]\label{reinfeldt}Let $G$ be a hyperbolic group, and let $L$ be a $G$-limit group. Let $T$ be the corresponding real tree. The following hold, for the action of $L$ on $T$:
\begin{itemize}
\item[$\bullet$]the pointwise stabilizer of any non-degenerate tripod is finite;
\item[$\bullet$]the pointwise stabilizer of any non-degenerate arc is finitely generated and finite-by-abelian;
\item[$\bullet$]the pointwise stabilizer of any unstable arc is finite.
\end{itemize}
\end{te}

\begin{rque}\label{rquereinfeldt}Note that the action of $L$ on $T$ is $M$-superstable, where $M$ denotes the maximal order of a finite subgroup of $L$ (which is bounded from above by the maximal order of a finite subgroup of $G$). In particular, Theorem \ref{guilev} is applicable, so the number of orbit of branch points in $T$ for the action of $L$ is finite. This fact will be useful later. Note also that the tree $T$ satisfies the ascending chain condition of Theorem \ref{guirardel1} since any ascending sequence of finite-by-abelian subgroups of a hyperbolic group stabilizes.\end{rque}

In order to study the set $\mathrm{Hom}(L,G)$, Sela introduced a technique called the shortening argument, later generalized to hyperbolic groups possibly with torsion by Reinfeldt and Weidmann in \cite{RW14}. In this paper, we need a relative version of this argument.

\begin{de}\label{modulnash}Let $G$ be a hyperbolic group and let $H$ be a finitely generated subgroup of $G$. Suppose that $G$ is one-ended relative to $H$. We denote by $\mathrm{Aut}_{H}(G)$ the subgroup of $\mathrm{Aut}(U)$ consisting of all automorphisms $\sigma$ such that the following two conditions hold:
\begin{enumerate}
\item $\sigma_{\vert H}=\mathrm{id}_{\vert H}$;
\item for every finite subgroup $F$ of $G$, there exists an element $g\in G$ such that $\sigma_{\vert F}=\mathrm{ad}(g)_{\vert F}$.
\end{enumerate}
\end{de}

\begin{de}Let $G$ and $\Gamma$ be two hyperbolic groups, and let $H$ be a finitely generated subgroup of $G$. Suppose that $G$ is one-ended relative to $H$ and that there exists a monomorphism $i : H \hookrightarrow \Gamma$. Let $S$ be a finite generating set of $G$. A homomorphism $\varphi : G \rightarrow \Gamma$ such that $\varphi_{\vert H}=\mathrm{ad}(\gamma)\circ i$ for some $\gamma\in \Gamma$ is said to be \emph{short} if its length $\ell(\varphi):=\max_{s\in S}d(1,\varphi(s))$ is minimal among the lengths of homomorphisms in the orbit of $\varphi$ under the action of $\mathrm{Aut}_{H}(G)\times \mathrm{Inn}(\Gamma)$.\end{de}

\begin{prop}\label{nashpaulinnash}We keep the same notations as in the previous definition. Let $(\phi_n:G\rightarrow\Gamma )_{n\in\mathbb{N}}$ be a stable sequence of distinct short homomorphisms. Then the stable kernel of the sequence is non-trivial.\end{prop}

For a proof of this result, we refer the reader to \cite{And18b} (based on \cite{Per08} and \cite{RW14}).

Here below are two very useful consequences of the shortening argument whose proofs can be found in \cite{RW14} (see \cite{Sel09} for the torsion-free case).

\begin{te}[Descending chain condition for $G$-limit groups]\label{chainesela}Let $G$ be a hyperbolic group. Let $(L_n)_{n\in\mathbb{N}}$ be a sequence of $G$-limit groups. If $(\varphi_n : L_n \rightarrow L_{n+1})_{n\in\mathbb{N}}$ is a sequence of epimorphisms, then $\varphi_n$ is an isomorphism for all $n$ sufficiently large.
\end{te}

\begin{de}A group $L$ is said to be \emph{equationally noetherian} if, for any system of equations in finitely many variables $\Sigma$, there exists a finite subsystem $\Sigma_0$ of $\Sigma$ such that $\mathrm{Sol}(\Sigma)=\mathrm{Sol}(\Sigma_0)$ in $L$.
\end{de}

\begin{te}\label{eqnoethhyp333}Let $G$ be a hyperbolic group, and let $L$ be a $G$-limit group. Then $L$ is equationally noetherian.\end{te}



\subsection{Tree of cylinders}\label{tree}

Let $k\geq 1$ be an integer, let $G$ be a finitely generated group, and let $T$ be a splitting of $G$ over finite groups of order exactly $k$. Recall that the deformation space of $T$ is the set of $G$-trees which can be obtained from $T$ by some collapse and expansion moves, or equivalently, which have the same elliptic subgroups as $T$. In \cite{GL11}, Guirardel and Levitt construct a tree that only depends on the deformation space of $T$. This tree is called the tree of cylinders of $T$, denoted by $T_c$. This tree will play a crucial role in our proof of the implication $(4)\Rightarrow (5)$ of Theorem \ref{principal} (see Section \ref{4 implique 5}). We summarize below the construction of the tree of cylinders $T_c$.

First, we define an equivalence relation $\sim$ on the set of edges of $T$. We declare two edges $e$ and $e'$ to be equivalent if $G_e=G_{e'}$. Since all edge stabilizers have the same order, the union of all edges having a given stabilizer $C$ is a subtree $Y_C$, called a cylinder of $T$. In other words, $Y_C$ is the subset of $T$ pointwise fixed by $C$. Two distinct cylinders meet in at most one point. The \emph{tree of cylinders} $T_c$ of $T$ is the bipartite tree with set of vertices $V_0(T_c)\sqcup V_1(T_c)$ such that $V_0(T_c)$ is the set of vertices $x$ of $T$ which belong to at least two cylinders, $V_1(T_c)$ is the set of cylinders of $T$, and there is an edge $\varepsilon=(x, Y_C)$ between $x\in V_0(T_c)$ and $Y_C\in V_1(T_c)$ in $T_c$ if and only if $x\in Y_C$. In other words, one obtains $T_c$ from $T$ by replacing each cylinder $Y_C$ by the cone on its boundary (defined as the set of
vertices of $Y_C$ belonging to some other cylinder). If $Y_C$ belongs to $V_1(T_c)$, the vertex group $G_{Y_C}$ is the global stabilizer of $Y_C$ in $T$, i.e.\ the normalizer of $C$ in $G$. 

\subsection{JSJ splittings over finite groups}\label{SD}

A splitting $T$ of a group $G$ is said to be \emph{reduced} if there is no edge of $T$ of the form $e=[v,w]$ with $G_v=G_e$.


Let $m\geq 1$ be an integer. A finitely generated group is termed $(\leq m)$\emph{-rigid} if it does not split non-trivially over a finite subgroup of order $\leq m$.

Let $G$ be a finitely generated virtually free group. Let $\mathcal{F}$ be the set of finite subgroups of $G$, and let $\mathcal{F}_m$ be the set of finite subgroups of $G$ of order $\leq m$. A splitting of $G$ over $\mathcal{F}$ all of whose vertex stabilizers are finite is called a \emph{Stallings splitting} of $G$, and a splitting of $G$ over $\mathcal{F}_m$ all of whose vertex stabilizers are $m$-rigid is called a $m$\emph{-JSJ splitting} of $G$. A $m$-JSJ splitting is not unique, but the conjugacy classes of vertex and edge groups do not depend on the choice of a reduced $m$-JSJ splitting. A vertex group of a reduced $m$-JSJ splitting is called a $m$\emph{-factor}.

Note that a Stallings splitting is a $m$-JSJ splitting whenever $m\geq M$, where $M$ denotes the maximal order of a finite subgroup of $G$. Note also that one gets a $m$-JSJ splitting from any $m'$-splitting of $G$, with $m'>m$, by collapsing all edges whose stabilizer has order $>m$.


\subsection{An extension lemma}

\begin{lemme}\label{s'étend}Let $G$ and $G'$ be two groups. Let $\varphi : G \rightarrow G'$ be a homomorphism. Consider a splitting of $G$ as a graph of groups $\Lambda$. For every vertex $v$ of $\Lambda$,  let $\widehat{G}_v$ be an overgroup of $G_v$, and let $\widehat{\varphi}_v : \widehat{G}_v\rightarrow G'$ be a homomorphism. Let $\widehat{G}$ denote the group obtained from $G$ by replacing each $G_v$ by $\widehat{G}_v$ in $\Lambda$. For every vertex $v\in \Lambda$ and every edge $e$ incident to $v$, suppose that there exists an element $g'_{e,v}\in G'$ such that \[({\widehat{\varphi}_v})_{\vert {G_e}}=\mathrm{ad}(g'_{e,v})\circ \varphi_{\vert G_e}.\]Then, there exists a homomorphism $\widehat{\varphi} : \widehat{G}\rightarrow G'$ such that, for every vertex $v$ of $\Lambda$, \[{\widehat{\varphi}}_{\vert \widehat{G}_v}=\mathrm{ad}(g'_v)\circ \widehat{\varphi}_v \text{ for some } g'_v\in G'.\]\end{lemme}

\begin{proof}We proceed by induction on the number of edges of the graph of groups $\Lambda$. It is enough to prove the lemma in the case where $\Lambda$ has only one edge.

\vspace{1mm}

\emph{First case.} Suppose that $G=G_v\ast_C G_w$. By hypothesis, there exist two elements $g'_1,g'_2\in G'$ such that \[({\widehat{\varphi}_v})_{\vert C}=\mathrm{ad}(g'_1)\circ {\varphi}_{\vert C} \ \ \text{   and   }\ \  ({\widehat{\varphi}_w})_{\vert C}=\mathrm{ad}(g'_2)\circ {\varphi}_{\vert C}.\] One can define $\widehat{\varphi}: \widehat{G}\rightarrow G'$ by \[{\widehat{\varphi}}_{\vert \widehat{G}_v}=\mathrm{ad}({g'_1}^{-1})\circ{\widehat{\varphi}_v} \ \ \text{   and   } \ \ {\widehat{\varphi}}_{\vert \widehat{G}_w}=\mathrm{ad}({g'_2}^{-1})\circ{\widehat{\varphi}_w}.\]

\emph{Second case.} Suppose that \[G=G_v\ast_{C}=\langle G_v,t\ \vert \ tct^{-1}=\alpha(c), \ \forall c\in C\rangle.\]By hypothesis, there exist two elements $g'_1,g'_2\in G'$ such that \[({\widehat{\varphi}_v})_{\vert C}=\mathrm{ad}(g'_1)\circ {\varphi}_{\vert C} \ \ \text{   and   } \ \ ({\widehat{\varphi}_v})_{\vert tCt^{-1}}=\mathrm{ad}(g'_2)\circ {\varphi}_{\vert tCt^{-1}}.\] One can define $\widehat{\varphi}: \widehat{G}\rightarrow G'$ by \[{\widehat{\varphi}}_{\vert \widehat{G}_v}={\widehat{\varphi}_v} \ \ \text{   and   } \ \ {\widehat{\varphi}}(t)=g'_2\varphi(t){g'_1}^{-1}.\]
\end{proof}

\section{Proof of $(1)\Rightarrow (3)$}



\begin{de}For any hyperbolic group $G$, we denote by $K_G$ the maximal order of a finite subgroup of $G$.\end{de}

\begin{prop}\label{defspefor}Let $G=\langle s_1,\ldots ,s_p \rangle$ be a hyperbolic group. There exists a universal formula \[\mathrm{Special}^{\forall}(x_1,\ldots,x_p)\] such that, for every hyperbolic group $G'$ satisfying $K_{G'}\leq K_G$, and for every tuple $\bm{g'}=(g'_1,\ldots,g'_p)$ of elements of $G'$, the following two assertions are equivalent:
\begin{enumerate}
\item the group $G'$ satisfies $\mathrm{Special}^{\forall}(\bm{g'})$;
\item the map $\varphi_{\bm{g'}}:\lbrace s_1,\ldots,s_p\rbrace\rightarrow G'$ defined by $s_i\mapsto g'_i$ for $1\leq i\leq p$ extends to a special (see Definition \ref{special1}) homomorphism from $G$ to $G'$.
\end{enumerate}
\end{prop}


\begin{proof}Let $G=\langle s_1,\ldots ,s_p \ \vert \ \Sigma(s_1,\ldots,s_p)=1\rangle$ be a finite presentation of $G$, where $\Sigma(s_1,\ldots,s_p)=1$ denotes a finite system of equations in $p$ variables. Let $\lbrace C_1,\ldots,C_q\rbrace$ be a set of representatives of the conjugacy classes of finite subgroups of $G$, and let $I$ be the subset of $\llbracket 1,q\rrbracket$ such that $N_G(C_i)$ is infinite virtually cyclic maximal if and only if $i\in I$.

Observe that there is a one-to-one correspondence between the set of homomorphisms $\mathrm{Hom}(G,G')$ and the set $\mathrm{Sol}_{G'}(\Sigma)=\lbrace\boldsymbol{g}'\in G'^p \ \vert \ \Sigma(\boldsymbol{g}')=1\rbrace$. 
	
For $1\leq i\leq q$, the injectivity of the homomorphism $\varphi_{\boldsymbol{g}'}:\bm{s}\rightarrow\bm{g}'$ on $C_i$ can be expressed by a quantifier-free formula $\mathrm{Inj}_i^1(\boldsymbol{g}')$.
	
For $i\in I$, the injectivity of $\varphi_{\boldsymbol{g}'}$ on $N_i$ can be expressed by a quantifier-free formula $\mathrm{Inj}_i^2(\boldsymbol{g}')$. Indeed, it is enough to say that $\varphi_{\boldsymbol{g}'}$ is injective on finite subgroups and that $\varphi_{\boldsymbol{g}'}(N_i)$ has infinite order. The latter point is easily expressible, since an element $g'$ of $G'$ has infinite order if and only if $g'^{K_G!}\neq 1$, because $K_G\geq K_{G'}$.

If $1\leq i\neq j\leq q$, the assertion that $\varphi_{\boldsymbol{g}'}(C_i)$ and $\varphi_{\boldsymbol{g}'}(C_j)$ are non-conjugate translates into a universal formula $\mathrm{NonConj}_{i,j}^{\forall}(\boldsymbol{g}')$. 

If $i$ belongs to $I$, since the quotient $N_{i}/D_{2K_G!}(N_{i})$ is finite, one can choose a finite collection of representatives $g_1,\ldots, g_{r}\in N_i$ of the cosets of $D_{2K_G!}(N_{i})$. Each element $g_k$ can be written as a word $g_k(\boldsymbol{s})$, and the injectivity of the induced homomorphism \[\overline{\varphi_{\boldsymbol{g}'}}:N_{i}/D_{2K_G!}(N_{i})\rightarrow N_{G'}(\varphi_{\boldsymbol{g}'}(C_i))/D_{2K_G!}(N_{G'}(\varphi_{\boldsymbol{g}'}(C_i))),\]well-defined since $K_{G'}\leq K_G$, translates into a universal formula $\mathrm{InjD}_{i}^{\forall}(\boldsymbol{g}')$ expressing the fact that $g_k^{-1}(\boldsymbol{g}')g_{\ell}(\boldsymbol{g}')$ does not belong to $D_{2K_G!}(N_{G'}(\varphi_{\boldsymbol{g}'}(C_{i})))$ if $k\neq \ell$. 

Now, let us define the formula $\mathrm{Special}^{\forall}(\boldsymbol{x})$ by \[ \mathrm{Special}^{\forall}(\boldsymbol{x}):(\Sigma(\boldsymbol{x})=1)\wedge \bigwedge_{i=1}^q\bigwedge_{j\neq i} \mathrm{NonConj}_{i,j}^{\forall}(\boldsymbol{x})\wedge\bigwedge_{i\in I} (\mathrm{Inj}_i^1(\boldsymbol{g}') \wedge \mathrm{Inj}_i^2(\boldsymbol{g}')\wedge\mathrm{InjD}_i^{\forall}(\boldsymbol{x})).\]For any $\boldsymbol{g}'\in G'^p$, the group $G'$ satisfies $\mathrm{Special}^{\forall}(\bm{g}')$ if and only if the homomorphism $\varphi_{\boldsymbol{g}'} : G \rightarrow G': \boldsymbol{s}\mapsto \boldsymbol{g}'$ is special.\end{proof}

The following result is a slight variation of the previous proposition.

\begin{prop}\label{plusdideedutout}Let $G=\langle s_1,\ldots ,s_p \rangle$ be a hyperbolic group. For every integer $n\geq 1$, there exists a universal formula \[\mathrm{Special}^{\forall}_n(x_1,\ldots,x_p)\] such that, for every hyperbolic group $G'$ satisfying $K_{G'}\leq K_G$, and for every tuple $\bm{g'}=(g'_1,\ldots,g'_p)$ of elements of $G'$, the following two assertions are equivalent:
\begin{enumerate}
\item the group $G'$ satisfies $\mathrm{Special}^{\forall}_n(\bm{g'})$;
\item the map $\varphi_{\bm{g'}}:\lbrace s_1,\ldots,s_p\rbrace\rightarrow G'$ defined by $s_i\mapsto g'_i$ for $1\leq i\leq p$ extends to a special (see Definition \ref{special1}) homomorphism from $G$ to $G'$ injective on the ball of radius $n$ in $G$, with respect to $\lbrace s_1,\ldots,s_p\rbrace$.
\end{enumerate}
\end{prop}

\begin{proof}We keep the same notations as in the proof of the previous proposition.

The injectivity of the homomorphism $\varphi_{\boldsymbol{g}'} : \boldsymbol{s}=(s_1,\ldots,s_p)\mapsto \boldsymbol{g}'\in \mathrm{Sol}_{G'}(\Sigma)$ on the ball of radius $n$ centered at $1$ in $G$ for the generating set $\lbrace s_1,\ldots ,s_p\rbrace$ is expressible by means of a finite system of inequations in $p$ variables $B_n(\boldsymbol{g'})\neq 1$.

The universal sentence $\mathrm{Special}_n^{\forall}(\boldsymbol{x})$ defined by \[ \mathrm{Special}_n^{\forall}(\boldsymbol{x}):\mathrm{Special}^{\forall}(\boldsymbol{x})\wedge (B_n(\boldsymbol{x})\neq 1)\]has a witness $\boldsymbol{g}'\in G'^p$ if and only if the homomorphism $\varphi_{\boldsymbol{g}'} : G \rightarrow G': \boldsymbol{s}\mapsto \boldsymbol{g}'$ is special and injective on the ball of radius $n$ in $G$.\end{proof}

\begin{co}\label{1904}Let $G$ and $G'$ be two hyperbolic groups. If $\mathrm{Th}_{\exists\forall}(G)\subset\mathrm{Th}_{\exists\forall}(G')$, then there exists a discriminating sequence of special homomorphisms $(\varphi_n : G \rightarrow G')_{n\in\mathbb{N}}$.\end{co}

\begin{proof}Note that for every integer $n\geq 1$, the group $G$ satisfies the $\exists\forall$-sentence \[\zeta_n:\exists \bm{x} \ \mathrm{Special}_n^{\forall}(\boldsymbol{x}),\] since the identity of $G$ is a special homomorphism. Then, $\mathrm{Th}_{\exists\forall}(G)$ being contained in $\mathrm{Th}_{\exists\forall}(G')$, the group $G'$ satisfies the sentence $\zeta_n$ as well. Since $G$ and $G'$ have the same existential theory, we have $K_G=K_{G'}$, hence Proposition \ref{plusdideedutout} above applies and tells us that there exists a special homomorphism $\varphi_n : G \rightarrow G'$ injective on $B_n$.\end{proof}

The following corollary is immediate.

\begin{co}[Implication $(1)\Rightarrow (3)$ of Theorem \ref{principal}]Let $G$ and $G'$ be two hyperbolic groups. If $\mathrm{Th}_{\exists\forall}(G)=\mathrm{Th}_{\exists\forall}(G')$, then there exists two discriminating sequences of special homomorphisms $(\varphi_n : G \rightarrow G')_{n\in\mathbb{N}}$ and $(\varphi'_n : G' \rightarrow G)_{n\in\mathbb{N}}$.
\end{co}

\section{Definition of $\zeta_G$ and proofs of $(2)\Leftrightarrow (4)$ and $(3)\Rightarrow (4)$}\label{zeta}

\subsection{Preliminary results}

\begin{de}\label{chain}Let $G$ be a hyperbolic group. A \emph{$G$-chain} is a tuple $(g_1,\ldots,g_c)$ of elements of $G$ of infinite order such that the inclusions \[M(g_1)\supset (M(g_1)\cap M(g_2))\supset \cdots \supset (M(g_1)\cap \cdots \cap M(g_c))\] are all strict.\end{de}

\begin{de}\label{complexity}Let $G$ be a hyperbolic group, and let $H$ be a non-elementary subgroup of $G$. The \emph{complexity} $c(H)$ of $H$ is the maximal size of a $G$-chain of elements of $H$.\end{de}

\begin{rque}If $(g_1,g_2)$ is a $G$-chain, then $M(g_1)\cap M(g_2)$ is finite. It follows that $c(H)<\infty$.\end{rque}

The following lemma is an immediate consequence of the fact that $E_G(H)=\bigcap_{h\in H^0} M(h)$, where $H^0$ denotes the set of elements of $H$ of infinite order.

\begin{lemme}\label{immlemm0307}If $(h_1,\ldots ,h_{c(H)})\in H^{c(H)}$ is a $G$-chain of length $c(H)$, then \[E_G(H)=\bigcap_{i=1}^{c(H)} M(h_i).\]
\end{lemme}

\begin{lemme}\label{chain2}Let $N\geq 1$ and $K\geq 1$ be two integers. There exists an existential formula $\mathrm{Chain}_N^{\exists}(\boldsymbol{x})$ with $N$ free variables such that, for any hyperbolic group $G$ all of whose finite subgroups have order $\leq K$, a tuple $\boldsymbol{g}\in G^N$ is a $G$-chain if and only if $G\models\mathrm{Chain}_N^{\exists}(\boldsymbol{g})$.\end{lemme}

\begin{proof}Let $\boldsymbol{g}\in G^N$. The fact that every $g_k$ has infinite order translates into $g_k^{K!}\neq 1$. Recall that $M(g_k)=\lbrace g\in G \ \vert \ [g_k^{K!},gg_k^{K!}g^{-1}]=1\rbrace$. The $N$-tuple $\boldsymbol{g}$ is a chain if and only if, for every $1\leq k<N$, there exists an element $x_k$ in $\bigcap_{i=1}^kM(g_i)$ that does not belong to $M(g_{k+1})$. This condition is clearly expressible by means of an existential sentence $\mathrm{Chain}_N^{\exists}(\boldsymbol{g})$.\end{proof}

Strongly special homomorphisms are not definable by means of a $\forall\exists$-sentence. For that reason, we introduce below a weaker definition.

\begin{de}\label{defpre}Let $G$ and $G'$ be hyperbolic groups. A special homomorphism $\varphi : G \rightarrow G'$ is said to be \emph{pre-strongly special} if, for every finite subgroup $C$ of $G$, the following conditions hold.
\begin{enumerate}
\item If the normalizer $N_G(C)$ is virtually cyclic, then $\varphi$ is injective in restriction to $N_G(C)$.
\item If the normalizer $N_G(C)$ is not virtually cyclic, then
\begin{enumerate}
\item the normalizer $N_{G'}(\varphi(C))$ is not virtually cyclic, and 
\item there exists a $G$-chain $(h_1,\ldots,h_c)$, with $h_i\in N_G(C)$ and $c:=c(N_G(C))$ (see Definition \ref{complexity}), such that $(\varphi(h_1),\ldots,\varphi(h_c))$ is a $G'$-chain.
\end{enumerate}
\end{enumerate}
\end{de}


In the case where two pre-strongly special homomorphisms $\varphi : G \rightarrow G'$ and $\varphi' : G' \rightarrow G$ exist simultaneously, then these two homomorphisms are in fact strongly special, as shown by the following lemma.

\begin{lemme}\label{lemmeAE}Let $G$ and $G'$ be hyperbolic groups. Let $\varphi : G \rightarrow G'$ and $\varphi' : G' \rightarrow G$ be pre-strongly special homomorphisms. Then $\varphi$ and $\varphi'$ are strongly special.\end{lemme}

\begin{proof}Let $C$ be a finite subgroup of $G$ such that $N_G(C)$ is non-elementary. By definition of a pre-strongly special homomorphism, there exists a $G$-chain $(h_1,\ldots,h_c)$, with $h_i\in N_G(C)$ and $c:=c(N_G(C))$, such that $(\varphi(h_1),\ldots,\varphi(h_c))$ is a $G'$-chain of elements of $N_{G'}(\varphi(C))$. As a consequence, the maximal size of a $G'$-chain of elements of $N_{G'}(\varphi(C))$ is $\geq c$, so \[c=c(N_G(C))\leq c(N_{G'}(\varphi(C))).\] Similarly, for every finite subgroup $C'$ of $G'$, the following inequality holds: \[c(N_{G'}(C'))\leq c(N_{G}(\varphi'(C'))).\] Now, note that the homomorphisms $\varphi$ and $\varphi'$ induce two bijections between the conjugacy classes of finite subgroups of $G$ and $G'$, because $\varphi$ and $\varphi'$ are special. It follows that the previous inequalities are in fact equalities. In particular, $(\varphi(h_1),\ldots,\varphi(h_c))$ is a maximal $G'$-chain of elements of $N_{G'}(\varphi(C))$. So we have \[E:=E_G(N_G(C))=\bigcap_{i=1}^c M(h_i) \ \ \text{ and } \ \ E':=E_{G'}(N_{G'}(\varphi(C)))=\bigcap_{i=1}^c M(\varphi(h_i)).\]It follows that \[\varphi(E)\subset \bigcap_{i=1}^c \varphi(M(h_i))\subset \bigcap_{i=1}^c M(\varphi(h_i))=E'.\]
By symmetry, these inclusions are in fact equalities.\end{proof}

\subsection{Definition of $\zeta_G$}

Recall that the universal formula $\mathrm{Special}^{\forall}(\bm{x})$ is defined in Proposition \ref{defspefor}.

\begin{prop}\label{mntab}Let $G=\langle s_1,\ldots ,s_p \rangle$ be a hyperbolic group. There exists an existential formula \[\mathrm{PreStrong}^{\exists}(x_1,\ldots,x_p)\] such that, for every hyperbolic group $G'$ satisfying $K_{G'}\leq K_G$, and for every tuple $\bm{g'}=(g'_1,\ldots,g'_p)$ of elements of $G'$, the following two assertions are equivalent:
\begin{enumerate}
\item the group $G'$ satisfies $\mathrm{Special}^{\forall}(\bm{g'})\wedge\mathrm{PreStrong}^{\exists}(\bm{g}')$;
\item the map $\varphi_{\bm{g'}}:\lbrace s_1,\ldots,s_p\rbrace\rightarrow G'$ defined by $s_i\mapsto g'_i$ for $1\leq i\leq p$ extends to a pre-strongly special homomorphism from $G$ to $G'$.
\end{enumerate}
\end{prop}

\begin{proof}Let $G=\langle s_1,\ldots ,s_p \ \vert \ \Sigma(s_1,\ldots,s_p)=1\rangle$ be a finite presentation of $G$, where $\Sigma(s_1,\ldots,s_p)=1$ denotes a finite system of equations in $p$ variables. Let $\lbrace C_1,\ldots,C_q\rbrace$ be a set of representatives of the conjugacy classes of finite subgroups of $G$. Let $I$ be the subset of $\llbracket 1,q\rrbracket$ such that $N_i=N_G(C_i)$ is virtually cyclic if and only if $i\in I$.





For $i\notin I$, there exists a quantifier-free formula $\mathrm{NonVC}_i(\boldsymbol{x})$ such that, for every $\boldsymbol{g}'$ in $\mathrm{Sol}_{G'}(\Sigma)$, the group $N_{G'}(\varphi_{\boldsymbol{g}'}(C_i))$ is not virtually cyclic if and only if $G'\models\mathrm{NonVC}_i(\boldsymbol{g}')$. This uses the fact that $K_{G'}\leq K_G$, which implies that two elements $g'_1,g'_2$ of $G'$ of infinite order generate a non virtually cyclic subgroup if and only if $[{g'_1}^{K_G!},{g'_2}^{K_G!}]\neq 1$ (see Corollary \ref{pour plus tard!!!}).

For $i\notin I$, let $E_i:=E_{G}(N_G(C_i))$ and $c_i:=c(N_G(E_i))$ (see Definition \ref{chain}). Consider $\boldsymbol{h}_i\in N_G(E_i)^{c_i}$ a $G$-chain. This chain can be written as a $c_i$-tuple of words $\boldsymbol{h}_i(\boldsymbol{s})$.

We define the existential formula $\mathrm{PreStrong}^{\exists}(x_1,\ldots,x_p)$ as follows:\[\mathrm{PreStrong}^{\exists}(\bm{x}) : \bigwedge_{i\notin I}\left(\mathrm{NonVC}_i(\boldsymbol{x})\wedge {\mathrm{Chain}}_i^{\exists}(\boldsymbol{h}_i(\boldsymbol{x}))\right),\]where $\mathrm{Chain}_i^{\exists}(\boldsymbol{h}_i)$ denotes the $\exists$-formula given by Lemma \ref{chain2}.

For any $\boldsymbol{g}'\in G'^p$, the group $G'$ satisfies $\mathrm{Special}^{\forall}(\bm{g'})\wedge\mathrm{PreStrong}^{\exists}(\bm{g}')$ if and only if the homomorphism $\varphi_{\boldsymbol{g}'} : G \rightarrow G': \boldsymbol{s}\mapsto \boldsymbol{g}'$ is pre-strongly special.\end{proof}

Corollary \ref{propi} below, which follows immediately from Proposition \ref{mntab} above, tells us that the existence of a pre-strongly special homomorphism from $G$ to a hyperbolic group $G'$ such that $K_{G'}\leq K_G$ is captured by a single $\exists\forall$-sentence $\zeta_G$ that does not depend on $G'$.

\begin{co}\label{propi}Let $G$ be a hyperbolic group. Let us define the $\exists\forall$-sentence $\zeta_G$ by \[\zeta_G : \exists \bm{x} \ \mathrm{Special}^{\forall}(\bm{x})\wedge\mathrm{PreStrong}^{\exists}(\bm{x}).\] Note that this sentence is satisfied by $G$ since the identity of $G$ is a pre-strongly special homomorphism. For every hyperbolic group $G'$ such that $K_{G'}\leq K_G$, the following two assertions are equivalent:
\begin{enumerate}
\item the group $G'$ satisfies $\zeta_G$;
\item there exists a pre-strongly special homomorphism from $G$ to $G'$.
\end{enumerate}
\end{co}

\subsection{Proof of $(2)\Leftrightarrow (4)$}

The equivalence $(2)\Leftrightarrow(4)$ of Theorem \ref{principal} follows immediately from Corollary \ref{propi} together with Lemma \ref{lemmeAE}.

\begin{co}[Equivalence $(2)\Leftrightarrow (4)$ of Theorem \ref{principal}]Let $G$ and $G'$ be hyperbolic groups. There exist strongly special homomorphisms $\varphi : G \rightarrow G'$ and $\varphi' : G' \rightarrow G$ if and only if $G'\models \zeta_G$ and $G\models \zeta_{G'}$.\end{co}

\begin{proof}Suppose that the morphisms $\varphi$ and $\varphi'$ exist. Since they are injective on finite groups, we have $K_G=K_{G'}$. Hence Corollary \ref{propi} applies and guarantees that $G$ satisfies $\zeta_{G'}$ and $G'$ satisfies $\zeta_G$.

Conversely, suppose that $G$ satisfies $\zeta_{G'}$ and $G'$ satisfies $\zeta_G$. Up to exchanging $G$ and $G'$, one can assume without loss of generality that $K_{G'}\leq K_G$. By Corollary \ref{propi}, there exists a pre-strongly special homomorphism $\varphi : G \rightarrow G'$. As a consequence, since $\varphi$ is injective on finite groups, we have $K_G=K_{G'}$. Then again by Corollary \ref{propi}, there exists a pre-strongly special homomorphism $\varphi' : G' \rightarrow G$.\end{proof}

\subsection{Proof of $(3)\Rightarrow (4)$}

\begin{prop}[$(3)\Rightarrow (4)$]\label{stronglyexistence}Let $G$ and $G'$ be hyperbolic groups. Let $(\varphi_n : G \rightarrow G')_{n\in\mathbb{N}}$ and $(\varphi'_n : G' \rightarrow G)_{n\in\mathbb{N}}$ be two discriminating sequences of special homomorphisms. Then, for $n$ large enough, $\varphi_n: G \rightarrow G'$ and $\varphi'_n: G' \rightarrow G$ are strongly special.
\end{prop}

Note that $K_G=K_{G'}$, with $G$ and $G'$ as above. Indeed, $\varphi_n$ and $\varphi'_n$ are injective on finite groups for $n$ large enough. So Proposition \ref{stronglyexistence} is an immediate consequence of the following result combined with Lemma \ref{lemmeAE}.

\begin{prop}\label{stronglyexistence2}Let $G$ and $G'$ be two hyperbolic groups, and let $(\varphi_n : G \rightarrow G')_{n\in\mathbb{N}}$ be a discriminating sequence of special homomorphisms. Suppose that $K_G=K_{G'}$. Then $\varphi_n$ is pre-strongly special, for $n$ large enough.\end{prop}

\begin{proof}We keep the same notations as in the proof of Proposition \ref{propi} above. Note that the existential formula $\mathrm{PreStrong}^{\exists}(\bm{x})$ is satisfied by the generating set $\boldsymbol{s}=\bm{x}$ of $G$. So, for $n$ large enough, the statement $\mathrm{PreStrong}^{\exists}(\varphi_n(\bm{s}))$ is satisfied by $G'$ since the sequence $(\varphi_n)_{n\in\mathbb{N}}$ is discriminating. Moreover, the morphism $\varphi_n$ being special, the tuple $\varphi_n(\boldsymbol{s})$ satisfies the universal formula $\mathrm{Special}^{\forall}(\bm{x})$. Thus, the group $G'$ satisfies the statement $\mathrm{Special}^{\forall}(\varphi_n(\boldsymbol{s}))\wedge\mathrm{PreStrong}^{\exists}(\varphi_n(\boldsymbol{s}))$. It follows from Proposition \ref{mntab} that $\varphi_n$ is pre-strongly special.\end{proof}

\begin{rque}Under the hypotheses of Proposition \ref{stronglyexistence2} above, $\varphi_n$ is not necessarily strongly special for $n$ large enough. More precisely, $\varphi_n(E_G(N_G(C)))$ is not necessarily equal to $E_{G'}(N_{G'}(\varphi_n(C)))$ for $n$ large enough, as shown by the following example. Let \[G=\langle c \ \vert \ c^4=1\rangle\times F_2 \ \ \text{and} \ \ G'=\langle G, t \ \vert \ [t,c^2]=1\rangle.\]Note that $E_G(N_G( c^2))=E_G(G)=\langle c\rangle$, and that $E_{G'}(N_{G'}( c^2))=E_{G'}(G')=\langle c^2\rangle$. Thus, the inclusion $\varphi$ of $G$ into $G '$ satisfies \[\varphi(E_G(N_G( c^2)))\not\subset E_{G'}(N_{G'}(\varphi( c^2))).\]Note that, in this example, there is no discriminating sequence $(\varphi'_n : G' \rightarrow G)$ since the commutator $[t,c]$ is killed by any homomorphism $\varphi' : G' \rightarrow G$.\end{rque}

\section{Proof of Theorems \ref{legalte} and \ref{legalteplus}}\label{HNN}



In this section, we prove Theorem \ref{legalte} and Theorem \ref{legalteplus}. The proof of these theorems consists mainly in revisiting and generalizing the key lemma of Sacerdote's paper \cite{Sac73} dating from 1973, using some of the tools developed since then by Sela and others (theory of group actions on real trees, shortening argument, test sequences).

\subsection{Legal large extensions}

Let $G$ be a hyperbolic group, and let $C_1,C_2$ be two finite subgroups of $G$. Suppose that $C_1$ and $C_2$ are isomorphic, and let $\alpha : C_1 \rightarrow C_2$ be an isomorphism. We want to find necessary and sufficient conditions under which $\mathrm{Th}_{\forall\exists}(G)=\mathrm{Th}_{\forall\exists}(G\ast_{\alpha})$. As observed in the introduction, if $G$ and $G\ast_{\alpha}$ have the same $\forall\exists$-theory, then $G$ is non-elementary. Indeed, a hyperbolic group is finite if and only if it satisfies the first-order sentence $\forall x \ (x^N=1)$ for some integer $N\geq 1$, and virtually cyclic if and only if it satisfies $\forall x\forall y \ ([x^N,y^N]=1)$ for some integer $N\geq 1$. We therefore restrict attention to non-elementary hyperbolic groups. We will prove the following result.

\begin{te}[see Theorems \ref{legalte} and \ref{legalteplus}]\label{theorie}Let $G$ be a non-elementary hyperbolic group, and let $C_1$ and $C_2$ be two finite isomorphic subgroups of $G$. Let $\alpha : C_1 \rightarrow C_2$ be an isomorphism. Let us consider the HNN extension $\Gamma=G\ast_{\alpha}=\langle G,t \ \vert \ \mathrm{ad}(t)_{\vert C_1}=\alpha\rangle$. The following three assertions are equivalent.
\begin{enumerate}
\item The inclusion of $G$ into $\Gamma$ is a $\exists\forall\exists$-elementary embedding (see Definition \ref{elememb}).
\item $\mathrm{Th}_{\forall\exists}(\Gamma)=\mathrm{Th}_{\forall\exists}(G)$.
\item The group $\Gamma=G\ast_{\alpha}$ is a legal large extension of $G$ in the sense of Definition \ref{legal}, i.e.\ there exists an element $g\in G$ such that $gC_1g^{-1}=C_2$ and $\mathrm{ad}(g)_{\vert C_1}=\alpha$, the normalizer $N_G(C_1)$ is non-elementary, and $E_G(N_G(C_1))=C_1$. 
\end{enumerate}\end{te} 



The implication $(1)\Rightarrow (2)$ of Theorem \ref{theorie} is obvious. In order to prove the implication $(2)\Rightarrow (3)$, the first step consists in finding some $\forall\exists$-invariants of hyperbolic groups, i.e.\ some numbers that are preserved by $\forall\exists$-equivalence among hyperbolic groups. 

\begin{de}Let $G$ be a hyperbolic group. We associate to $G$ the following five integers:
\begin{itemize}
\item[$\bullet$]the number $n_1(G)$ of conjugacy classes of finite subgroups of $G$,
\item[$\bullet$]the sum $n_2(G)$ of $\vert \mathrm{Aut}_G(C_k)\vert$ for $1\leq k\leq n_1(G)$, where the $C_k$ are representatives of the conjugacy classes of finite subgroups of $G$, and \[\mathrm{Aut}_G(C_k)=\lbrace \alpha\in \mathrm{Aut}(C_k) \ \vert \ \exists g\in N_G(C_k), \ \mathrm{ad}(g)_{\vert C}=\alpha\rbrace,\]
\item[$\bullet$]the number $n_3(G)$ of conjugacy classes of finite subgroups $C$ of $G$ such that $N_G(C)$ is infinite virtually cyclic,
\item[$\bullet$]the number $n_4(G)$ of conjugacy classes of finite subgroups $C$ of $G$ such that $N_G(C)$ is non-elementary,
\item[$\bullet$]the number $n_5(G)$ of conjugacy classes of finite subgroups $C$ of $G$ such that $N_G(C)$ is non-elementary and $E(N_G(C))\neq C$.
\end{itemize}
\end{de}

One can easily see that these numbers are preserved by elementary equivalence. However, proving that they are preserved by $\forall\exists$-equivalence is a little bit more tedious.

\begin{lemme}Let $G$ and $G'$ be two hyperbolic groups. Suppose that $\mathrm{Th}_{\forall\exists}(G)=\mathrm{Th}_{\forall\exists}(G')$. Then $n_i(G)=n_i(G')$, for $1\leq i\leq 5$.\end{lemme}

As usual, we denote by $K_G$ the maximal order of a finite subgroup of $G$. Since $G$ and $G'$ have the same existential theory, we have $K_G=K_{G'}$.

\begin{proof}Let $n\geq 1$ be an integer. If $n_1(G)\geq n$, then the following $\exists\forall$-sentence, written in natural language for convenience of the reader and denoted by $\theta_{1,n}$, is satisfied by $G$: there exist $n$ finite subgroups $F_1,\ldots ,F_n$ of $G$ such that, for every $g\in G$ and $1\leq i\neq j\leq n$, the groups $gF_ig^{-1}$ and $F_j$ are distinct. Since $G$ and $G'$ have the same $\exists\forall$-theory, the sentence $\theta_{1,n}$ is satisfied by $G'$ as well. As a consequence, $n_1(G')\geq n$. It follows that $n_1(G')\geq n_1(G)$. By symmetry, we have $n_1(G)=n_1(G')$.

\smallskip

In the rest of the proof, we give similar sentences $\theta_{2,n},\ldots,\theta_{5,n}$ such that the following series of implications hold: $n_i(G)\geq n \Rightarrow G$ satisfies $\theta_{i,n} \Rightarrow$ $G'$ satisfies $\theta_{i,n}\Rightarrow n_i(G')\geq n$.

\smallskip

If $n_2(G)\geq n$, then $G$ satisfies $\theta_{2,n}$: there exist $\ell$ finite subgroups $F_1,\ldots ,F_{\ell}$ of $G$ and a finite subset $\lbrace g_{i,j}\rbrace_{1\leq i\leq \ell, \ 1\leq j\leq n_i}$ of $G$, with $n_1+\cdots+n_{\ell}=n$, such that for every $g\in G$ and $1\leq p\neq q\leq \ell$, the groups $gF_pg^{-1}$ and $F_q$ are distinct, and for every $1\leq i\leq \ell$, we have:
\begin{itemize}
\item[$\bullet$]for every $1\leq j\leq n_i$, the element $g_{i,j}$ normalizes $F_i$,
\item[$\bullet$]and for every $1\leq j\neq k\leq n_i$, the automorphisms $\mathrm{ad}(g_j)_{\vert F_i}$ and $\mathrm{ad}(g_k)_{\vert F_i}$ of $F_i$ are distinct.
\end{itemize}

\smallskip

If $n_3(G)\geq n$, then $G$ satisfies $\theta_{3,n}$: there exist $n$ finite subgroups $F_1,\ldots ,F_n$ of $G$ such that, for every $g\in G$ and $1\leq i\neq j\leq n$, the groups $gF_ig^{-1}$ and $F_j$ are distinct, and for every $1\leq i\leq n$ and $g,h\in N_G(F_i)$, we have $[g^{K_G!},h^{K_G!}]=1$.

\smallskip

If $n_4(G)\geq n$, then $G$ satisfies $\theta_{4,n}$: there exist $n$ finite subgroups $F_1,\ldots ,F_n$ of $G$ and a finite subset $\lbrace g_{i,j}\rbrace_{1\leq i\leq n, \ 1\leq j\leq 2}$ of $G$ such that, for every $g\in G$ and $1\leq p\neq q\leq n$, the groups $gF_pg^{-1}$ and $F_q$ are distinct, and for every $1\leq i\leq n$, we have:
\begin{itemize}
\item[$\bullet$]$g_{i,1}$ and $g_{i,2}$ normalize $F_i$,
\item[$\bullet$]and $[g_{i,1}^{K_G!},g_{i,2}^{K_G!}]\neq 1$.
\end{itemize}

\smallskip

If $n_5(G)\geq n$, then $G$ satisfies $\theta_{5,n}$: there exist $2n$ finite subgroups $F_1,\ldots ,F_n$ and $F'_1,\ldots,F'_n$ of $G$ and a finite subset $\lbrace g_{i,j}\rbrace_{1\leq i\leq n, \ 1\leq j\leq 2}$ of $G$ such that, for every $g\in G$ and $1\leq p\neq q\leq n$, the groups $gF_pg^{-1}$ and $F_q$ are distinct, and for every $1\leq i\leq n$, we have:
\begin{itemize}
\item[$\bullet$]$g_{i,1}$ and $g_{i,2}$ normalize $F_i$,
\item[$\bullet$]$[g_{i,1}^{K_G!},g_{i,2}^{K_G!}]\neq 1$,
\item[$\bullet$]for every $h\in G$, if $h$ normalizes $F_i$ then $h$ normalizes $F'_i$,
\item[$\bullet$]and $F_i$ is strictly contained in $F'_i$.
\end{itemize}

This concludes the proof of the lemma.\end{proof}

As an application, we prove the implication $(2)\Rightarrow (3)$ of Theorem \ref{theorie}.

\begin{prop}[Implication $(2)\Rightarrow (3)$ of Theorem \ref{theorie}]\label{facile}We keep the same notations as in Theorem \ref{theorie}. If $\mathrm{Th}_{\forall\exists}(G)=\mathrm{Th}_{\forall\exists}(\Gamma)$, then $\Gamma$ is a legal large extension of $G$.\end{prop}

\begin{proof}Thanks to the previous lemma, we know that $n_i(G)=n_i(\Gamma)$ for $1\leq i\leq 5$. Recall that $\Gamma=\langle G,t \ \vert \ \mathrm{ad}(t)_{\vert C_1}=\alpha\rangle$, where $\alpha : C_1 \rightarrow C_2$ is an isomorphism. The equality $n_1(G)=n_1(\Gamma)$ implies that $gC_1g^{-1}=C_2$ for some $g\in G$. It follows that $g^{-1}t$ induces an automorphism of $C_1$. 

Then, observe that for every finite subgroup $F$ of $G$, we have $\vert \mathrm{Aut}_{\Gamma}(F)\vert \geq\vert \mathrm{Aut}_{G}(F)\vert$. Thus, the equality $n_2(G)=n_2(\Gamma)$ guarantees that $\vert \mathrm{Aut}_{\Gamma}(C_1)\vert =\vert \mathrm{Aut}_{G}(C_1)\vert$. Hence, since $\mathrm{ad}(g^{-1}t)_{\vert C_1}$ belongs to $\mathrm{Aut}_{\Gamma}(C_1)$, there exists an element $g'\in N_G(C_1)$ such that $\mathrm{ad}(g')_{\vert C_1}=\mathrm{ad}(g^{-1}t)_{\vert C_1}$. Therefore, $gg'\in G$ induces by conjugacy the same automorphism of $C_1$ as $t$, which proves that the first condition of Definition \ref{legal} holds. 

The equalities $n_3(G)=n_3(\Gamma)$ and $n_4(G)=n_4(\Gamma)$ ensure that $N_G(C_1)$ is non-elementary. Indeed, if $N_G(C_1)$ were finite, then $N_{\Gamma}(C_1)$ would be infinite virtually cyclic and $n_3(\Gamma)$ would be at least $n_3(G)+1$; similarly, if $N_G(C_1)$ were infinite virtually cyclic, then $N_{\Gamma}(C_1)$ would be non-elementary and $n_4(\Gamma)\geq n_4(G)+1$. Hence, the second condition of Definition \ref{legal} is satisfied. 

Lastly, it follows from the fact that $n_5(G)=n_5(\Gamma)$ that $E_G(N_G(C_1))=C_1$, otherwise $n_5(\Gamma)\geq n_5(G)+1$, since $E_{\Gamma}(N_{\Gamma}(C_1))=C_1$. Thus, the third condition of Definition \ref{legal} holds. As a conclusion, $\Gamma$ is a legal large extension of $G$.\end{proof}

We shall now prove the difficult part of Theorem \ref{theorie}, namely the following result.

\begin{prop}[Implication $(3)\Rightarrow (1)$ of Theorem \ref{theorie}]\label{theorie2}We keep the same notations as in Theorem \ref{theorie}. If $\Gamma$ is a legal large extension of $G$, then the inclusion of $G$ into $\Gamma$ is a $\exists\forall\exists$-elementary embedding.\end{prop}

First, we define a notion of test sequences adapted to our context.


\subsection{Test sequences}

\begin{de}\label{suitetest}Let $G$ be a non-elementary hyperbolic group, and let \[\Gamma = \langle G,t \ \vert \ \mathrm{ad}(t)_{\vert C}=\mathrm{id}_C\rangle\] be a legal extension of $G$ over a finite subroup $C$. Let $(\varphi_n : \Gamma \rightarrow G)_{n\in\mathbb{N}}$ be a sequence of homomorphisms. For every integer $n$, let $t_n:=\varphi_n(t)$. The sequence $(\varphi_n)_{n\in\mathbb{N}}$ is called a test sequence if the following conditions hold:
\begin{enumerate}
\item for every $n$, the morphism $\varphi_n$ is a retraction, i.e.\ $\varphi_n(g)=g$ for every $g\in G$;
\item the translation length $\vert \vert t_n\vert \vert$ of $t_n$ goes to infinity when $n$ goes to infinity;
\item there exists a sequence $(\varepsilon_n)_{n\in\mathbb{N}}$ converging to $0$ such that, for every $n$, the element $t_n$ satisfies the $\varepsilon_n$-small cancellation condition (see Definition \ref{SCC}), and $M(t_n)=\langle t_n\rangle\times C$. Therefore, the image of $t_n$ in $M(t_n)/C$ has not root.
\end{enumerate}
\end{de}

\begin{rque}Note that any subsequence of a test sequence is a test sequence as well.
\end{rque}

The following easy lemma will be useful in the sequel.

\begin{lemme}\label{intersection}Let $(\varphi_n)_{n\in\mathbb{N}}$ be a test sequence. For every infinite subset $A\subset \mathbb{N}$, we have \[\bigcap_{n\in A}M(t_n)=C.\] 
\end{lemme}

\begin{proof}Suppose that $g$ belongs to $M(t_n)$ for every $n\in A$. Then, there exists an integer $k_n$ and an element $c_n\in C$ such that $g=t_n^{k_n}c_n$, for every $n\in A$. Now, observe that $k_n$ must be equal to $0$ for every $n$ large enough, otherwise (up to extracting a subsequence) $\vert\vert t_n^{k_n}\vert\vert$ goes to infinity, and so does the constant $\vert\vert g\vert\vert$, which is a contradiction. It follows that $g$ belongs to $C$.\end{proof}

The following result is well-known, however we are not aware of a reference in the literature.

\begin{lemme}\label{quasiconvfini}Let $G$ be a $\delta$-hyperbolic group, and let $C$ be a finite subgroup of $G$. Then the centralizer $C_G(C)$ of $C$ is quasi-convex in $G$.\end{lemme}

\begin{proof}
Let $c\in C$, and let $h\in C_G(c)$. Let $x$ be a vertex on the geodesic $[1,h]$ in the Cayley graph of $G$, for a given finite generating set $S$ of $G$. Let $d_S$ be the induced metric on $G$. Since the elements $c$ and $x^{-1}cx$ are conjugate, there exists an element $y$ such that $x^{-1}cx=ycy^{-1}$ and $d_S(1,y) \leq 2d_S(1,c) +R(\delta,\vert S\vert)=:K$ according to Proposition 2.3 in \cite{BH05}, where $R(\delta,\vert S\vert)$ is a constant that only depends on $\delta$ and $\vert S\vert$. Observe that $xy$ centralizes $c$, and that $d_S(x,xy)=d_S(1,y)\leq K$. This shows that $C_G(c)$ is $K$-quasi-convex in $G$. Now, recall that the intersection of two quasi-convex subgroups of a hyperbolic group is still quasi-convex (see for instance \cite{Sho91}). Hence, since $C$ is finite, the intersection $C_G(C)=\cap_{c\in C}C_G(c)$ is quasi-convex in $G$.\end{proof}

We now build a test sequence.

\begin{prop}\label{jesaispas}Let $G$ be a non-elementary hyperbolic group, and let \[\Gamma = \langle G,t \ \vert \ \mathrm{ad}(t)_{\vert C}=\mathrm{id}_C\rangle\] be a legal extension of $G$ over a finite subroup $C$. Then, there exists a test sequence $(\varphi_n : \Gamma \rightarrow G)_{n\in\mathbb{N}}$.\end{prop}

\begin{proof}We proceed in two stages.
\begin{enumerate}
\item First, we define a subgroup $ \langle a, b \rangle $ of the centralizer $ C_G (C) $ of $C$ in $G$ that is free of rank 2 and quasi-convex in $ G $.
\item Then, we build a sequence of elements $ (t_n)_{n\in\mathbb{N}} $ that satisfies the $ (1 / n)$ -small cancellation condition in the free group $ \langle a, b \rangle $, and check that, for every $n$, the retraction $ \varphi_n: \Gamma \twoheadrightarrow G: t\mapsto t_n $ is well-defined and has the expected properties.
\end{enumerate}

Recall that $E_G(N_G(C))=C$. So, by \cite{Ol93} Lemma 3.3, there exists an element $a\in C_G(C)$ of infinite order such that $M(a)=\langle a \rangle \times C$ in $ G $. Recall that $N_G(C)$ is non-elementary, by definition of a legal large extension. As a finite-index subgroup of $N_G(C)$, the centralizer $C_G(C)$ is non-elementary as well. By \cite{Cha12} Corollary I.1.9, there exists an element $b\in C_G(C)$ of infinite order such that the subgroup of $C_G(C)$ generated by $\lbrace a,b,C\rbrace$ is quasi-convex in $C_G(C)$ and isomorphic to $\langle a,C\rangle\ast_C\langle b,C\rangle = C\times (\langle a\rangle\ast \langle b\rangle)$. Hence, the subgroup $\langle a, b \rangle$ is quasi-convex in $C_G(C) $, and free of rank 2. By Lemma \ref{quasiconvfini}, the centralizer $ C_G (C) $ of $C$ is quasi-convex in $ G $. Thus, the free group $ \langle a, b \rangle $ is quasi-convex in $ G $ as well. For any integer $n\geq 0 $, we set $t_n=a^n b a^{n+1} b\cdots a^{2n} b$. Let $\varphi_n: \Gamma\twoheadrightarrow G$ be the retraction defined by $\varphi_n (t) = t_n $ and $\varphi_n (g) = g $ for all $g\in G $, which is well-defined since $t_n$ centralizes $C$. We will prove that $(\varphi_n)_{n\in\mathbb{N}}$ is a test sequence. Since the definition of a test sequence is clearly invariant by a change of choice of a finite generating set of $G$, let us consider a convenient finite generating set $ S $ of $ G $ that contains the two elements $a$ and $b$ introduced above. Let $ (X, d) $ denote the Cayley graph of $ G $ for this generating set. Let $ \tau_n $ be the path of $ X $ that links $ 1 $ to $ t_n $ and is labeled with the word $ t_n $ in $ a $ and $ b $, and consider the bi-infinite path $\overline{\tau}_n=\cup_{k\in\mathbb{Z}} t_n^k\tau_n$. A standard argument shows that $\overline{\tau}_n$ is a quasi-geodesic in $(X,d)$, for some constants that do not depend on $n$. Consequently, $\overline{\tau}_n$ lies in the $\lambda$-neighborhood of $A(t_n)$ for some constant $\lambda\geq 0$ independent from $n$. Similarly, let $\alpha$ be the edge of $X$ linking $1$ to $a$, let $\overline{\alpha}$ denote the quasi-geodesic $\overline{\alpha}=\cup_{k\in\mathbb{Z}} a^k\alpha$ and let $\mu$ be a constant such that $\overline{\alpha}$ lies in the $\mu$-neighborhood of $A(a)$.

Let $d'$ denote the metric in the free group $ \langle a, b \rangle $ for the generating set $ \lbrace a, b \rbrace $. Since $t_n$ is cyclically reduced in $\langle a,b\rangle$,  $\vert\vert t_n\vert\vert_{d'}= d'(1, t_n) \sim (3/2) n^2 $. Since $ \langle a, b \rangle $ is quasi-convex in $ G $ by construction, there is a constant $ R> 0 $ such that $ \vert\vert t_n\vert\vert \geq Rn^ 2 $ for all $ n $ large enough.

It remains to prove the third condition of Definition \ref{suitetest}. Classically, since the element $ a $ has infinite order, there exists a constant $ N \geq 0 $ such that, for every element $ g \in G $, if $ \Delta (a, gag^{-1}) \geq N $, then $ g $ belongs to $M(a)=\langle a\rangle\times C$ (see paragraph \ref{Coulon}). Let $n_0$ be an integer such that $ Rn_0 $ is large compared to $N'=N+204\delta+2\lambda+2\mu$. We will show that for every $ n \geq n_0 $, the element $ t_n $ satisfies the $( 1 / n) $-small cancellation condition. Let $ n $ be an integer greater than $ n_0 $. Consider an element $ g \in G $ such that $\Delta(t_n,gt_ng^{-1})\geq \vert\vert t_n\vert\vert/n$. We will show that $ g $ belongs to the subgroup $\langle t_n\rangle\times C$.

We first show that $ g $ belongs to the subgroup $\langle a,b\rangle\times C$. Since \[\Delta(t_n,gt_ng^{-1})\geq \vert\vert t_n\vert\vert/n\geq Rn\geq Rn_0 >> N',\] we can choose two subpaths $\nu_n$ and $\mu_n$ of $\overline{\tau}_n$ and $g\overline{\tau}_n$ respectively, of length $ N' $ and labeled by $ a^{N'} $, such that $\mathrm{diam}((\nu_n)^{+(100\delta+\lambda)}\cap(\mu_n)^{+(100\delta+\lambda)})\geq N'$. Denoting by $ x_n $ and $ y_n $ the initial points of $ \nu_n $ and $ \mu_n $ respectively, we have $\mathrm{diam}(x_n\overline{\alpha}^{+(100\delta+\lambda)}\cap y_n\overline{\alpha}^{+(100\delta+\lambda)})\geq N'$. It follows that 
$\mathrm{diam}(A(a)^{+(100\delta+\lambda+\mu)}\cap x_n^{-1}y_nA(a)^{+(100\delta+\lambda+\mu)})\geq N'$. By Lemma 2.13 in \cite{Cou13}, we have:
{\small
\begin{align*}
\Delta(a,{x_n^{-1}y_n}a{x_n^{-1}y_n}^{-1})& \geq \mathrm{diam}(A(a)^{+(100\delta+\lambda+\mu)}\cap x_n^{-1}y_nA(a)^{+(100\delta+\lambda+\mu)}) - (204\delta+2\lambda+2\mu) \\
 & \geq N' - (204\delta+2\lambda+2\mu) =N.
\end{align*}}
So $ x_n ^ {-1} y_n $ belongs to $M(a)=\langle a\rangle\times C$. Now, observe that $ x_n $ is a word in $ a $ and $ b $ since it is on the quasi-geodesic $ \overline {\tau}_n $. Similarly, $ y_n $ can be written as $ y_n = gz_n $ with $ z_n $ a word in $ a $ and $ b $. It follows that $ g $ belongs to the subgroup $ \langle a, b \rangle \times C $.

Up to replacing $g$ with $gc$ for some $c \in C$, we can now assume that $g$ belongs to the free group $\langle a,b\rangle$. This does not affect the condition $\Delta(t_n,gt_ng^{-1})\geq \vert\vert t_n\vert\vert/n$; indeed, $gct_n{gc}^{-1}=gt_ng^{-1}$, since $ t_n $ centralizes $ C $. Recall in addition that the group $ \langle a, b \rangle $ is quasi-isometrically embedded into $ G $. Denoting by $ A> 0 $ and $ B \geq 0 $ two constants such that \[\frac{1}{A}d'(x,y)-B \leq d(x,y)\leq Ad'(x,y)+B\] for all $ x, y \in \langle a, b \rangle \times C $, we can verify that the following inequality holds, in the Cayley graph $ Y $ of the free group $ \langle a, b \rangle $ equipped with the distance $ d'$:\[\mathrm{diam}\left((\overline{\tau}_n)^{+(A(100\delta+B)+1)}\cap (g\overline{\tau}_n)^{+(A(100\delta+B)+1)}\right)\geq d'(1,t_n)/(2A^2n).\] Since the Cayley graph of $ \langle a, b \rangle $ is a tree, this inequality tells us that the axes of $ t_n $ and $ gt_ng^{-1} $ have an overlap of length larger than $4n-2$ in this tree.

To conclude, let us observe that two distinct cyclic conjugates of $a^nba^{n+1}b\cdots a^{2n}b$ have at most their first $ 4n-2 $ letters in common. Thus, if the axes of $ t_n $ and $ gt_ng^{-1} $ have a common subsegment in $ Y $ of length $> 4n-2 $, then $ t_n $ and $ gt_ng^{-1} $ have the same axis, so $ t_n $ and $ g $ have a common root. Now, observe that $ t_n $ has no root. It follows that $ g $ is a power of $ t_n $, which concludes the proof.\end{proof}


Let $G$ be a non-elementary hyperbolic group. Consider $\Gamma=\langle G, t \ \vert \ \mathrm{ad} (t) _ {\vert C} = \mathrm{id}_C \rangle$ a legal extension of $G$, and $(\varphi_n: \Gamma \rightarrow G)_{n\in\mathbb{N}}$ a test sequence. Let $\Gamma'$ be a finitely generated overgroup of $\Gamma$. Suppose that each $\varphi_n$ extends to a homomorphism $\widehat{\varphi}_n : \Gamma'\rightarrow G$. Let $L$ be the quotient of $\Gamma'$ by the stable kernel of the sequence $(\widehat{\varphi}_n)_{n\in\mathbb{N}}$, and $ r: \Gamma' \twoheadrightarrow L $ the associated epimorphism. Since $G$ is equationally noetherian (according to \cite{Sel09} and \cite{RW14} Corollary 6.13), there exists (for $n$ sufficiently large) a unique homomorphism $\rho_n : L\rightarrow G $ such that $\varphi_n=\rho_n\circ r$. Let $\lambda_n={\mathrm{max}}_{s\in S} \ d(1,\rho_n(s))$ be the displacement of $\rho_n$, where $S$ is a finite generating set of $L$ containing the image of $t$ in $L$. Let $(X,d)$ denote a Cayley graph of $G$, and consider the rescaled metric $d_n=d/\lambda_n$. Last, let $\omega$ be a non-principal ultrafilter and let $(X_{\omega},d_{\omega})$ be the ultralimit of $((X,d_n))_{n\in\mathbb{N}}$. By Theorem \ref{existencearbre}, $X_{\omega}$ is a real tree and there exists a unique minimal $L$-invariant non-degenerate subtree $T_L\subset X_{\omega}$. Moreover, some subsequence of the sequence $((X,d/\lambda_n))_{n\in\mathbb{N}}$ converges to $T_L$ in the Gromov-Hausdorff topology.

\begin{lemme}\label{lemme2}If $\Gamma$ does not fix a point in $T_L$, then the minimal subtree $ T_{\Gamma} $ is isometric to the Bass-Serre tree $T'$ of the splitting $ G \ast_C $, up to rescaling the metric on $T'$.\end{lemme}


\begin{proof}Suppose that $\Gamma$ does not fix a point of $T_L$, and let us prove that $ T_{\Gamma} $ is isometric to the Bass-Serre tree $T'$ of the splitting $G\ast_C$. Observe that $G$ is elliptic in $T_L$, by the first assumption of Definition \ref{suitetest}. More precisely, the point $x:=(1)_{\omega}\in T_L$ is fixed by $G$. Note that $\vert \vert t_n\vert\vert/\lambda_n$ does not approach 0 as $n$ goes to infinity, otherwise $\Gamma$, which is generated by $G$ and $t$, would be elliptic in $T_L$. Hence, $t$ acts hyperbolically on $T_{\Gamma}$. In addition, the translation axis of $t$ contains $x$. Therefore, up to rescaling the metric on $T_{\Gamma}$, there exists a simplicial map $ f: T'\rightarrow T_{\Gamma}$ that is isometric in restriction to the axis of $t$. In order to prove that $f$ is an isometric embedding, let us prove that there is no folding. Note that the surjectivity is automatically satisfied because $T_{\Gamma}$ is minimal. Assume towards a contradiction that there is a folding at the vertex $ v \in T '$ fixed by $G$. Let $w$ and $w'$ denote two vertices of $ T' $ adjacent to $ v $ such that $f([v,w])\cap f([v,w'])$ is non-degenerate. One can assume without loss of generality that $w=tv$ (up to translating $w$ by an element of $G$ and replacing $t$ with $t^{- 1}$). Since $tv$ and $t^{- 1}v$ are on the axis of $t$, and since $ f $ is isometric on this axis, $f([v, tv])\cap f([v, t^{- 1} v])=\lbrace f(v)\rbrace $. Therefore, the vertex $ w '$ is of the form $ gtv $ or $ gt^{- 1} v $, for some element $ g\in G $ that does not belong to $ C $ (indeed, if $ g \in C $, then $ gtv = tv $ and $ gt^{- 1} v = t^{- 1} v $). Thus, the axes of $ t $ and $ gtg^{-1} $ have an overlap $ I $ of length $> 0 $ in the limit tree $ T $. It follows that \[\Delta(t_n,{\rho_n(g)}t_n{\rho_n(g)}^{-1})\geq \eta \lambda_n\] for every $ n $ large enough, for some $\eta>0$. But $\lambda_n/\vert\vert t_n\vert\vert$ is greater than $1$, and $(\varepsilon_n)_{n\in\mathbb{N}}$ approaches $0$ when $n$ goes to infinity. Thus, for $n$ large enough, we have: \[\Delta(t_n,{\rho_n(g)}t_n{\rho_n(g)}^{-1})\geq \eta \lambda_n \geq \varepsilon_n \vert \vert t_n\vert \vert.\]Then, since $t_n$ satisfies the $\varepsilon_n$-small cancellation condition, the element $\rho_n(g)=g$ belongs to $M(t_n)$. By Lemma \ref{intersection}, $g$ belongs to $C$, which is a contradiction. Hence, $T_{\Gamma}$ is isometric to $T'$.\end{proof}

\begin{co}\label{disc}Every test sequence is discriminating.\end{co}

\begin{proof}
Let $G$ be a non-elementary hyperbolic group. Consider $\Gamma=\langle G, t \ \vert \ \mathrm{ad} (t) _ {\vert C} = \mathrm{id}_C \rangle$ a legal extension of $G$, and $(\varphi_n: \Gamma \rightarrow G)_{n\in\mathbb{N}}$ a test sequence. By taking $\Gamma'=L=\Gamma$ in the previous lemma, $T_{\Gamma}$ is isometric to the Bass-Serre tree $T'$ of the splitting $\Gamma=G\ast_C$. Let $\gamma\in \Gamma$ be a non-trivial element. 

If $\gamma$ belongs to a conjugate of $G$, then $\rho_n(\gamma)\neq 1$ for every $n$ since $\rho_n$ is the identity on $G$. 

If $\gamma$ does not belong to a conjugate of $G$, i.e.\ if $\gamma$ is not elliptic in the splitting $T'$, then it acts hyperbolically on $T$, because $T'$ and $T$ are isometric. Thus, $\rho_n(\gamma)$ is non-trivial for infinitely many $n$ (otherwise $\gamma$ would be elliptic). It remains to prove that $\rho_n(\gamma)$ is non-trivial for every $n$ large enough. Assume towards a contradiction that some infinite subsequence $(\rho_{f(n)})$ kills $\gamma$ for every $n$. Applying the previous argument to $(\rho_{f(n)})_{n\in\mathbb{N}}$ instead of $(\rho_n)_{n\in\mathbb{N}}$, we get a contradiction. Hence, the sequence of morphisms $(\rho_n)_{n\in\mathbb{N}}$ is discriminating.\end{proof}


\begin{co}\label{lemme22}With the same notations and the same hypotheses as in Lemma \ref{lemme2}, $T_{\Gamma} $ is transverse to its translates, i.e.\ for every $ h \in L \setminus \Gamma $, $hT_{\Gamma}\cap T_{\Gamma}$ is at most one point. In addition, if $e$ is an edge of $T_{\Gamma}$, there are only finitely many branch points on $e$ in $T_L$.\end{co}

\begin{proof}Let $h$ be an element of $L$ such that $hT_{\Gamma}\cap T_{\Gamma} $ is non-degenerate. Since $ T_{\Gamma} $ is isometric to the Bass-Serre tree of the splitting $G\ast_C $ of $ \Gamma $, we can find two elements $u, v \in \Gamma $ such that the axes of $ utu^{-1} $ and $ h(vtv^{-1})h^{-1} $ have a non-trivial overlap in the limit tree $T_L$, so \[\Delta(t_n,{\rho_n(u^{-1}hv)}t_n{\rho_n(u^{-1}hv)}^{-1})\geq \varepsilon_n \vert \vert t_n\vert \vert\] for $ n $ large enough. Hence, $ \rho_n (u^{- 1} hv) $ belongs to $ M (t_n) = C \times \langle t_n \rangle $. So, for every $ n $, there is an element $ c_n \in C $ and an integer $ p_n $ (possibly zero if $ u^{- 1} hv $ has finite order) such that $\rho_n(u^{-1}hv)=c_nt_n^{p_n}=\rho_n(c_nt^{p_n})$. Since $ C $ is finite, we can pass to a subsequence and assume that $ c_n = c $ for all $ n $. On the other hand, since $ t $ acts hyperbolically on $T_{\Gamma}$, the integer $ p_n $ is bounded by a constant that does not depend on $ n $. Otherwise, up to extracting, $\vert\vert\rho_n(t)\vert\vert/\vert\vert \rho_n(u^{-1}hv)\vert\vert$ tends to $0$. Hence, since $\rho_n(u^{-1}hv)/\lambda_n$ is bounded, $\vert\vert \rho_n(t)\vert\vert /\lambda_n$ tends to $0$, contradicting that $t$ is hyperbolic. Up to extracting, one can assume that $ p_n = p $ for all $ n $. So we have $\rho_n(u^{-1}hv)=\rho_n(ct^p)$ for all $ n $. The sequence $ (\rho_n)_{n\in\mathbb{N}} $ being discriminating, $h=uct^pv^{-1}\in \Gamma$, since $u,v\in \Gamma$.

Last, let $e$ be an edge of $T_{\Gamma}$. Let us prove that there are only finitely many branch points on $e$ in $T_L$. Let $M$ denote the maximal order of a finite subgroup of $G$. Note that $M$ is also the maximal order of a finite subgroup of $L$. By Theorem \ref{reinfeldt}, due to Reinfeldt and Weidmann, the action of $L$ on $T_{L}$ is $M$-superstable with finitely generated arc stabilizers. By Theorem \ref{guilev}, due to Guirardel and Levitt, the number of orbits of directions at branch points in $T_L$ is finite. Assume towards a contradiction that there are infinitely many branch points on $e$. Then there exist two non-degenerate subsegments $I$ and $J$ in $e$, with $I\cap J=\varnothing$, and an element $g\in G$ such that $gI=J$. But we proved that $T_{\Gamma}$ is transverse to its translates, so $g$ belongs to $\Gamma$. Since $T_{\Gamma}$ is isometric to the Bass-Serre tree of $G\ast_C$, it follows that $g$ fixes $e$. This is a contradiction.\end{proof}

\color{black}

\subsection{Generalized Sacerdote's lemma}

We are now ready to prove a generalization of the main lemma of \cite{Sac73}.

\begin{prop}\label{sacerdote1}Let $\Gamma$ be a legal large extension of $G$. Let $(\varphi_n : \Gamma \rightarrow G)_{n\in\mathbb{N}}$ be a test sequence. Let $\bm{g}$ be a tuple of elements of $G$. Let $\Sigma(\bm{x},\bm{y},\bm{g})=1\wedge \Psi(\bm{x},\bm{y},\bm{g})\neq 1$ be a conjunction of equations and inequations in the $p$-tuple $\bm{x}$ and the $q$-tuple $\bm{y}$. Let $\bm{\gamma}$ be a $p$-tuple of elements of $\Gamma$. Suppose that $G$ satisfies the following condition: for every $n$, there exists a $q$-tuple $\bm{g}_n\in G^q$ such that \[\Sigma(\varphi_n(\bm{\gamma}),\bm{g}_n,\bm{g})=1 \ \wedge \ \Psi(\varphi_n(\bm{\gamma}),\bm{g}_n,\bm{g})\neq 1.\] Then there exists a retraction $r$ from $\Gamma_{\Sigma}:=\langle \Gamma , \bm{y} \ \vert \ \Sigma(\bm{\gamma},\bm{y},\bm{g})=1\rangle$ onto $\Gamma$ such that all components of the tuple $r(\Psi(\bm{\gamma},\bm{y},\bm{g}))$ are non-trivial. In particular, the $q$-tuple $\bm{\gamma}':=r(\bm{y})\in \Gamma^q$ satisfies \[\Sigma(\bm{\gamma},\bm{\gamma}',\bm{g})=1 \ \wedge \ \Psi(\bm{\gamma},\bm{\gamma}',\bm{g})\neq 1.\]\end{prop}

Before proving this result, whose proof is quite technical, we will use it to deduce that the inclusion of $G$ into $\Gamma$ is a $\exists\forall\exists$-elementary embedding. We begin by proving a corollary to Proposition \ref{sacerdote1} above, which allows us to deal with disjunctions of systems of equations and inequations.

\begin{co}\label{sacerdote2}Let $\Gamma$ be a legal large extension of $G$, and let $(\varphi_n : \Gamma \rightarrow G)_{n\in\mathbb{N}}$ be a test sequence. Let $g$ be a tuple of elements of $G$. Let \[\bigvee_{k=1}^N(\Sigma_k(\bm{x},\bm{y},\bm{g})=1 \ \wedge \ \Psi_k(\bm{x},\bm{y},\bm{g})\neq 1)\] be a disjunction of systems of equations and inequations in the $p$-tuple $\bm{x}$ and the $q$-tuple $\bm{y}$. Let $\bm{\gamma}$ be a $p$-tuple of elements of $\Gamma$. Suppose that $G$ satisfies the following condition: for every integer $n$, there exists a $q$-tuple $\bm{g}_n\in G^q$ such that \[\bigvee_{k=1}^N(\Sigma_k(\varphi_n(\bm{\gamma}),\bm{g}_n,\bm{g})=1 \ \wedge \ \psi_k(\varphi_n(\bm{\gamma}),\bm{g}_n,\bm{g})\neq 1).\] Then there exists a $q$-tuple $\bm{\gamma}'$ such that \[\bigvee_{k=1}^N(\Sigma_k(\bm{\gamma},\bm{\gamma'},\bm{g})=1 \ \wedge \ \Psi_k(\bm{\gamma},\bm{\gamma'},\bm{g})\neq 1).\]\end{co}

\begin{proof}Up to extracting a subsequence of $(\varphi_n)$ (which is still a test sequence), one can assume that there exists an integer $1\leq k\leq N$ such that, for every $n$, there exists $\bm{g}_n\in G^q$ such that $\Sigma_k(\varphi_n(\bm{\gamma}),\bm{g}_n,\bm{g})=1$ and $\Psi_k(\varphi_n(\bm{\gamma}),\bm{g}_n,\bm{g})\neq 1$. Proposition \ref{sacerdote1} applies and establishes the existence of a tuple $\bm{\gamma}'\in \Gamma^q$ satisfying $\Sigma_k(\bm{\gamma},\bm{\gamma'},\bm{g})=1$ and $\Psi_k(\bm{\gamma},\bm{\gamma'},\bm{g})\neq1$, which concludes the proof.\end{proof}

We deduce Proposition \ref{theorie2} from Corollary \ref{sacerdote2}.



\medskip

\begin{proofnash}Let $\theta(\bm{t})$ be a $\exists\forall\exists$-formula with $m$ free variables. This formula has the following form:\[\theta(\bm{t}):\exists \bm{x}\forall \bm{y} \exists \bm{z}  \bigvee_{k=1}^N(\Sigma_k(\bm{x},\bm{y},\bm{z},\bm{t})=1 \ \wedge \ \Psi_k(\bm{x},\bm{y},\bm{z},\bm{t})\neq 1).\]Let $\bm{g}$ be a tuple of elements of $G$ such that $G\models \theta(\bm{g})$, and let us prove that $\Gamma\models \theta(\bm{g})$. There exists a tuple of elements of $G$, denoted by $\bm{x}$, such that \begin{equation}G\models \forall \bm{y} \exists \bm{z} \bigvee_{k=1}^N(\Sigma_k(\bm{x},\bm{y},\bm{z},\bm{g})=1 \ \wedge \ \Psi_k(\bm{x},\bm{y},\bm{z},\bm{g})\neq 1).
\end{equation}Let us prove that \begin{equation}\Gamma\models \forall \bm{y} \exists \bm{z} \bigvee_{k=1}^N(\Sigma_k(\bm{x},\bm{y},\bm{z},\bm{g})=1 \ \wedge \ \Psi_k(\bm{x},\bm{y},\bm{z},\bm{g})\neq 1).
\end{equation}
Let $\bm{y}$ be a tuple of elements of $\Gamma$. By (1), for every integer $n$, there exists a tuple $\bm{z}_n$ of elements of $G$ such that
\begin{equation}G\models \bigvee_{k=1}^N(\Sigma_k(\bm{x},\varphi_n(\bm{y}),\bm{z}_n,\bm{g})=1 \ \wedge \ \Psi_k(\bm{x},\varphi_n(\bm{y}),\bm{z}_n,\bm{g})\neq 1).
\end{equation}
By Corollary \ref{sacerdote2}, there exists a tuple $\bm{z}$ of elements of $\Gamma$ such that \begin{equation}\Gamma\models \bigvee_{k=1}^N(\Sigma_k(\bm{x},\bm{y},\bm{z},\bm{g})=1 \ \wedge \ \Psi_k(\bm{x},\bm{y},\bm{z},\bm{g})\neq 1).
\end{equation}\end{proofnash}

We now prove Proposition \ref{sacerdote1}, that is the generalized Sacerdote's lemma. 

\medskip

\begin{proofnashsac}Let $\bm{\gamma}\in\Gamma^p$. Suppose that, for every integer $n$, there exists $\bm{g}_n\in G^q$ such that \[\Sigma(\varphi_n(\bm{\gamma}),\bm{g}_n,\bm{g})=1 \ \wedge \ \Psi(\varphi_n(\bm{\gamma}),\bm{g}_n,\bm{g})\neq 1.\]
Let $\Gamma_{\Sigma}=\langle \Gamma , \bm{y} \ \vert \ \Sigma(\bm{\gamma},\bm{y},\bm{g})=1\rangle$. Let $i$ denote the natural morphism from $\Gamma$ to $\Gamma_{\Sigma}$. By hypothesis, for every $n$, there exists a homomorphism $\widehat{\varphi}_n:\Gamma_{\Sigma}\rightarrow G$ mapping $\bm{y}$ to $\bm{g}_n$ such that $\widehat{\varphi}_n\circ i=\varphi_n$. Note that the test sequence $(\varphi_n)_{n\in\mathbb{N}}$ is discriminating by Corollary \ref{disc}. As a consequence, the homomorphism $i$ is injective. From now on, we omit mentioning the morphism $i$. 

We shall construct a retraction $r : \Gamma_{\Sigma}\twoheadrightarrow \Gamma$ that does not kill any component of the tuple $\Psi(\bm{\gamma},\bm{y},\bm{g})$. Let $L=\Gamma_{\Sigma}/\ker((\widehat{\varphi}_n)_{n\in\mathbb{N}})$ and let $\pi : \Gamma_{\Sigma} \twoheadrightarrow L$ be the associated epimorphism. As a $G$-limit group, $L$ is equationally noetherian (see \cite{RW14} Corollary 6.13). It follows that there exists a unique homomorphism $\rho_{n} : L\rightarrow G$ such that $\widehat{\varphi}_n=\rho_{n}\circ \pi$. 

Since the sequence $(\varphi_n)_{n\in\mathbb{N}}$ is discriminating, we have $\widehat{\varphi}_n(x)=\rho_n(\pi(x))\neq 1$ for every $x\in\Gamma$ and every $n$ large enough. In addition, by construction, the morphism $\widehat{\varphi}_n$ does not kill any component of $\Psi(\varphi_n(\bm{\gamma}),\bm{g}_n,\bm{g})$. Thus, the homomorphism $\pi : \Gamma \rightarrow L$ is injective and does not kill any component of $\Psi(\varphi_n(\bm{\gamma}),\bm{g}_n,\bm{g})$. In the sequel, we identify $\Gamma$ and $\pi(\Gamma)$.

In order to construct $r$, we will construct a discriminating sequence of retractions $(r_n : L \twoheadrightarrow \Gamma)_{n\in\mathbb{N}}$. Then, we will conclude by taking $r:=r_n\circ \pi$ for $n$ sufficiently large. 


Let $(X,d)$ be a Cayley graph of $G$. Let us consider a Stallings splitting $\Lambda$ of $L$ relative to $\Gamma$, and let $U$ be the one-ended factor that contains $\Gamma$. Let $S$ be a generating set of $U$. Recall that $\mathrm{Aut}_{\Gamma}(U)$ is the subgroup of $\mathrm{Aut}(U)$ consisting of all automorphisms $\sigma$ satisfying the following two conditions:
\begin{enumerate}
\item $\sigma_{\vert \Gamma}=\mathrm{id}_{\vert \Gamma}$;
\item for every finite subgroup $F$ of $U$, there exists an element $u\in U$ such that $\sigma_{\vert F}=\mathrm{ad}(u)_{\vert F}$.
\end{enumerate}
Recall that a homomorphism $\varphi : U \rightarrow G$ is said to be \emph{short} if its length $\ell(\varphi):=\max_{s\in S}d(1,\varphi(s))$ is minimal among the lengths of homomorphisms in the orbit of $\varphi$ under the action of $\mathrm{Aut}_{\Gamma}(U)\times \mathrm{Inn}(G)$. Since $\vert\vert t_n\vert\vert$ goes to infinity, there exists a sequence of automorphisms $(\sigma_{n})_{n\in\mathbb{N}}\in\mathrm{Aut}_{\Gamma}(U)^{\mathbb{N}}$ and a sequence of elements $(x_n)_{n\in\mathbb{N}}\in G^n$ such that the homomorphisms $\mathrm{ad}(x_n)\circ \rho_{n}\circ \sigma_{n}$ are short, pairwise distinct, and such that the sequence $(\mathrm{ad}(x_n)\circ \rho_{n}\circ \sigma_{n})_{n\in\mathbb{N}}$ is stable (see paragraph \ref{stabletreenash}), up to extracting a subsequence. Since $\rho_n$ coincides with the identity on $G$, we have $\mathrm{ad}(x_n)\circ \rho_{n}\circ \sigma_{n}= \rho_{n}\circ \mathrm{ad}(x_n)\circ \sigma_{n}$ and, up to replacing $\sigma_n$ by $\mathrm{ad}(x_n)\circ \sigma_{n}$, we can forget the postconjugation by $x_n$ and assume that ${\sigma_n}_{\vert \Gamma}$ is a conjugation by an element of $G$.

We claim that $\sigma_n$ extends to an automorphism of $L$, still denoted by $\sigma_n$. By the second condition above, $\sigma_n$ is a conjugacy on finite subgroups of $U$. We proceed by induction on the number of edges of $\Lambda$ (the Stallings splitting of $L$ relative to $\Gamma$ used previously in order to define $U$). It is enough to prove the claim in the case where $\Lambda$ has only one edge. 

If $L=U\ast_C B$ with ${\sigma_n}_{\vert C}=\mathrm{ad}(u)$, one defines $\sigma_n : L \rightarrow G$ by ${\sigma_n}_{\vert U}=\sigma_n$ and ${\sigma_n}_{\vert B}=\mathrm{ad}(u)$. 

If $L=U\ast_C=\langle U,t \ \vert \ tct^{-1}=\alpha(c), \forall c\in C\rangle$ with ${\sigma_n}_{\vert C}=\mathrm{ad}(u_1)$ and ${\sigma_n}_{\vert \alpha{(C)}}=\mathrm{ad}(u_2)$, one defines $\sigma_n : L \rightarrow G$ by ${\sigma_n}_{\vert U}=\sigma_n$ and $\sigma_n(t)=u_2^{-1}tu_1$.

In order to complete the proof of the generalized Sacerdote's lemma, we will use the following result.

\begin{lemme}\label{abovelemma}We keep the same notations as above. Let $T$ be the limit tree of the sequence of metric spaces $(X,d/\ell({\rho_n\circ \sigma_n}))_{n\in\mathbb{N}}$. The following dichotomy holds:
\begin{itemize}
\item[$\bullet$]either $\Gamma$ does not fix a point of $T$, in which case there exists a discriminating sequence of retractions $(r_n : L \twoheadrightarrow\Gamma)_{n\in\mathbb{N}}$,
\item[$\bullet$]or $\Gamma$ is elliptic, and there exist a proper quotient $L_1$ of $L$, an embedding $\Gamma\hookrightarrow L_1$ allowing us to identify $\Gamma$ with a subgroup of $L_1$, and two discriminating sequences $(\rho_n^1 : L_1\rightarrow G)_{n\in\mathbb{N}}$ and $(\theta_n^1 : L\twoheadrightarrow L_1)_{n\in\mathbb{N}}$ such that the following three conditions are satisfied:
\begin{enumerate}
\item $\rho_n\circ \sigma_n=\rho_n^1\circ \theta_n^1$;
\item $\rho_n^1$ coincides with $\rho_n$ on $\Gamma$; in particular, $({\rho_n^1}_{\vert \Gamma} : \Gamma \rightarrow G)_{n\in\mathbb{N}}$ is a test sequence.
\item There exists an element $g_n\in G$ such that $\sigma_n$ and $\theta_n^1$ coincides with $\mathrm{ad}(g_n)$ on $\Gamma$.
\end{enumerate}
\end{itemize}
\end{lemme}

Before proving this lemma, we will use it to conclude the proof of Proposition \ref{sacerdote1}. If $\Gamma$ does not fix a point of $T$, we are done. If $\Gamma$ fixes a point of $T$, by iterating Lemma \ref{abovelemma}, we get a sequence of proper quotients \[L_0=L \twoheadrightarrow L_1\twoheadrightarrow \cdots L_i\twoheadrightarrow \cdots\]such that, for every integer $i\geq 1$, there exist two discriminating sequences of morphisms $(\rho_{n}^i : L_i \rightarrow G)_{n\in\mathbb{N}}$ and $(\theta_{n}^i: L_{i-1} \twoheadrightarrow L_{i})_{n\in\mathbb{N}}$ such that $\rho_{n}^{i-1}\circ\sigma_n=\rho_{n}^{i}\circ \theta_{n}^i$, $\rho_n^i$ coincides with $\rho_n^{i-1}$ on $\Gamma$, and there exists an element $g_n\in G$ such that $\sigma_n$ and $\theta_n^i$ coincides with $\mathrm{ad}(g_n)$ on $\Gamma$.

By the descending chain condition \ref{chainesela}, the iteration eventually terminates. Let $L_k$ be the last quotient of the series. By Lemma \ref{abovelemma}, there exists a discriminating sequence of retractions $(r_n : L_k \twoheadrightarrow\Gamma)_{n\in\mathbb{N}}$. For every finite set $F\subset L$, one can find some integers $n_1,\ldots ,n_{k}$ such that the morphism $r_{n_{k}}\circ \theta_{n_k}^k\circ \cdots\circ \theta_{n_1}^1 : L\twoheadrightarrow\Gamma$ is injective on $F$. Moreover, since every $\theta_{n_i}$ is a conjugation on $\Gamma$ by an element of $G$, there exists an element $v_n\in G$ such that $\mathrm{ad}(v_n)\circ r_{n_{k}}\circ \theta_{n_k}^k\circ \cdots\circ \theta_{n_1}^1 : L\twoheadrightarrow\Gamma$ is a retraction. This concludes the proof of Lemma \ref{sacerdote1}.\end{proofnashsac}

It remains to prove Lemma \ref{abovelemma}. 

\medskip

\begin{proofnashlemme}Recall that $U$ denotes the one-ended factor of $L$ relative to $\Gamma$. We distinguish two cases.

\vspace{2mm}

\textbf{First case.} Suppose that $\Gamma$ fixes a point of $T$. Let us prove that the stable kernel of the sequence $(\rho_n\circ \sigma_n)_{n\in\mathbb{N}}$ is non-trivial. Assume towards a contradiction that the stable kernel is trivial. Then, by Theorem 1.16 of \cite{RW14}, the action of $(U,\Gamma)$ on the limit tree $T$ has the following properties:
\begin{itemize}
\item[$\bullet$]the pointwise stabilizer of any non-degenerate tripod is finite;
\item[$\bullet$]the pointwise stabilizer of any non-degenerate arc is finitely generated and finite-by-abelian;
\item[$\bullet$]the pointwise stabilizer of any unstable arc is finite.
\end{itemize}
In particular, the tree $T$ satisfies the ascending chain condition of Theorem \ref{guirardel1} since any ascending sequence of finitely generated and finite-by-abelian subgroups of a hyperbolic group stabilizes.

Then, it follows from Theorem \ref{guirardel1} that either $(U,\Gamma)$ splits over the stabilizer of an unstable arc, or over the stabilizer of an infinite tripod, or $T$ has a decomposition into a graph of actions. Since $U$ is one-ended relative to $\Gamma$, and since the stabilizer of an unstable arc or of an infinite tripod is finite, it follows that $T$ has a decomposition into a graph of actions.

Now, it follows from Theorem \ref{nashpaulinnash} that there exists a sequence of automorphisms $(\alpha_n)_{n\in\mathbb{N}}\in\mathrm{Aut}_{\Gamma}(U)^{\mathbb{N}}$ such that $(\rho_n\circ\sigma_n)\circ \alpha_n$ is shorter than $\rho_n\circ \sigma_n$ for $n$ large enough. This is a contradiction since the morphisms $\rho_n\circ \sigma_n$ are assumed to be short. Hence, the stable kernel of the sequence $(\rho_n\circ \sigma_n)_{n\in\mathbb{N}}$ is non-trivial.

As in the proof of Lemma \ref{sacerdote1} above, since $\sigma_{n}$ coincides with an inner automorphism on each finite subgroup of $U$ (by definition of $\mathrm{Aut}_{\Gamma}(U)$), it extends to an automorphism of $L$, still denoted by $\sigma_n$. Let $L_1:=L/\ker(({\rho_{n}}\circ \sigma_{n})_{n\in\mathbb{N}})$ and let $\pi_1 : L \twoheadrightarrow L_1$ be the corresponding epimorphism. Observe that $\pi_1$ is injective on $\Gamma$, allowing to identify $\Gamma$ with a subgroup of $L_1$. 

As a $G$-limit group, $L_1$ is equationally noetherian (see \cite{RW14} Corollary 6.13). It follows that, for every integer $n$, there exists a unique homomorphism $\tau_n^1 : L_1\rightarrow G$ such that $\rho_n\circ \sigma_n=\tau_n^1\circ\pi_1$. There exists an element $g_n\in G$ such that $\sigma_n$ coincides with $\mathrm{ad}(g_n)$ on $\Gamma$. Hence, since $\rho_n$ coincides with the identity on $G$, we can write $\rho_n\circ \sigma_n=\mathrm{ad}(g_n)\circ \rho_n^1\circ\pi_1$ in such a way that $\rho_n^1:=\mathrm{ad}(g_n^{-1})\circ \tau_n^1$ coincides with the identity on $G$. For every $n$, let $\theta_{n}^1= \pi_1\circ (\sigma_n)^{-1}$, so that $\rho_n=\mathrm{ad}(g_n)\circ \rho_n^1\circ \theta_n^1$. The sequence $(\theta_n^1 : L \rightarrow L_1)_{n\in\mathbb{N}}$ is therefore discriminating, and every homomorphism $\theta_n^1$ coincides with $\mathrm{ad}(g_n^{-1})$ on $\Gamma$.

\vspace{2mm}

\emph{Second case.} Suppose that $\Gamma$ is not elliptic in $T$. We will construct a discriminating sequence of retractions $(r_n : L \rightarrow \Gamma)_{n\in\mathbb{N}}$. Let $T_{\Gamma}\subset T$ be the minimal invariant subtree of $\Gamma$. By Lemma \ref{lemme2}, we may assume up to rescaling that $T_{\Gamma}$ is isometric to the Bass-Serre tree of the splitting $\Gamma=G\ast_C$. 

Let $\sim$ be the relation on $T$ defined by $x\sim y$ if $[x,y]\cap u T_{\Gamma}$ contains at most one point, for every element $u\in U$. Note that $\sim$ is an equivalence relation. Let $(Y_j)_{j\in J}$ denote the equivalence classes that are not reduced to a point. Each $Y_j$ is a subtree of $T$. Let us prove that $(Y_j)_{j\in J}\cup \lbrace uT_{\Gamma} \ \vert \ u\in U/\Gamma\rbrace$ is a transverse covering of $T$, in the sense of Definition \ref{transverse}.

\begin{itemize}
\item[$\bullet$]\emph{Transverse intersection.} For every $i\neq j$, the intersection $Y_i\cap Y_j$ is clearly empty. For every $i$ and $u\in U$, $Y_i\cap uT_{\Gamma}$ contains at most one point by definition. For every $u,u'\in U$ such that $u'u^{-1}\notin \Gamma$, $\vert uT_{\Gamma}\cap u'T_{\Gamma}\vert \leq 1$ thanks to Lemma \ref{lemme22}.
\item[$\bullet$]\emph{Finiteness condition.} Let $x$ and $y$ be two points of $T$. By Lemma \ref{lemme22}, there exists a constant $\varepsilon > 0$ such that, for every $u\in U$, if the intersection $[x,y]\cap uT_{\Gamma}$ is non-degenerate, the length of $[x,y]\cap uT_{\Gamma}$ is bounded from below by $\varepsilon$. Consequently, the arc $[x,y]$ is covered by at most $\lfloor d(x,y)/\varepsilon\rfloor$ translates of $T_{\Gamma}$ and at most $\lfloor d(x,y)/\varepsilon\rfloor+1$ distinct subtrees $Y_j$.
\end{itemize}

Hence, the collection $(Y_j)_{j\in J}\cup \lbrace uT_{\Gamma} \ \vert \ u\in U\rbrace$ is a transverse covering of $T$. One can construct what Guirardel calls the skeleton of this transverse covering (see Definition \ref{squelette}), denoted by $T_c$. Since the action of $U$ on $T$ is minimal, the same holds for the action of $U$ on $T_c$, according to Lemma 4.9 of \cite{Gui04}. The question is now to understand the decomposition $\Delta_c=T_c/U$ of $U$ as a graph of groups.

We begin with a description of the stabilizer in $U$ of an edge $e$ of $T_{\Gamma}$. Let $u$ be an element of $U$ that fixes $e$. Then $e$ is contained in $T_{\Gamma}\cap u T_{\Gamma}$, so $u$ belongs to $\Gamma$, thanks to Lemma \ref{lemme22}. It follows that $u$ belongs to $C$, because the stabilizer of $e$ in $\Gamma$ is equal to $C$ (indeed, recall that $T_{\Gamma}$ is isometric to the Bass-Serre tree of the splitting $\Gamma=G\ast_C$, by Lemma \ref{lemme2}). Thus, the stabilizer of $e$ in $U$ is equal to $C$.

We now prove that if one of the subtrees of the covering other than $T_{\Gamma}$ intersects $T_{\Gamma}$ in a point, then this point is necessarily one of the extremities of a translate of the edge $e\in T_{\Gamma}$. Assume towards a contradiction that $Y_j$ or $uT_{\Gamma}$ with $u\notin \Gamma$ intersects $T_{\Gamma}$ in a point $x$ that is not one of the extremities of $e$. Then, $T_c$ contains an edge $\varepsilon=(x,T_{\Gamma})$ whose stabilizer is $\mathrm{Stab}(x) \cap \Gamma$ (where $\mathrm{Stab}(x)$ denotes the stabilizer of $x$ in $U$), which is equal to $C$ by the previous paragraph. So the splitting $\Delta_c$ of $U$ is a non-trivial splitting over the finite subgroup $C$, relative to $\Gamma$. This is impossible since $U$ is one-ended relative to $\Gamma$, by definition of $U$. Hence, if $Y_j\cap T_{\Gamma}=\lbrace x\rbrace$ or $uT_{\Gamma}\cap T_{\Gamma}=\lbrace x\rbrace$ with $u\notin \Gamma$, then the point $x$ is one of the extremities of $e$ in $T_{\Gamma}$. As a consequence, $\mathrm{Stab}(x)$ is a conjugate of $G$ in $\Gamma$, and every edge adjacent to $T_{\Gamma}$ in $T_c$ is of the form $(\gamma x,T_{\Gamma})=\gamma\varepsilon$ with $\varepsilon=(x,T_{\Gamma})$. 

Therefore, $\varepsilon$ is the only edge adjacent to $T_{\Gamma}$ in the quotient graph $\Delta_c$. Its stabilizer is $G$. By collapsing all edges of $\Delta_c$ except $\varepsilon$, one gets a splitting of $U$ of the following form: $U=\Gamma\ast_G U'$ for some subgroup $U'\subset U$. This splitting can be written as \[U=U'\ast_C=\langle U', t \ \vert \ [t,c]=1, \ \forall c\in C\rangle.\] 

Since every finite subgroup of $U$ is conjugate to a finite subgroup of $U'$, the group $L$ splits as $L=U''\ast_C=\langle U'', t \ \vert \ [t,c]=1, \ \forall c\in C\rangle$ with $G\subset U'\subset U''$. One now defines a retraction $r_n : L \rightarrow \Gamma$ by $r_n(t)=t$ and ${r_n}_{\vert U''}={\rho_n}_{\vert U''}$, well-defined since $\rho_n$ coincides with the identity on $G$. 

To conclude, let us prove that the sequence $(r_n)_{n\in\mathbb{N}}$ is discriminating. Let $\ell$ be a non-trivial element of $L$. This element can be written in reduced form, with respect to the HNN extension $L=U''\ast_C$, as $\ell=u_0t^{\varepsilon_1}u_1t^{\varepsilon_2}u_2\cdots t^{\varepsilon_p}u_{p+1}$, with $u_i\in U''$. For every $i$, if $\varepsilon_i=-\varepsilon_{i+1}$, then $u_i\notin C$. Thus $r_n(u_i)\notin C$ for every $n$ large enough (otherwise, up to extracting a subsequence, one can assume that $r_n(u_i)=u_n(c)$ for every $n$, so $u_i=c$ (since $({r_n}_{\vert U''})$ is discriminating), which is impossible). Hence, for every $n$ large enough, $r_n(\ell)=\rho_n^k(u_0)t^{\varepsilon_1}\rho_n^k(u_1)t^{\varepsilon_2}\rho_n^k(u_2)\cdots t^{\varepsilon_p}\rho_n^k(u_{p+1})$ is non-trivial. \end{proofnashlemme}

\section{Proof of Theorem \ref{legal2te}}\label{a}

We will prove the following result.

\begin{te}[see Theorem \ref{legal2te}]\label{jjjjjjj}Let $G$ be a hyperbolic group that splits as $A\ast_C B$ or $A\ast_C$ over a finite subgroup $C$ whose normalizer $N$ is infinite virtually cyclic and non-elliptic in the splitting. Let $K_G$ be the maximal order of a finite subgroup of $G$. Let $N'$ be a virtually cyclic group such that $K_{N'}\leq K_G$, and let $\iota : N\hookrightarrow N'$ be a $K_G$-nice embedding (see Definition \ref{special0}). Let us define $G'$ by\[G'=G\ast_N N'=\langle G, N' \ \vert \ g=\iota(g), \ \forall g\in N\rangle.\]The following two assertions are equivalent.
\begin{enumerate}
\item The group $G'$ is a legal small extension of $G$ in the sense of Definition \ref{legal2}, i.e.\ there exists a $K_G$-nice embedding $\iota' : N'\hookrightarrow N$.
\item $\mathrm{Th}_{\forall\exists}(G')=\mathrm{Th}_{\forall\exists}(G)$.
\end{enumerate}\end{te}

\begin{rque}As observed in the introduction, $G$ is not $\exists\forall$-elementarily embedded into $G'$ in general.\end{rque}

First, we prove the easy direction.

\begin{prop}We keep the same notations as in Theorem \ref{jjjjjjj}. If $\mathrm{Th}_{\forall\exists}(G')=\mathrm{Th}_{\forall\exists}(G)$, then $G'$ is a legal small extension of $G$.
\end{prop}

\begin{proof}According to the implication $(1)\Rightarrow (4)$ of Theorem \ref{principal}, there exists a strongly special morphism $\varphi' : G'\rightarrow G\subset G'$. Since $\varphi'$ maps non-conjugate finite subgroups to non-conjugate finite subgroups, there exists an integer $n\geq 1$ such that $\varphi'^n$ maps $C$ to $g'Cg'^{-1}$ for some $g'\in G'$. Hence we have $\varphi'^n(N')\subset N_{G}(g'Cg'^{-1})=g'Ng'^{-1}$, since $\varphi'(G')$ is contained in $G$. Let us define $\iota'$ by $\iota':=(\mathrm{ad}(g'^{-1})\circ \varphi'^n)_{\vert N'}$. This morphism is a $K_G$-nice embedding, because $\varphi'$ is strongly special.\end{proof}

In the rest of this section, we prove that the converse also holds.

\begin{te}We keep the same notations as in Theorem \ref{jjjjjjj}. If $G'$ is a legal small extension of $G$, then $\mathrm{Th}_{\forall\exists}(G')=\mathrm{Th}_{\forall\exists}(G)$.
\end{te}

Recall that an infinite virtually cyclic group $N$ can be written as an extension of exactly one of the following two forms:\[1\rightarrow C\rightarrow N\rightarrow \mathbb{Z}\rightarrow 1 \ \ \ \text{or} \ \ \ 1\rightarrow C\rightarrow N\rightarrow D_{\infty}\rightarrow 1,\] where $C$ is finite and $D_{\infty}=\mathbb{Z}/2\mathbb{Z}\ast\mathbb{Z}/2\mathbb{Z}$ denotes the infinite dihedral group. In the first case, $N$ splits as $N=C\rtimes\langle t\rangle$, where $t$ denotes an element of infinite order. In the second case, $N$ splits as an amalgamated free product $\langle C,a\rangle\ast_C\langle C,b\rangle$ where $a$ and $b$ have order $2$ in $N/C$. We choose such elements $a$ and $b$ and we define $t$ by $t=ba$. Note that the image of $t$ in $N/C$ generates $N^{+}/C$, where $N^{+}$ denotes the kernel of the epimorphism $N\twoheadrightarrow D_{\infty}\twoheadrightarrow\mathbb{Z}/2\mathbb{Z}$. In other words, $N^{+}\simeq C\rtimes \langle t\rangle$.

In the sequel, we say that two elements $g,g'\in N$ are equal modulo $C$, and we write $g'=g\mod C$, if $g^{-1}g'$ belongs to $C$.

Recall that for every integer $p$, we denote by $D_p(N)$ the definable subset $\lbrace n^{p} \ \vert \ n\in N\rbrace$.

\begin{lemme}\label{claim}Let $N$ be a virtually cyclic group. Let $C$ be the maximal finite normal subgroup of $N$, and let $m$ be the order of $\mathrm{ad}(t)_{\vert C}$ in $\mathrm{Aut}(C)$. Then $D_{2m\vert C\vert}(N)=\langle t^{2m\vert C\vert}\rangle$.\end{lemme}

\begin{rque}In particular, $D_{2m\vert C\vert}(N)$ is a subgroup of $N$. Moreover, it is central in $N$.\end{rque}

\begin{proof}Let $g$ be an element of $N$. The element $g^2$ can be written as $ct^{2r}$ for some $c\in C$ and $r\in \mathbb{Z}$, so $g^{2m}=c't^{2mr}$ for some $c'\in C$. By definition of $m$, $t^m$ commutes with $c'$, so $(g^{2m})^{\vert C\vert }=(c')^{\vert C\vert}t^{2m\vert C\vert}=t^{2m\vert C\vert }$. As a consequence, $D_{2m\vert C\vert}(N)$ is contained in $\langle t^{2m\vert C\vert}\rangle$. The other inclusion is obvious.\end{proof}

\begin{rque}Note that the action of $\mathrm{Aut}(C)$ on $C\setminus \lbrace 1\rbrace$ gives an embedding from $\mathrm{Aut}(C)$ into the symmetric group $\mathfrak{S}_{\vert C\vert -1}$. It follows that $\vert C\vert \mathrm{Aut}(C)$ divides $\vert C\vert !$. Hence, $2m\vert C\vert$ divides $2 \vert C\vert !$.\end{rque}


In the sequel, $G$ denotes a hyperbolic group and $K$ denotes the maximal order of a finite subgroup of $G$. For every virtually cyclic infinite subgroup $N$ of $G$, we define the subgroup $D(N)$ of $N$ by $D(N):=D_{2K!}(N)$. The following result in an immediate consequence of Lemma \ref{claim} (see the remark above).

\begin{co}\label{introintrointrointro}If $C$ is a finite subgroup of $G$ whose normalizer $N$ is virtually cyclic infinite, then $D(N)=\langle t^{2K!}\rangle$ (for any element $t$ chosen as above).\end{co}



Recall that the normalizer of a finite edge group in a splitting is an infinite virtually cyclic group if and only if one of the situations described below arises.

\begin{lemme}\label{diédral11}Let $G$ be a group. Suppose that $G$ splits as an amalgamated free product $G=A\ast_CB$ over a finite group $C$, and that $N_G(C)$ is not contained in a conjugate of $A$ or $B$. Then $N_G(C)$ is virtually cyclic if and only if $C$ has index 2 in $N_A(C)$ and in $N_B(C)$. In this case, $N_G(C)$ is of dihedral type, equal to $N_A(C)\ast_C N_B(C)$.\end{lemme}

\begin{lemme}\label{diédral22}Let $G$ be a group. Suppose that $G$ splits as an HNN extension $G=A\ast_C$ over a finite group $C$. Let $C_1$ and $C_2$ denote the two copies of $C$ in $A$ and $t$ be the stable letter associated with the HNN extension. Suppose that $N_G(C)$ is not contained in a conjugate of $A$. Then $N_G(C)$ is virtually cyclic if and only if one of the following two cases holds.
\begin{enumerate}
\item If $C_1$ and $C_2$ are conjugate in $A$ and $N_A(C_1)=C_1$, then the normalizer $N_G(C_1)$ is of cyclic type, equal to $C_1\rtimes \langle at\rangle$, where $a$ denotes an element of $A$ such that $aC_2a^{-1}=C_1$.
\item If $C_1$ and $C_2=t^{-1}C_1t$ are non-conjugate in $G$ and $C_1$ has index 2 in $N_A(C_1)$ and $N_{tAt^{-1}}(C_1)$, then the normalizer $N_G(C_1)$ is of dihedral type, equal to \[N_A(C_1)\ast_{C_1} N_{tAt^{-1}}(C_1).\]
\end{enumerate}\end{lemme}

\subsection{Small test sequences}

\begin{de}[Twist]Let $G$ be a group. Suppose that $G$ splits as $A\ast_C B$ or $A\ast_C$ over a finite subgroup $C$ whose normalizer $N$ is virtually cyclic infinite and non-elliptic in the splitting. Let $\delta$ be an element of $D(N)$.
\begin{itemize}
\item[$\bullet$]If $G=A\ast_C=\langle A,t \ \vert \ tct^{-1}=\theta(c), \forall c\in C\rangle$ where $\theta\in\mathrm{Aut}(C)$, the \emph{twist} $\tau_{\delta}$ is the endomorphism of $G$ that coincides with the identity on $A$ and that maps the stable letter $t$ to $t \delta$.
\item[$\bullet$]If $G=A\ast_C B$, the twist $\tau_{\delta}$ is the endomorphism of $G$ that coincides with the identity on $A$, and that coincides with $\mathrm{ad}(\delta)$ on $B$.
\end{itemize}
\end{de}

\begin{rque}\label{inj nashville}Note that $\tau_{\delta}$ is well-defined in both cases because $\delta$ centralizes $C$, as en element of $D(N)$. Moreover, in both cases, $\tau_{\delta}$ is a monomorphism: in the first case, it suffices to observe that $t\delta$ has infinite order, which is true since $t$ has infinite order and $\delta$ is a power of $t^{2K!}$; in the second case, the injectivity is automatic thanks to Britton's lemma. In addition, in the second case, note that $\tau_{\delta}$ maps $t=ba$ to $\delta b\delta^{-1}a=t^rbt^{-r}a=(ba)^rb(ba)^{-r}a=(ba)^{2r+1}=t^{2r+1}=t\delta^2$ for some multiple $r$ of $2K!$.\end{rque}


By analogy with test sequences defined in the previous section, we introduce below the notion of a \emph{small test sequence}, designed for legal small extensions.

\begin{lemme}Let $N$ and $N'$ be two virtually cyclic infinite groups. Let $C$ and $C'$ be the maximal normal finite subgroups of $N$ and $N'$ respectively. Suppose that there exist two embeddings $\iota : N \hookrightarrow N'$ and $\iota ' : N'\hookrightarrow N$. Then $\iota(C)=C'$ and $\iota'(C')=C$.\end{lemme}

\begin{proof}
Note that $N$ and $N'$ are both of cyclic type or of dihedral type, since a virtually cyclic group of dihedral type does not embed into a virtually cyclic group of cyclic type.

\emph{First case.} Suppose that $N$ and $N'$ are of cyclic type. Then $N=C\rtimes\mathbb{Z}$ and $N'=C'\rtimes \mathbb{Z}$. It follows that $\iota(C)\subset C'$ and $\iota'(C')\subset C$. Hence $C$ and $C'$ have the same cardinality, and we have $\iota(C)= C'$ and $\iota'(C')= C$.

\emph{Second case.} Suppose that $N$ and $N'$ are of dihedral type. There exist two elements $a,b\in N$ such that $N=\langle C,a,b \ \vert \ a^2\in C, b^2\in C, aC=Ca, bC=Cb\rangle$. Note that all finite subgroups of $N$ that are not contained in $C$ are of the form $n\langle C_1,a\rangle n^{-1}$ or $n \langle C_1 , b \rangle n^{-1}$ with $C_1\subset C$ and $n\in N$. In addition, note that the normalizers of $a$ and $b$ in $N$ are finite. Thus, the normalizer of the finite groups $\langle C_1,a\rangle$ and $\langle C_1,b\rangle$ are finite. Then, observe that the normalizer of $\iota'(C')$ is equal to $\iota'(N')$, which is infinite since $\iota'$ is injective and $N'$ is infinite. It follows that $\iota'(C')$ is contained in $C$. Likewise, $\iota(C)$ is contained in $C'$. Hence $C$ and $C'$ have the same cardinality, and we have $\iota(C)= C'$ and $\iota'(C')= C$.\end{proof} 

Let $G$ be a hyperbolic group that splits over a finite group $C$ whose normalizer $N$ is virtually cyclic infinite and non-elliptic in the splitting. Let \[\Gamma=G\ast_N N'=\langle G, N' \ \vert \ g=\iota(g), \ \forall g\in N\rangle\] be a legal small extension of $G$, where $N'$ is virtually cyclic and $\iota : N \hookrightarrow N'$ is $K$-nice. Let $C$ and $C'$ be the maximal normal finite subgroups of $N$ and $N'$ respectively. By the previous lemma, $\iota(C)$ is equal to $C'$. As a consequence, in $\Gamma$, we have $C=C'$. We make the following two observations.

\begin{enumerate}
\item The group $\Gamma$ splits as $A'\ast_C B'$ or $A'\ast_C$. The normalizer $N_{\Gamma}(C)$ of $C$ in $\Gamma$ is equal to $N'$. This group is virtually cyclic infinite and non-elliptic in the previous splitting of $\Gamma$ over $C$. 
\item The maximal order of a finite subgroup is the same for $\Gamma$ and $G$. Indeed, every finite subgroup $F$ of $\Gamma$ is contained in a conjugate of $G$ or in a conjugate of $N'$. In the second case, $F$ embeds into $N$ since there exists an embedding from $N'$ into $N$. Hence, in both cases, $F$ embeds into $G$. As a consequence, $K_{\Gamma}$ is equal to $K_G$ and there is no ambiguity about the notation $D(N')$
\end{enumerate}

Thanks to these two observations, the Dehn twist $\tau_{\delta}$ is well-defined, as an endomorphism of $\Gamma$, for any element $\delta\in D(N')$. We are now ready to define small test sequences.

\begin{de}[Small test sequence]\label{testsuite}Let $G$ be a hyperbolic group. Suppose that $G$ splits over a finite group $C$ whose normalizer $N$ is virtually cyclic infinite and non-elliptic in the splitting. Let \[\Gamma=G\ast_N N'=\langle G, N' \ \vert \ g=\iota(g), \ \forall g\in N\rangle\] be a legal small extension of $G$, where $N'$ is virtually cyclic and $\iota : N \hookrightarrow N'$ is $K$-nice. A sequence of homomorphisms $(\varphi_n : \Gamma \rightarrow G)_{n\in\mathbb{N}}$ is called a small test sequence if there exist a strictly increasing sequence of prime numbers $(p_n)_{n\in\mathbb{N}}$ and a sequence $(\delta_n)_{n\in\mathbb{N}}\in D(N')^{\mathbb{N}}$ such that $\varphi_n=\tau_{\delta_n}$ (viewed as an endomorphism of $\Gamma$) and $[N:\varphi_n(N')]=p_n$, for every integer $n$.\end{de}


The following lemma shows that small test sequences exist as soon as $\iota : N \hookrightarrow N'$ is not surjective.

\begin{lemme}\label{guaranteed nash}Let $G$ be a hyperbolic group. Suppose that $G$ splits over a finite group $C$ whose normalizer $N=N_G(C)$ is virtually cyclic infinite and non-elliptic in the splitting. Let $\Gamma=G\ast_N N'$ be a legal small extension of $G$. Suppose that $N$ is a strict subgroup of $N'$ in $\Gamma$. Then, there exists a small test sequence $(\varphi_n : \Gamma \rightarrow G)_{n\in\mathbb{N}}$.\end{lemme}


\begin{proof}By definiton of a legal small extension, there exist two $K$-nice embeddings $\iota : N \hookrightarrow N'$ and $\iota' : N'\hookrightarrow N$. Let $C$ and $C'$ be the maximal normal finite subgroups of $N$ and $N'$ respectively. In $\Gamma$, we have the identification $\iota(n)=n$ for every $n\in N$. In particular, $C$ and $C'$ are identified. In the sequel, we do not mention $\iota$ anymore. We distinguish two cases.

\smallskip

\emph{First case.} Suppose that $N$ is virtually cyclic of cyclic type. Since $\iota'$ is special, there exists an integer $m\geq 1$ such that $\iota'^m$ coincides with the identity on $C$ and induces the identity of $N'/D(N')$. Up to replacing $\iota'$ by $\iota'^{m}$, one can assume without loss of generality that $\iota'$ coincides with the identity on $C$ and induces the identity of $N'/D(N')$. 

Let $t$ and $z$ denote two elements such that such that $N=N_G(C)=C\rtimes \langle t\rangle$ and $N'=N_{\Gamma}(C)=C\rtimes \langle z\rangle$. Recall that $D(N')=\langle z^{2K!}\rangle$, by Corollary \ref{introintrointrointro}. Since $\iota'$ induces the identity of $N'/D(N')$, we have $\iota'(z)=z^{1+ 2K!q}$ for some integer $q$. Note that $q$ is non-zero because $N$ is a proper subgroup of $N'$ by assumption. Let $k$ and $\ell$ denote two integers such that $t=z^{\ell} \mod C$ and $\iota'(z)=t^k\mod C$. It follows that $k\ell=1+2K!q$. In particular, $\gcd(k,2K!)=1$. By the Dirichlet prime number theorem, there exists a strictly increasing sequence of integers $(\lambda_n)_{n\in\mathbb{N}}$ such that $p_n:=k+2K!\lambda_n$ is prime for every integer $n$. Let $\delta=t^{2K!}$, and let us define $\iota'_n : N' \rightarrow N$ by $\iota'_n(z)=\iota'(z)\delta^{\lambda_n}$ and $\iota'_n(c)=\iota'(c)=c$ for every $c\in C$. This homomorphism is well-defined: if $zcz^{-1}=c'$, then $\iota'_n(zcz^{-1})=\iota'(z)\delta^{\lambda_n}\iota'(c)(\iota'(z)\delta^{\lambda_n})^{-1}=\iota'(z)\iota'(c)\iota'(z)^{-1}=\iota'(czc^{-1})=\iota'(c')=\iota'_n(c')$, because $\delta$ belongs to the center $Z(N)$ of $N$. An easy calculation gives $\iota'_n(z)=t^{p_n} \mod C$. As a consequence $p_n=[N:\iota'_n(N')]$.

Last, by considering the decompositions $G=A\ast_C N$ and $\Gamma=A\ast_C N'$, one can define a homomorphism $\varphi_n : \Gamma \rightarrow G$ that coincides with the identity on $A$ and with $\iota'_n$ on $N'$ (well-defined since $\iota'_n$ coincides with the identity on $C$). Note that $\delta_n=z^{-1}\varphi_n(z)=z^{-1}t^{p_n}=z^{\ell p_n-1}=z^{\ell k+2K!\ell \lambda_n-1}=z^{2K!(\ell\lambda_n-q)}$ belongs to $D(N')$. Hence, the sequence $(\varphi_n)_{n\in\mathbb{N}}$ is a small test sequence in the sense of Definition \ref{testsuite}.

\smallskip

\emph{Second case.} Suppose that $N$ is virtually cyclic of dihedral type. It splits as $N=\langle C,a\rangle\ast_C \langle C,b\rangle$ with $a,b$ of order 2 modulo $C$. There exists two elements $a'$ and $b'$ of order 2 modulo $C$ such that $N'=\langle C,a'\rangle\ast_{C} \langle C,b'\rangle$. Up to exchanging $a'$ and $b'$, one can suppose without loss of generality that $a$ is a conjugate of $a'$ in $N'$ (modulo $C$) and that $b$ is a conjugate of $b'$ (modulo $C$); indeed, the inclusion of $N$ into $N'$ maps non-conjugate finite subgroups to non-conjugate finite subgroups ($\iota$ is $K$-nice).

Since $\iota'$ is $K$-nice, $\iota'(a')$ and $\iota'(b')$ are not conjugate modulo $C$. Hence, there exists an integer $j\in\lbrace 1,2\rbrace$ such that $(\iota'\circ\iota)^j$ maps $a'$ to a conjugate $gag^{-1}$ of $a$, with $g\in G$, and $b'$ to a conjugate of $b$. Up to replacing $\iota'$ by $\mathrm{ad}(g^{-1})\circ\iota'^{j}$, one can assume without loss of generality that $\iota'(a')=a$. Then, note that there exists an integer $m\geq 1$ such that $\iota'^m$ coincides with the identity on $C$ and induces the identity of $N'/D(N')$. Up to replacing $\iota'$ by $\iota'^{m}$, one can assume without loss of generality that $\iota'$ coincides with the identity on $C$ and induces the identity of $N'/D(N')$. 


Let us define $z$ by $z=b'a'\in N'\subset \Gamma$. Note that $\iota'(z)=z^{1+ 2K!q}$ for some integer $q$, because $\iota'$ induces the identity of $N'/D(N')$ and $D(N')=\langle z^{2K!}\rangle$, by Corollary \ref{introintrointrointro}. The integer $q$ is non-zero since $N$ is a strict subgroup of $N'$ by assumption. Let $k$ and $\ell$ denote two integers such that $t=z^{\ell} \mod C$ and $\iota'(z)=t^k\mod C$. It follows that $k\ell=1+2K!q$. In particular, $\gcd(k,2K!)=1$. By the Dirichlet prime number theorem, there exists a strictly increasing sequence of integers $(\lambda_n)_{n\in\mathbb{N}}$ such that $p_n:=k+2K!\lambda_n$ is prime for every integer $n$. Let $\delta=t^{K!}$, and let us define $\iota'_n : N' \rightarrow N$ by $\iota'_n=\iota'$ on $\langle C,a'\rangle$ and $\iota'_n=\mathrm{ad}(\delta^{\lambda_n})\circ\iota'$ on $\langle C,b'\rangle$. This homomorphism is well-defined since $\delta$ centralizes $C$ and $N'$ splits as $N'=\langle C,a'\rangle\ast_C\langle C,b'\rangle$. Since $\iota'(a')=a$ and $t=ba$, the following series of equalities holds (modulo $C$):
\begin{align*}
\iota'_n(z)&=\delta^{\lambda_n}\iota'(b')\delta^{-\lambda_n}\iota'(a')\\
&= \delta^{\lambda_n}\iota'(b')(ba)^{-\lambda_n K!}a\\
&=\delta^{\lambda_n}\iota'(b')(ab)^{\lambda_n K!}a\\
&=\delta^{\lambda_n}\iota'(b')a(ba)^{\lambda_n K}\\
&=\delta^{\lambda_n}\iota'(b'a')\delta^{\lambda_n}\\
 &= \iota'(z)\delta^{2\lambda_n}\\
 &= t^{p_n}.
\end{align*}
It follows that $p_n=[N:\iota'_n(N')]$.

Last, as in the first case, $\iota'_n$ extends to a homomorphism $\varphi_n : \Gamma \rightarrow G$ that coincides with the identity on $A$ and with an inner automorphism on $B$.\end{proof}






\subsection{Main result}

We will prove the difficult part of Theorem \ref{jjjjjjj}, that is the following result.

\begin{te}\label{mainsac}Let $G$ be a hyperbolic group. Suppose that $G$ splits over a finite subgroup $C$ whose normalizer $N$ is virtually cyclic and non-elliptic in the splitting. Let $\Gamma=G\ast_N N'$ be a legal small extension of $G$. Then $\mathrm{Th}_{\forall\exists}(G)= \mathrm{Th}_{\forall\exists}(\Gamma)$.\end{te}

The proof of this theorem relies on the following lemma, which can be viewed as an analogue of Proposition \ref{sacerdote1}.

\begin{lemme}\label{sacerdote3}We keep the same notations as in the statement of Theorem \ref{mainsac}. Let $(\varphi_n : \Gamma\rightarrow G)_{n\in\mathbb{N}}$ be a small test sequence (whose existence is guaranteed by Lemma \ref{guaranteed nash}). Let $\Sigma(\boldsymbol{x},\boldsymbol{y})=1\wedge \Psi(\boldsymbol{x},\boldsymbol{y})\neq 1$ be a system of equations and inequations, where $\boldsymbol{x}$ denotes a $p$-tuple of variables and $\boldsymbol{y}$ denotes a $q$-tuple of variables. Let $\boldsymbol{\gamma}\in\Gamma^p$. Suppose that $G$ satisfies the following condition: for every integer $n$, there exists $\boldsymbol{g}_n\in G^q$ such that \[\Sigma(\varphi_n(\boldsymbol{\gamma}),\boldsymbol{g}_n)=1 \ \wedge \ \Psi(\varphi_n(\boldsymbol{\gamma}),\boldsymbol{g}_n)\neq 1.\] Then there exists a retraction $r$ from $\Gamma_{\Sigma,\boldsymbol{\gamma}}=\langle \Gamma , \boldsymbol{y} \ \vert \ \Sigma(\boldsymbol{\gamma},\boldsymbol{y})=1\rangle$ onto $\Gamma$ such that each component of the tuple $r(\Psi(\boldsymbol{\gamma},\boldsymbol{y}))$ is non-trivial. In particular, the $q$-tuple $\boldsymbol{\gamma}':=r(\boldsymbol{y})\in \Gamma^q$ satisfies \[\Sigma(\boldsymbol{\gamma},\boldsymbol{\gamma}')=1 \ \wedge \ \Psi(\boldsymbol{\gamma},\boldsymbol{\gamma}')\neq 1.\]\end{lemme}

\subsubsection{Proof of Lemma \ref{sacerdote3} in a particular case}


We first prove Lemma \ref{sacerdote3} in the case where $G$ and $\Gamma$ are virtually cyclic of cyclic type. Let $G=C\rtimes \langle t\rangle$ and $\Gamma=C\rtimes \langle z\rangle$. The main difference compared to the general case is that we do not have to deal with actions on real trees here.

\begin{proof}For every $n$, the morphism $\varphi_n$ extends to a homomorphism $\widehat{\varphi}_n:\Gamma_{\Sigma,\boldsymbol{\gamma}}\rightarrow G$ mapping $\bm{y}$ to $\bm{g}_n$. We shall construct a retraction $r : \Gamma_{\Sigma,\boldsymbol{\gamma}}\twoheadrightarrow \Gamma$ that does not kill any component of the tuple $\Psi(\bm{\gamma},\bm{y})$. Up to extracting a subsequence, one can suppose that the sequence $(\widehat{\varphi}_n)_{n\in\mathbb{N}}$ is stable. Let $K=\ker((\widehat{\varphi}_n)_{n\in\mathbb{N}})$ be the stable kernel of the sequence, let $L=\Gamma_{\Sigma,\boldsymbol{\gamma}}/K$ and $\pi : \Gamma_{\Sigma,\boldsymbol{\gamma}} \twoheadrightarrow L$ be the associated epimorphism. As a $G$-limit group, $L$ is equationally noetherian (see \cite{RW14} Corollary 6.13). It follows that there exists a unique homomorphism $\rho_{n} : L\rightarrow G$ such that $\widehat{\varphi}_n=\rho_{n}\circ \pi$, for $n$ sufficiently large. 

By Remark \ref{inj nashville}, every $\rho_n$ is injective. As a consequence, the homomorphism $\pi : \Gamma \rightarrow L$ is injective, and every component of the tuple $\pi(\Psi(\bm{\gamma},\bm{y}))\in L$ is non-trivial. From now on, $\Gamma$ is viewed as a subgroup of $L$ and we do not mention the monomorphism $\pi : \Gamma \hookrightarrow L$ anymore.

In order to construct $r$, we will construct a discriminating sequence of retractions $(r_n : L \twoheadrightarrow \Gamma)_{n\in\mathbb{N}}$. Then, we will conclude by taking $r:=r_n\circ \pi$ for $n$ sufficiently large. Note that $\rho_n$ coincides with $\widehat{\varphi}_n$ on $\Gamma$; in particular, $({\rho_n}_{\vert \Gamma} : \Gamma \rightarrow G)_{n\in\mathbb{N}}$ is a small test sequence.  

Note that $L$ is finitely generated, as a quotient of the finitely generated group $\Gamma_{\Sigma,\boldsymbol{\gamma}}$. In addition, the sequence $(\rho_n)_{n\in\mathbb{N}}$ is discriminating. Therefore, $L$ is an extension \[1\rightarrow C\rightarrow L \overset{\chi}{\twoheadrightarrow}V\simeq \mathbb{Z}^m\rightarrow 1.\]First, we prove that $\chi(z)$ has no root in $V$. This element can be written as $\chi(z)=v^j$ for some element $v\in V$ with no root, and some integer $j\neq 0$. We will prove that $j=\pm 1$. Each homomorphism $\rho_n : L \rightarrow G$ induces a homomorphism $\overline{\rho}_n : V \rightarrow G/C\simeq \langle t\rangle$. For every integer $n$, we have $\overline{\rho}_n(v)^j=t^{p_n}$. It follows that $j$ divides $p_n$, for every $n$. Since $(p_n)_{n\in\mathbb{N}}$ is a strictly increasing sequence of prime numbers, $j=\pm 1$. 

In order to define the retraction $r_n : L \twoheadrightarrow \Gamma$, we use a presentation of $L$. Let $(v_1,v_2,\ldots,v_m)$ be a basis of $V$, with $v_1=v$. For $1\leq i\leq m$, let $z_i$ be a preimage of $v_i$ in $L$. One can suppose that $z_1=z$. Each element $z_i$ induces an automorphism $\zeta_i$ of $C$, and each commutator $[z_i,z_j]$ is equal to an element $c_{i,j}\in C$. Here is a presentation of $L$:\[ \langle C,z_1,\ldots ,z_m \ \vert \ \mathrm{ad}(z_i)_{\vert C}=\zeta_i,\ [z_i,z_j]=c_{i,j}\rangle.\]

Let $\gamma=z^{K!}$. We denote by $\tau_{n}$ the endomorphism of $\Gamma$ defined by $\tau_n=\mathrm{id}$ on $C$ and $\tau_n(z)=z\gamma^n$ (well-defined since $\gamma$ centralizes $C$). Let us define $r_n : L \rightarrow \Gamma$ by $r_n=\mathrm{id}$ on $\langle C,z\rangle$ and $r_n(z_i)=\tau_n\circ \rho_n(z_i)$ for $2\leq i\leq m$. In order to verify that $r_n$ is well-defined, it suffices to show that $[r_n(z),r_n(z_i)]=r_n(c_{1,i})$, since $r_n$ coincides with the identity on $C$. Recall that $\rho_n(z)=z\delta_n$ with $\delta_n\in D(\Gamma)$. As a consequence, $\tau_n\circ \rho_n(z)=zz^{R_n}$ with $z^{R_n}$ in the center of $\Gamma$. Therefore,\[\tau_n\circ \rho_n([z,z_i])=[zz^{R_n},\tau_n\circ \rho_n(z_i)]=[z,\tau_n\circ \rho_n(z_i)]=c_{1,i}.\]Hence, $[r_n(z),r_n(z_i)]=r_n(c_{1,i})$, so $r_n$ is well-defined.

It remains to prove that the sequence $(r_n)_{n\in\mathbb{N}}$ is discriminating. Let $x\in L$ be a non-trivial element, and let us prove that $r_{n}(x)$ is non-trivial for every $n$ sufficiently large. The element $x$ can be written as $x=cz^{q_1}z_2^{q_2}\cdots z_m^{q_m}$ with $c\in C$ and $q_i\in\mathbb{Z}$. If $x$ lies in $\Gamma$, then $r_n(x)=x\neq 1$. Else, if $x$ does not belong to $\Gamma$, then $y=z_2^{q_2}\cdots z_m^{q_m}$ has infinite order (otherwise, $y$ would belong to $C$, so $x$ would belong to $\Gamma$). Since the sequence $(\rho_n)_{n\in\mathbb{N}}$ is discriminating, $\rho_n(y)$ has infinite order for every $n$ large enough, so $\rho_n(y)=z^{\ell_n} \mod C$ with $\ell_n\neq 0$. Thus $\tau_n\circ \rho_n(y)=(zz^{nK!})^{\ell_n}=z^{\ell_n(1+nK!)}$. For $n> \vert q_1\vert$, $\vert \ell_n(1+nK!)\vert > \vert q_1\vert$, so $r_n(x)$ is non-trivial. As a conclusion, the sequence of retractions $(r_{n}: L \twoheadrightarrow \Gamma)_{n\in\mathbb{N}}$ is discriminating.\end{proof}

\subsubsection{Proof of Lemma \ref{sacerdote3} in the general case}We now prove Lemma \ref{sacerdote3}.

\begin{proof}For every $n$, the map $\varphi_n$ extends to a homomorphism $\widehat{\varphi}_n:\Gamma_{\Sigma,\boldsymbol{\gamma}}\rightarrow G$ mapping $\bm{y}$ to $\bm{g}_n$. We shall construct a retraction $r : \Gamma_{\Sigma,\boldsymbol{\gamma}}\twoheadrightarrow \Gamma$ that does not kill any component of the tuple $\Psi(\bm{\gamma},\bm{y})$. Up to exctracting a subsequence, one can suppose that the sequence $(\widehat{\varphi}_n)_{n\in\mathbb{N}}$ is stable. Let $K=\ker((\widehat{\varphi}_n)_{n\in\mathbb{N}})$ be the stable kernel of the sequence, let $L=\Gamma_{\Sigma,\boldsymbol{\gamma}}/K$ and $\pi : \Gamma_{\Sigma,\boldsymbol{\gamma}} \twoheadrightarrow L$ be the associated epimorphism. As a $G$-limit group, $L$ is equationally noetherian (see \cite{RW14} Corollary 6.13). It follows that there exists a unique homomorphism $\rho_{n} : L\rightarrow G$ such that $\widehat{\varphi}_n=\rho_{n}\circ \pi$, for $n$ sufficiently large. 

By Remark \ref{inj nashville}, every $\rho_n$ is injective. It follows that the homomorphism $\pi : \Gamma \rightarrow L$ is injective, and every component of the tuple $\pi(\Psi(\bm{\gamma},\bm{y}))\in L$ is non-trivial. From now on, $\Gamma$ is viewed as a subgroup of $L$ and we do not mention the monomorphism $\pi : \Gamma \hookrightarrow L$ anymore.

In order to construct $r$, we will construct a discriminating sequence of retractions $(r_n : L \twoheadrightarrow \Gamma)_{n\in\mathbb{N}}$. Then, we will conclude by taking $r:=r_n\circ \pi$ for $n$ sufficiently large. Note that $\rho_n$ coincides with $\widehat{\varphi}_n$ on $\Gamma$; in particular, $({\rho_n}_{\vert \Gamma} : \Gamma \rightarrow G)_{n\in\mathbb{N}}$ is a small test sequence. 


Let $(X,d)$ be a Cayley graph of $G$. Let us consider a Stallings splitting $\Lambda$ of $L$ relative to $\Gamma$, and let $U$ be the one-ended factor that contains $\Gamma$. Let $S$ be a generating set of $U$. Recall that $\mathrm{Aut}_{\Gamma}(U)$ is the subgroup of $\mathrm{Aut}(U)$ consisting of all automorphisms $\sigma$ satisfying the following two conditions:
\begin{enumerate}
\item $\sigma_{\vert \Gamma}=\mathrm{id}_{\vert \Gamma}$;
\item for every finite subgroup $F$ of $U$, there exists an element $u\in U$ such that $\sigma_{\vert F}=\mathrm{ad}(u)_{\vert F}$.
\end{enumerate}
Recall that a homomorphism $\varphi : U \rightarrow G$ is said to be \emph{short} if its length $\ell(\varphi):=\max_{s\in S}d(1,\varphi(s))$ is minimal among the lengths of homomorphisms in the orbit of $\varphi$ under the action of $\mathrm{Aut}_{\Gamma}(U)\times \mathrm{Inn}(G)$. 

Since $\vert\vert t_n\vert\vert$ goes to infinity, there exists a sequence of automorphisms $(\sigma_{n})_{n\in\mathbb{N}}\in\mathrm{Aut}_{\Gamma}(U)^{\mathbb{N}}$ and a sequence of elements $(x_n)_{n\in\mathbb{N}}\in G^n$ such that the homomorphisms $\mathrm{ad}(x_n)\circ \rho_{n}\circ \sigma_{n}$ are short, pairwise distinct, and such that the sequence $(\mathrm{ad}(x_n)\circ \rho_{n}\circ \sigma_{n})_{n\in\mathbb{N}}$ is stable (see paragraph \ref{stabletreenash}), up to extracting a subsequence. Since $\rho_n$ coincides with the identity on $G$, we have $\mathrm{ad}(x_n)\circ \rho_{n}\circ \sigma_{n}= \rho_{n}\circ \mathrm{ad}(x_n)\circ \sigma_{n}$ and, up to replacing $\sigma_n$ by $\mathrm{ad}(x_n)\circ \sigma_{n}$, we can forget the postconjugation by $x_n$ and assume that ${\sigma_n}_{\vert \Gamma}$ is a conjugation by an element of $G$.

We claim that $\sigma_n$ extends to an automorphism of $L$, still denoted by $\sigma_n$. By the second condition above, $\sigma_n$ is a conjugacy on finite subgroups of $U$. We proceed by induction on the number of edges of $\Lambda$ (the Stallings splitting of $L$ relative to $\Gamma$ used previously in order to define $U$). It is enough to prove the claim in the case where $\Lambda$ has only one edge. 

If $L=U\ast_C B$ with ${\sigma_n}_{\vert C}=\mathrm{ad}(u)$, one defines $\sigma_n : L \rightarrow G$ by ${\sigma_n}_{\vert U}=\sigma_n$ and ${\sigma_n}_{\vert B}=\mathrm{ad}(u)$. 

If $L=U\ast_C=\langle U,t \ \vert \ tct^{-1}=\alpha(c), \forall c\in C\rangle$ with ${\sigma_n}_{\vert C}=\mathrm{ad}(u_1)$ and ${\sigma_n}_{\vert \alpha{(C)}}=\mathrm{ad}(u_2)$, one defines $\sigma_n : L \rightarrow G$ by ${\sigma_n}_{\vert U}=\sigma_n$ and $\sigma_n(t)=u_2^{-1}tu_1$.

In order to complete the proof of Lemma \ref{sacerdote3}, we will use the following lemma. 

\begin{lemme}\label{abovelemma2}We keep the same notations as above. Let $T$ be the limit tree of the sequence of metric spaces $(X,d/\ell({\rho_n\circ \sigma_n}))_{n\in\mathbb{N}}$. The following dichotomy holds:
\begin{itemize}
\item[$\bullet$]either $\Gamma$ does not fix a point of $T$, in which case there exists a discriminating sequence of retractions $(r_n : L \twoheadrightarrow\Gamma)_{n\in\mathbb{N}}$,
\item[$\bullet$]or $\Gamma$ is elliptic, and there exist a proper quotient $L_1$ of $L$, an embedding $\Gamma\hookrightarrow L_1$ allowing us to identify $\Gamma$ with a subgroup of $L_1$, and two discriminating sequences $(\rho_n^1 : L_1\rightarrow G)_{n\in\mathbb{N}}$ and $(\theta_n^1 : L\twoheadrightarrow L_1)_{n\in\mathbb{N}}$ such that the following three conditions are satisfied: 
\begin{enumerate}
\item $\rho_n\circ \sigma_n=\rho_n^1\circ \theta_n^1$;
\item $\rho_n^1$ coincides with $\rho_n$ on $\Gamma$; in particular, $({\rho_n^1}_{\vert \Gamma} : \Gamma \rightarrow G)_{n\in\mathbb{N}}$ is a test sequence.
\item There exists an element $g_n\in G$ such that $\sigma_n$ and $\theta_n^1$ coincides with $\mathrm{ad}(g_n)$ on $\Gamma$.
\end{enumerate}
\end{itemize}
\end{lemme}


Before proving this lemma, we will use it to conclude the proof of Lemma \ref{sacerdote3}. If $\Gamma$ does not fix a point of $T$, we are done. If $\Gamma$ fixes a point of $T$, by iterating Lemma \ref{abovelemma2}, we get a sequence of proper quotients \[L_0=L \twoheadrightarrow L_1\twoheadrightarrow \cdots L_i\twoheadrightarrow \cdots\]such that, for every integer $i\geq 1$, there exist two discriminating sequences of morphisms $(\rho_{n}^i : L_i \rightarrow G)_{n\in\mathbb{N}}$ and $(\theta_{n}^i: L_{i-1} \twoheadrightarrow L_{i})_{n\in\mathbb{N}}$ such that $\rho_{n}^{i-1}\circ\sigma_n=\rho_{n}^{i}\circ \theta_{n}^i$, $\rho_n^i$ coincides with $\rho_n^{i-1}$ on $\Gamma$, and there exists an element $g_n\in G$ such that $\sigma_n$ and $\theta_n^i$ coincides with $\mathrm{ad}(g_n)$ on $\Gamma$.

By the descending chain condition, the iteration eventually terminates. Let $L_k$ be the last quotient of the series. By Lemma \ref{abovelemma2}, there exists a discriminating sequence of retractions $(r_n : L_k \twoheadrightarrow\Gamma)_{n\in\mathbb{N}}$. For every finite set $F\subset L$, one can find some integers $n_1,\ldots ,n_{k}$ such that the morphism $r_{n_{k}}\circ \theta_{n_k}^k\circ \cdots\circ \theta_{n_1}^1 : L\twoheadrightarrow\Gamma$ is injective on $F$. Moreover, since every $\theta_{n_i}$ is a conjugation on $\Gamma$ by an element of $G$, there exists an element $v_n\in G$ such that $\mathrm{ad}(v_n)\circ r_{n_{k}}\circ \theta_{n_k}^k\circ \cdots\circ \theta_{n_1}^1 : L\twoheadrightarrow\Gamma$ is a retraction. This concludes the proof of Lemma \ref{sacerdote3}\end{proof}

It remains to prove Lemma \ref{abovelemma2}. 

\begin{proof}Recall that $U$ denotes the one-ended factor of $L$ relative to $\Gamma$. We distinguish two cases.

\vspace{2mm}

\textbf{First case.} Suppose that $\Gamma$ fixes a point of $T$. Let us prove that the stable kernel of the sequence $(\rho_n\circ \sigma_n)_{n\in\mathbb{N}}$ is non-trivial. Assume towards a contradiction that the stable kernel is trivial. Then, by Theorem 1.16 of \cite{RW14}, the action of $(U,\Gamma)$ on the limit tree $T$ has the following properties:
\begin{itemize}
\item[$\bullet$]the pointwise stabilizer of any non-degenerate tripod is finite;
\item[$\bullet$]the pointwise stabilizer of any non-degenerate arc is finitely generated and finite-by-abelian;
\item[$\bullet$]the pointwise stabilizer of any unstable arc is finite.
\end{itemize}

In particular, the tree $T$ satisfies the ascending chain condition of Theorem \ref{guirardel1} since any ascending sequence of finitely generated and finite-by-abelian subgroups of a hyperbolic group stabilizes.

Then, it follows from Theorem \ref{guirardel1} that either $(U,\Gamma)$ splits over the stabilizer of an unstable arc, or over the stabilizer of an infinite tripod, or $T$ has a decomposition into a graph of actions. Since $U$ is one-ended relative to $\Gamma$, and since the stabilizer of an unstable arc or of an infinite tripod is finite, it follows that $T$ has a decomposition into a graph of actions.

Now, it follows from Theorem \ref{nashpaulinnash} that there exists a sequence of automorphisms $(\alpha_n)_{n\in\mathbb{N}}\in\mathrm{Aut}_{\Gamma}(U)^{\mathbb{N}}$ such that $(\rho_n\circ\sigma_n)\circ \alpha_n$ is shorter than $\rho_n\circ \sigma_n$ for $n$ large enough. This is a contradiction since the morphisms $\rho_n\circ \sigma_n$ are assumed to be short. Hence, the stable kernel of the sequence $(\rho_n\circ \sigma_n)_{n\in\mathbb{N}}$ is non-trivial.

As in the proof of Lemma \ref{sacerdote3} above, since $\sigma_{n}$ coincides with an inner automorphism on each finite subgroup of $U$ (by definition of $\mathrm{Aut}_{\Gamma}(U)$), it extends to an automorphism of $L$, still denoted by $\sigma_n$. Let $L_1:=L/\ker(({\rho_{n}}\circ \sigma_{n})_{n\in\mathbb{N}})$ and let $\pi_1 : L \twoheadrightarrow L_1$ be the corresponding epimorphism. Observe that $\pi_1$ is injective on $\Gamma$, allowing to identify $\Gamma$ with a subgroup of $L_1$. 

As a $G$-limit group, $L_1$ is equationally noetherian (see \cite{RW14} Corollary 6.13). It follows that, for every integer $n$, there exists a unique homomorphism $\tau_n^1 : L_1\rightarrow G$ such that $\rho_n\circ \sigma_n=\tau_n^1\circ\pi_1$. There exists an element $g_n\in G$ such that $\sigma_n$ coincides with $\mathrm{ad}(g_n)$ on $\Gamma$. Hence, since $\rho_n$ coincides with the identity on $G$, we can write $\rho_n\circ \sigma_n=\mathrm{ad}(g_n)\circ \rho_n^1\circ\pi_1$ in such a way that $\rho_n^1:=\mathrm{ad}(g_n^{-1})\circ \tau_n^1$ coincides with the identity on $G$. For every $n$, let $\theta_{n}^1= \pi_1\circ (\sigma_n)^{-1}$, so that $\rho_n=\mathrm{ad}(g_n)\circ \rho_n^1\circ \theta_n^1$. The sequence $(\theta_n^1 : L \rightarrow L_1)_{n\in\mathbb{N}}$ is therefore discriminating, and every homomorphism $\theta_n^1$ coincides with $\mathrm{ad}(g_n^{-1})$ on $\Gamma$.

\vspace{2mm}

\textbf{Second case.} Suppose that $\Gamma$ does not fix a point of $T$. In the sequel, the letter $z$ denotes an element of $N'=N_{\Gamma}(C)$ defined as follows.
\begin{itemize}
\item[$\bullet$]If $\Gamma=A\ast_C$ and the copies of $C$ in $A$ are conjugate, one can suppose without without loss of generality that there are equal. The letter $z$ denotes the stable letter of the HNN extension. We have $N_{\Gamma}(C)=C\rtimes \langle z\rangle$.
\item[$\bullet$]If $\Gamma=A\ast_C B$, we define $z$ by $z=ba$ where $b$ and $a$ are such that $N_B(C)=C\rtimes \langle b\rangle$ et $N_A(C)=C\rtimes \langle a\rangle$. Note that $b$ and $a$ have order 2 modulo $C$ in $N_B(C)$ and $N_A(C)$ respectively.
\item[$\bullet$]If $\Gamma=A\ast_C$ and the two copies of $C$ in $A$ are not conjugate, let $s$ be the stable letter of the HNN extension, and let $a\in A$ and $b\in sAs^{-1}$ be two elements such that $N_{A}(C)=C\rtimes \langle a\rangle$ and $N_{sAs^{-1}}(C)=C\rtimes \langle b\rangle$. Then we define $z$ by $z=ba$.
\end{itemize}
In the same way, we define $t\in N=N_G(C)$.

Since $\Gamma$ is not elliptic in $T$, and since $A$ and $B$ are elliptic, $z$ acts hyperbolically on $T$. Let $d$ denote its axis. We will proceed in two steps.

\begin{itemize}
\item[$\bullet$]\textbf{First step.} We prove that the one-ended factor $U$ relative to $\Gamma$ can be decomposed as a graph of groups $\Delta$ in which $S$ is a vertex group, where $S$ denotes the global stabilizer of the axis $d$ of $z$ in $T$.
\item[$\bullet$]\textbf{Second step.} We construct a discriminating sequence of retractions $(r_n : L \rightarrow \Gamma)$. 
\end{itemize}

\vspace{2mm}

\textbf{First step.} We shall prove that $U$ splits as a graph of groups in which $S$ is a vertex group. First, we prove that the translates of $d$ are transverse. Let $u$ be an element of $U$ such that $ud\cap d$ contains a segment $I$ which is not reduced to a point. Let us prove that $ud=d$. There exists a constant $R_t$ such that the following holds: every element $g\in G$ such that $A(t)^{+\delta}\cap gA(t)^{+\delta}$ has a diameter larger than $R_t$ belongs to $M(t)$, where $A(t)$ refers to the quasi-axis of $t$ (see Section \ref{Coulon}). Taking $g=\rho_n(u)$, we have $\rho_n(u)\in M(t)=N_G(C)$ for $n$ large enough. Therefore, since $\rho_n(z)$ belongs to $\langle t\rangle$, one of the following two possibilities occurs, for $n$ sufficiently large: either $\rho_n([u,z])$ belongs to $C$, or $\rho_n((uz)^2)$ belongs to $C$. Since $C=\rho_n(C)$, and since the sequence $(\rho_n)$ is discriminating, we deduce that $[u,z]$ or $(uz)^2$ belongs to $C$. Since $C$ fixes $d$ pointwise, $u$ fixes $d$ pointwise in the first case; in the second case, $u$ acts on $d$ by reversing the orientation. As a conclusion, $ud=d$. Hence, the translates of the axis $d$ are transverse.

Recall that $S$ is the global stabilizer of $d$ in $U$. The following facts can be proved by using the discriminating sequence $({\rho_n}_{\vert S} : S \rightarrow N_G(C)=M(t))$: 
\begin{enumerate}
\item $C$ is the maximal finite normal subgroup of $S$,
\item $S=N_{U}(C)$,
\item $S$ is virtually abelian. To be more precise, $S/C$ is abelian in the case where $N_G(C)$ is of cyclic type, and $S/C$ has an abelian subgroup of index 2 in the case where $N_G(C)$ is of dihedral type (see Lemma \ref{lemmeDinf} below).
\end{enumerate}
By \cite{RW14} Theorem 6.3, the group $S$ is finitely generated. We are ready to prove that $U$ splits as a graph of groups in which $S$ is a vertex group. 

\vspace{2mm}

\emph{First case.} Suppose that the action of $S$ on $d$ is discrete. We define an equivalence relation $\sim$ on $T$ by $x\sim y$ if the intersection $[x,y]\cap u d$ contains at most one point, for any element $u\in U$. Let $(Y_j)_{j\in J}$ denote the equivalence classes that are not reduced to a point. Note that each $Y_j$ is a closed subtree of $T$. Let us verify that the family $(Y_j)_{j\in J}\cup \lbrace ud\ \vert \in U\rbrace$ is a transverse covering of $T$ (see Definition \ref{transverse}). 

\begin{itemize}
\item[$\bullet$]\emph{Transverse intersection.} By definition, $Y_i\cap Y_j=\varnothing$ for every $i\in I$, and $\vert Y_i\cap ud\vert \leq 1$ for every $i\in I$ and every $u\in U$. In addition, $d$ is transverse to its translates.
\item[$\bullet$]\emph{Finiteness condition.} Let $x$ and $y$ be two points of $T$. Since the action of $S$ on $d$ is discrete, there exists a constant $\varepsilon > 0$ such that the following holds: for any $u\in U$, if the intersection $[x,y]\cap ud$ is non-degenerate, then the length of $[x,y]\cap ud$ is bounded from below by $\varepsilon$. Therefore, the arc $[x,y]$ is covered by at most $\lfloor d'(x,y)/\varepsilon\rfloor$ translates of the axis $d$ and at most $\lfloor d'(x,y)/\varepsilon\rfloor+1$ distinct subtrees $Y_j$, where $d'$ denotes the metric on $T$.
\end{itemize}

Thus, the family $(Y_j)_{j\in J}\cup \lbrace ud \ \ \vert \ u\in U\rbrace$ is a transverse covering of the tree $T$. We can now build the skeleton of this transverse covering in the sense of Guirardel (see Definition \ref{squelette}), denoted by $T'$. Since the action of $U$ on $T$ is minimal, so is the action of $U$ on $T'$, according to Lemma 4.9 of \cite{Gui04}. One of the vertex group of $T'$ is equal to $S$, by construction. This concludes the proof of the first case.

\vspace{2mm}

\emph{Second case.} Suppose that the action of $S$ on $d$ has dense orbits. The previous argument no longer works (because $\varepsilon$ does not exist). Note that $S$ is virtually $\mathbb{Z}^n$ with $n\geq 2$. 

Since the discriminating sequence $(\rho_n)$ coincides with inner automorphisms on $A$ and $B$, these groups are elliptic in $T$. We will apply Theorem \ref{guirardel1} to the pair $(U,\lbrace A,B\rbrace)$ (resp.\ $(U,A)$) in order to decompose $T$ as a graph of actions. This is enough to conclude because $d$ is one of the components of the graph of actions. Indeed, the action of $S$ on $d$ is indecomposable (see Definition 1.17 in \cite{Gui08}), so $d$ is contained in one of the components of the graph of actions by Lemma 1.18 of \cite{Gui08}. Let $\mathcal{C}$ denote this component. Note that each component given by Theorem \ref{guirardel1} is either indecomposable or simplicial. Since $S\curvearrowright d$ has dense orbits, $\mathcal{C}$ is not simplicial. Thus, $\mathcal{C}$ is indecomposable. The axis $d\subset\mathcal{C}$ being transverse to its translates, $d$ is necessarily equal to $\mathcal{C}$.

Assume towards a contradiction that Theorem \ref{guirardel1} does not give a splitting of $T$ as a graph of actions. Then, still by Theorem \ref{guirardel1}, the pair $(U,A)$ (resp.\ $(U,\lbrace A,B\rbrace)$) splits over a finite subgroup $E$ which is either the pointwise stabilizer of an unstable arc, or the pointwise stabilizer of an infinite tripod whose normalizer contains the free group $F_2$. Let $Y$ be the Bass-Serre tree of this splitting. We will prove that $\Gamma$ is elliptic in $Y$, contradicting the fact that $U$ is one-ended relative to $\Gamma$. Let $Y_{\Gamma}$ be the minimal subtree of $\Gamma$. Let $Z$ be the tree of cylinders of the splitting $A\ast_C$ or $A\ast_C B$ of $\Gamma$. Its vertex groups are $A$ and $N$ in the first case, and $A,B$ and $N$ in the second case. In the cyclic case, there is only one edge group in $Z$, namely $C$; in the dihedral case, there are at most two edges groups $C_1$ and $C_2$ that contain $C$ with index 2. We claim that $Z$ dominates $Y_{\Gamma}$, i.e.\ that its vertex groups are elliptic in $Y_{\Gamma}$. Since $Y_{\Gamma}$ is a splitting relative to the pair $\lbrace A,B\rbrace$, we juste have to prove that $N$ is elliptic in $Y_{\Gamma}$. Since $N\subset S$, it is enough to observe that $S$ is elliptic in $Y$, as a one-ended group (because it is virtually $\mathbb{Z}^n$ with $n\geq 2$). Thus, the tree of cylinders $Z$ dominates $Y_{\Gamma}$. As a consequence, $E$ contains $C$ up to conjugation. Since $N_U(C)=S$ is virtually abelian, the normalizer $N_U(E)$ of $E$ is virtually abelian as well. It follows that $E$ is not the pointwise stabilizer of an infinite tripod whose normalizer contains the free group $F_2$. So $E$ is the stabilizer of an unstable arc $I\subset T$. Therefore, there exists a subarc $I'\subset I$ whose stabilizer $E'$ satisfies $E'\supset E$ and $E'\neq E$. The set of fixed points of $C$ is exactly $d$. Since $E\supset C$, $\mathrm{Fix}(E)$ is contained in $d=\mathrm{Fix}(C)$, and since $d$ is transverse to its translates, $\mathrm{Fix}(E)=d$ or $\mathrm{Fix}(E)$ is one point of $d$. This second possibility cannot happen since $\mathrm{Fix}(E)$ contains $I$. Hence, $E$ is contained in $S$, so $E=S_0$ (the pointwise stabilizer of $d$). The same argument shows that $E'=S_0=E$, a contradiction. As a conclusion, $T$ splits as a graph of actions one component of which is $d$. So $U$ splits as a graph of groups in which $S$ is a vertex group. This completes the second case.

\vspace{2mm}

It remains to construct a discriminating sequence of retractions $(r_n : L \rightarrow \Gamma)$. We distinguish two cases, depending on the type of $N_G(C)$. In the sequel, we denote by $\Delta$ a splitting of $U$ as a graph of groups in which $S$ is a vertex group.

\vspace{2mm}

\textbf{Second step.} Construction of the discriminating sequence of retractions $(r_n : L \rightarrow \Gamma)$.

\vspace{2mm}

\textbf{A. The cyclic case.} Suppose that $N_G(C)$ is of cyclic type. Then $S$ is an extension \[1\rightarrow C\rightarrow S \overset{\chi}{\twoheadrightarrow}V\simeq \mathbb{Z}^m\rightarrow 1,\]for some integer $m$. We claim that $\chi(z)$ has no root in $V$, i.e.\ that $\langle \chi(z)\rangle$ is cyclic maximal. This element can be written as $\chi(z)=v^j$ for some element $v\in V$ with no proper root, and some integer $j\neq 0$. We will prove that $j=\pm 1$. Each homomorphism ${\rho_n}_{\vert S} : S \rightarrow G$ induces a homomorphism $\overline{\rho}_n : V \rightarrow G/C\simeq \langle t\rangle$. For every integer $n$, we have $\overline{\rho}_n(v)^j=t^{p_n}$. It follows that $j$ divides $p_n$, for every $n$. Since $(p_n)$ is a strictly increasing sequence of prime numbers, $j=\pm 1$. 

In order to define the retraction $r_n : L \twoheadrightarrow \Gamma$, we need a presentation of $S$. Let $(v_1,v_2,\ldots,v_m)$ be a basis of $V$, with $v_1=v$. For $1\leq i\leq m$, let $z_i$ be a preimage of $v_i$ in $L$. One can suppose that $z_1=z$. Each element $z_i$ induces an automorphism $\zeta_i$ of $C$, and each commutator $[z_i,z_j]$ is equal to an element $c_{i,j}\in C$. Here is a presentation of $S$:\[ \langle C,z_1,\ldots ,z_m \ \vert \ \mathrm{ad}(z_i)_{\vert C}=\zeta_i,\ [z_i,z_j]=c_{i,j}\rangle.\]

Let $v$ denote the vertex of $\Delta$ whose stabilizer is $S$. Note that any edge adjacent to $v$ in $\Delta$ has a stabilizer contained in $S_0=\langle C, z_2,\ldots ,z_m\rangle$, because any stabilizer of a point of $d$ is contained in $S_0$. Let us refine $\Delta$ by replacing the vertex $v$ with the decomposition $S=S\ast_{S_0}S_0$ of $S$. More precisely, 
\begin{itemize}
\item[$\bullet$]the $v$ vertex is replaced by a graph of groups with two vertices denoted by $v_0$ and $v_1$, linked by a single edge $e$, such that the stabilizers of $v_0$ and $e$ are equal to $S_0$ and the stabilizer of $v_1$ is equal to $S$;
\item[$\bullet$]we replace each edge $\varepsilon=[v,w]$ $\Delta$ with an edge $[v_0,w]$ if $w\neq v$, and with an edge $[v_0,v_0]$ if $w=v$, which is always possible since $U_{\varepsilon}$ is contained in $S_0$.
\end{itemize}
We denote by $\Delta'$ this new decomposition of $U$ as a graph of groups. Then, let us consider a JSJ splitting of $L$ relative to $U$ over finite groups, and let $\Lambda$ be the splitting of $L$ obtained from this JSJ splitting by replacing the vertex fixed by $U$ with the splitting $\Delta' $ of $U$. We keep the notations $v_0$ and $v_1$ for the vertices of $\Lambda$ corresponding to the vertices $v_0$ and $v_1$ of $\Delta'$. Since all finite subgroups of $S$ are contained in $S_0$, we can assume without loss of generality (up to performing some slides of edges) that there is a only one edge $e$ adjacent to $v_1$ in $\Lambda$, namely the one that connects $v_1$ to $v_0$. Now, we collapse all the edges of $\Lambda$ other than $e$, and we get a decomposition of $L$ as an amalgamated product $L=L_0\ast_{S_0}S$, which can also be viewed as an HNN extension $L=L_0\ast_{S_0}$ with stable letter $z$.

We are now ready to define the retraction $r_n: L \rightarrow \Gamma$. First, we define $r_n$ on $S$. In order to do this, we follow the same strategy as in the proof of the particular case discussed above. Let $\gamma=z^{K!}$. We denote by $\tau_{n}$ the endomorphism of $\Gamma$ defined by $\tau_n=\mathrm{id}$ on $C$ and $\tau_n(z)=z\gamma^n$ (well-defined since $\gamma$ centralizes $C$). Let us define $r_n : S \rightarrow \Gamma$ by $r_n=\mathrm{id}$ on $\langle C,z\rangle$ and $r_n(z_i)=\tau_{f(n)}\circ \rho_n(z_i)$ for $2\leq i\leq m$ (where $f:\mathbb{N}\rightarrow\mathbb{N}$ denotes a strictly increasing function that will be specified later). In order to verify that $r_n$ is well-defined, it suffices to show that $[r_n(z),r_n(z_i)]=r_n(c_{1,i})$, since $r_n$ coincides with the identity on $C$. Recall that $\rho_n(z)=z\delta_n$ with $\delta_n\in D(\Gamma)$. As a consequence, $\tau_{f(n)}\circ \rho_n(z)=zz^{R_n}$ with $z^{R_n}$ in the center of $\Gamma$. Therefore,\[\tau_{f(n)}\circ \rho_n([z,z_i])=[zz^{R_n},\tau_{f(n)}\circ \rho_n(z_i)]=[z,\tau_{f(n)}\circ \rho_n(z_i)]=c_{1,i}.\]Hence, $[r_n(z),r_n(z_i)]=r_n(c_{1,i})$, so $r_n$ is well-defined.

Since $L=S\ast_{S_0} L_0$ and since $r_n$ coincides with $\tau_{f(n)}\circ\rho_n$ on $S_0$, the morphism $r_n : S \rightarrow \Gamma$ extends to a morphism from $L$ to $\Gamma$ that coincides with $r_n$ on $S$ and with $\tau_{f(n)}\circ \rho_n$ on $L_0$. This new morphism is still denoted by $r_n$. Note that $r_n$ is a retraction from $L$ onto $\Gamma$. Indeed, $r_n(z)=z$ and $r_n=\tau_{f(n)}\circ \rho_n=\mathrm{id}$ on $A$, and $\Gamma=\langle z,A\rangle$.

To conclude, we will prove that the sequence $(r_n)$ is discriminating, provided that the function $f:\mathbb{N}\rightarrow\mathbb{N}$ is properly chosen. Let $F\subset L\setminus \lbrace 1\rbrace$ be a finite set. For simplicity, one can assume that $F=\lbrace x\rbrace$, the proof being identical in the case where $F$ has more than one element.

If $x$ belongs to $L_0$, since the sequence $(\rho_n)$ is discriminating, we have $\rho_n(x)\neq 1$ for $n$ large. Since $\tau_{f(n)} : \Gamma \rightarrow \Gamma$ is injective, we have $\tau_{f(n)}\circ \rho_n(x)=r_n(x)\neq 1$. Now, suppose that $x$ does not belong to $L_0$. Then $x$ can be written as $x=y_{1}z^{n_{1}}y_{2}z^{n_{2}}\cdots y_kz^{n_k}$ with $k\geq 1$, $y_{i}\in L\setminus S_0$ and $n_{i}\neq 0$ for all $1\leq i\leq k$ (except possibly $n_k$, which can be zero). Note that $S_0$ is the normalizer of $C$ in $L_0$. It follows from this observation that for $n$ sufficiently large, $\rho_n(y_{i})$ does not belong to $N_{\Gamma}(C)=\langle C,z\rangle$. We can therefore write $\rho_n(y_i)$ as a product $a_{i,1}z^{\ell_{i,1}}a_{i,2}z^{\ell_{i,2}}\cdots a_{i,q_i}z^{\ell_{i,q_i}}$ with $q_i\geq 1$, $a_{i,j}\in A\setminus C$ for all $1\leq j\leq q_i$ and $\ell_{i,j}\neq 0$ for all $1\leq j\leq q_i-1$ (the integer $\ell_{i,q_i}$ can be zero). So $r_n(y_i)$ can be written as a reduced form as follows \[r_n(y_i)=\tau_{f(n)}\circ\rho_n(y_i)=a_{i,1}z^{\ell_{i,1}f(n)}a_{i,2}z^{\ell_{i,2}f(n)}\cdots a_{i,q_i}z^{\ell_{i,q_i}f(n)}.\] 

To prove that $r_n(x)$ is non-trivial, it suffices to prove that the word obtained by concatenating the reduced forms of $r_n(y_i)$ and $z^{n_i}$ is still in reduced form in the splitting $\Gamma=A\ast_C$ with stable letter $z$, i.e.\ is of the form $r_n(x)=u_1z^{k_1}u_2z^{k_2}\cdots u_Mz^{k_M}$ with $M\geq 2$, $k_i\neq 0$ and $u\in A\setminus C$ (except maybe $k_M=0$ and $u_1\in C$). Let's look at the subword at the junction of $r_n(y_i)$, $z^{n_i}$ and $r_n(y_{i+1})$, for $1\leq i<k$. This subword is of the form $a_{i,q_i}z^{\ell_{i,q_i}f(n)}z^{n_i}a_{i+1,1}$. Since $i<k$, the integer $n_i$ is not zero. We distinguish two possibilities: either $\ell_{i,q_i}= 0$, in which case there is nothing to be done, or $\ell_{i,q_i}\neq  0$, in which case we can choose $f(n)$ large enough so that $\ell_{i,q_i}f(n)+n_i\neq 0$. 

As a conclusion, $f(n)$ can be chosen sufficiently large so that $r_n(x)\neq 1$, and the same proof is still valid if one replaces the singleton $\lbrace x\rbrace$ by a finite subset of $L\setminus \lbrace 1\rbrace$. Hence, the sequence of retractions $(r_n : L \twoheadrightarrow \Gamma)$ is discriminating.

\vspace{2mm} 

\textbf{B. The dihedral case.} Suppose that $N_G(C)$ is of dihedral type. Let us begin with an easy lemma about $D_{\infty}$-limit groups.

\begin{lemme}\label{lemmeDinf}Let $L$ be a group. Suppose that $\mathrm{Th}_{\forall}(D_{\infty})\subset \mathrm{Th}_{\forall}(L)$. Then, either $L$ is torsion-free abelian, or $L$ is a semi-direct product $Z\rtimes \langle w\rangle$ with $Z$ torsion-free abelian, and $w$ an element of order 2 acting by $-\mathrm{id}$ on $Z$.\end{lemme}

\begin{proof}The universal sentence $\forall x\forall y \ ((x^2\neq 1)\wedge (y^2\neq 1)\wedge (xy\neq 1))\Rightarrow ((xy)^2\neq 1)$, which is satisfies by $D_{\infty}$, express the fact that the $\forall$-definable subset of $D_{\infty}$ composed of all elements of order $>2$ is a subgroup (if one adds $1$). Moreover, one can express by means of universal sentences that this subgroup is torsion-free abelian, and has index at most 2. Therefore, any group $L$ such that $\mathrm{Th}_{\forall}(D_{\infty})\subset \mathrm{Th}_{\forall}(L)$ contains a torsion-free abelian subgroup of index at most 2. It follows that if $L$ is not torsion-free abelian, then it splits as a semi-direct product $Z\rtimes \langle w\rangle$ with $Z$ torsion-free abelian, and $w$ of order 2. Last, the action of $w$ on $Z$ is described by the following universal sentence, satisfied by $D_{\infty}$:\[\forall x \forall y \ ((x\neq 1)\wedge (x^2=1)\wedge (y^2\neq 1)) \Rightarrow (xyx^{-1}=y^{-1}).\]This concludes the proof of Lemma \ref{lemmeDinf}.\end{proof}

Since $N_G(C)$ is of dihedral type, it follows from the lemma above that $S$ is an extension \[1\rightarrow C\rightarrow S \overset{\chi}{\twoheadrightarrow}V\rtimes \mathbb{Z}_2\rightarrow 1,\]with $V\simeq \mathbb{Z}^m$ such that $\chi(z)\in V$, for some integer $m$. We claim that $\chi(z)$ has no root in $V$. This element can be written as $\chi(z)=v^j$ for some element $v\in V$ with no proper root, and some integer $j\neq 0$. We will prove that $j=\pm 1$. Each homomorphism ${\rho_n}_{\vert S} : S \rightarrow G$ induces a homomorphism $\overline{\rho}_n : V \rightarrow \langle t\rangle$. For every integer $n\geq n_0$, we have $\overline{\rho}_n(v)^j=t^{p_n}$. It follows that $j$ divides $p_n$, for every $n\geq n_0$. Since $(p_n)$ is a strictly increasing sequence of prime numbers, $j=\pm 1$. 

Let $\lbrace z_1 , z_2,\ldots ,z_r\rbrace\subset S$ be a finite set, with $z_1=z$, such that $\lbrace \chi(z_1),\ldots ,\chi(z_r)\rbrace$ is a basis of $V$. Before we construct the retraction $r_n: L \twoheadrightarrow \Gamma$, we need a splitting of $L$. Let $v$ denote the vertex of $\Delta$ whose stabilizer is $S$. We refine $\Delta$ by replacing $v$ by the decomposition $S=S_1\ast_{S_0}S_2$ where $S_0=\langle C,z_2,\ldots ,z_r\rangle$, $S_1=S_0\rtimes \langle a\rangle$ and $S_2=S_0\rtimes \langle b\rangle$ with $b=za$ (where $a$ and $b$ have been defined above). Note that any edge adjacent to $v$ in $\Delta$ has a stabilizer contained in a conjugate of $S_1$ or $S_2$ in $S$. We refine $\Delta$ by replacing $v$ by this splitting of $S$. Let $\Delta'$ denote this new decomposition of $U$ as a graph of groups. Then, consider a JSJ decomposition of $L$ relative to $U$ over finite groups, and let $\Lambda$ denote the decomposition of $L$ obtained from this JSJ splitting by replacing the vertex fixed by $U$ with the splitting $\Delta' $ of $U$. Then, collapsing edges if necessary, we obtain a decomposition of $L$ of the form \[L=L_1\ast_{S_1}\ast S\ast_{S_2} L_2 \ \ \ \text{ our } \ \ \ L=(L_1\ast_{S_1}\ast S)\ast_{S_2}.\]

\emph{First case:} $L=L_1\ast_{S_1}\ast S\ast_{S_2} L_2$. The group $L$ can also be decomposed as $L=L_1\ast_{S_0}L_2$. We will suppose that $\Gamma=A\ast_C B$, with $C$ is strictly contained in $A$ and in $B$. The proof is similar in the case where $\Gamma=A\ast_C$ (left to the reader).


Since $A$ and $B$ are elliptic in the decomposition $L=L_1\ast_{S_0} L_2$, and since $z$ has a translation length equal to 2, one can suppose without loss of generality that $A\subset L_1$ et $B\subset L_2$.

Recall that the morphism $\rho_n$ coincides with the identity on $A$ and with the inner automorphism $\mathrm{ad}(\delta_n)$ on $B$, with $\delta_n\in D(N_{\Gamma}(C))=\langle z^{2K!}\rangle$. Set $\delta=z^{2K!}$, so that $\delta_n=\delta^{\lambda_n}=z^{2K!\lambda_n}$ for a certain integer $\lambda_n\neq 0$ (which goes to infinity as $n$ goes to infinity.

We denote by $\tau_{n}$ the endomorphism of $\Gamma$ defined by $\tau_n=\mathrm{id}$ on $A$ and $\tau_n=\mathrm{ad}(\delta^n)$ on $B$ (well-defined since $\delta$ centralizes $C$). Let $f: \mathbb{N}\rightarrow\mathbb{N}$ be a strictly increasing function that will be specified later. The homomorphism $\tau_{f(n)}\circ \rho_n : L \rightarrow \Gamma$ coincides with $\mathrm{ad}(\tau_{f(n)}(\delta_n)\delta^{f(n)})$ on $B$. Let $\gamma_n=\tau_{f(n)}(\delta_n)\delta^{f(n)}$. An easy calculation shows that $\gamma_n=z^{\mu_n}$ where $\mu_n=2K!\lambda_n(1+4K!f(n))+2K!f(n)$. 

Let us define $r_n$ on $L$ by $r_n=\tau_{f(n)}\circ \rho_n$ on $L_1$ and $r_n=\mathrm{ad}(\gamma_n^{-1})\circ \tau_{f(n)}\circ \rho_n$ on $L_2$. Note that $r_n$ is a retraction from $L$ onto $\Gamma=A\ast_C B$ (well-defined since $\gamma_n$ centralizes $S_0$). Let us summarize the properties of $r_n$.

\begin{align*}
& r_n(a)=a \\
& r_n(b)=b \\
& r_n(z)=z \\
& {r_n}_{\vert S_0}=\tau_{f(n)}\circ {\rho_n}_{\vert S_0} \\
& {r_n}_{\vert L_1}=\tau_{f(n)}\circ {\rho_n}_{\vert L_1} \\
& {r_n}_{\vert L_2}=\mathrm{ad}(\gamma_n^{-1})\circ \tau_{f(n)}\circ {\rho_n}_{\vert S_0}
\end{align*}

To conclude, we will prove that the sequence $(r_n : L \rightarrow \Gamma)$ is discriminating provided that the sequence $(f(n))$ is properly chosen. Let us consider the decompositions \[L=L_1\ast_{S_1}\ast S\ast_{S_2} L_2 \ \ \ \ \text{ and } \ \ \ \Gamma=A\ast_{C_1}\ast N\ast_{C_2} B,\]where $C_1$ and $C_2$ are overgroups of $C$ of index 2. More precisely, $C_1=\langle C, a\rangle$ and $C_2=\langle C, b\rangle$ where $a$ and $b$ denote two elements of order 2 modulo $C$ such that $z=ba$.

The sequences $({r_n}_{\vert L_1})$ and $({r_n}_{\vert L_2})$ are both discriminating since the sequence $(\rho_n)$ is discriminatig and since the homomorphisms $\tau_{f(n)}$ are injective. We claim that the sequence $({r_n}_{\vert S})$ is discriminating as well. Let $F\subset S\setminus \lbrace 1\rbrace$ be a finite set. For simplicity, we suppose that $F=\lbrace s\rbrace$, the proof being identical in the case where $F$ has more than one element. This element $s$ can be written as $s=z^ks_0a^{\varepsilon} \mod C$ with $k\in\mathbb{Z}$, $s_0\in S_0$ and $\varepsilon\in\lbrace 0,1\rbrace$, where $a$ is an element of order 2 modulo $C$ such that $S_1=\langle S_0,a\rangle$. We distinguish two cases. If $\varepsilon=1$, then $szs^{-1}=z^{-1}\mod C$, and this relation is preserved by $r_n$, so $r_n(s)$ is non-trivial. If $\varepsilon=0$, either $s_0$ has finite order, in which case there is nothing to be done, or $s_0$ has infinite order, in which case $\rho_n(s_0)$ has infinite order for $n$ large enough. Then $\rho_n(s_0)=z^{\ell_n}\mod C$ with $\ell_n\neq 0$, so $r_n(s_0)=\tau_{f(n)}\circ \rho_n(s_0)=z^{\ell_n(1+4K!f(n))}\mod C$. For $f(n)$ large enough, the element $r_n(s)$ has infinite order. Hence, the sequence $({r_n}_{\vert S})$ is discriminating. It remains to prove that $(r_n)$ is discriminating.


Let $ x \in L $ be a non-trivial element. We will prove that $ r_n (x) $ is non-trivial for $ n $ large enough, by appropriately choosing the sequence of integers $(f (n))$. If $ x $ belongs to a conjugate of one of the vertex groups of the graph of groups $ L = L_1 \ast_{S_1} \ast S \ast_{S_2} L_2 $, then $ r_n ( x) \neq 1$ by the previous paragraph. From now on, we assume that $ x $ is not elliptic in this decomposition. Let us write $ x $ as a non-trivial reduced word $ x_1s_1x_2s_2x_3 \cdots $ in the graph of groups $ L = L_1 \ast_{S_1} \ast S \ast_{S_2} L_2 $, with $ x_i \in L_1 $ or $ L_2 $ and $ s_i \in S $. By definition of a reduced word, the following three conditions hold.
\begin{itemize}
\item[$\bullet$]$x_i$ does not belong to $S_1=\langle S_0,a\rangle$ or $S_2=\langle S_0,b\rangle$.
\item[$\bullet$]If $x_i$ and $x_{i+1}$ are both in $L_1$, then $s_i$ does not belong to $S_1$.
\item[$\bullet$]If $x_i$ and $x_{i+1}$ are both in $L_2$, then $s_i$ does not belong to $S_2$.
\end{itemize}

We will prove that $ r_n (x)$ can be written as a non-trivial reduced word in the graph of groups $ \Gamma = A \ast_{C_1} N \ast_{C_2} B $, which implies that $r_n(x)\neq 0$. Note first that $ r_n (s_i) $ belongs to $ N $ since $ S = N_{L} (C) $. Each element $ r_n (x_i) $ can be decomposed as a reduced word $w_i$ in the graph of groups $ \Gamma = A \ast_{C_1} N \ast_{C_2} B $. Let us consider the concatenation of these reduced words $ w = w_1r_n (s_1) w_2r_n (s_2) w_3 \cdots$. We will prove that $ w $ is (almost) a reduced word in the decomposition $ \Gamma = A \ast_{C_1} N \ast_{C_2} B $. The subwords of $ x $ (viewed as a non-trivial reduced word in $L_1 \ast_{S_1} \ast S \ast_{S_2} L_2 $) are of one of the following three forms.
\begin{itemize}
\item[$\bullet$]Case I: $x_is_ix_{i+1}$ with $x_i,x_{i+1}\in L_1$ (or $L_2$).
\item[$\bullet$]Case II: $x_is_ix_{i+1}$ with $x_i\in L_1$ and $x_{i+1}\in L_2$ (or $x_i\in L_2$ and $x_{i+1}\in L_1$).
\item[$\bullet$]Case III: $s_ix_is_{i+1}$ with $x_i\in L_1$ (or $L_2$).
\end{itemize}
In each case, we will see that the corresponding subword $w_ir_n(s_i)w_{i+1}$ or $r_n(s_i)w_{i}r_n(s_{i+1})$ of $w$ is (almost) reduced.

\vspace{2mm}

\textbf{Case I.} Let $x_is_ix_{j}$ be a subword of $x$ with $s_i\in S$ and $x_i,x_{j}\in L_1$, where $j=i+1$ (the case $x_i,x_j\in L_2$ is identical). Since $x$ is reduced, $s_i$ does not belong to $S_1=\langle S_0,a\rangle$, so $s_i=z^ks_0a^{\varepsilon}\mod C$ with $k\neq 0$, $s_0\in S_0$ and $\varepsilon\in\lbrace 0,1\rbrace$. We have \[r_n(s_i)=z^{k+\ell_n(1+4K!f(n))}a^{\varepsilon}\mod C,\] where $\ell_n$ denotes the integer such that $\rho_n(s_0)=z^{\ell_n}$ (modulo $C$). 

Then, let us decompose $r_n(x_i)$ and $r_n(x_j)$ as reduced words in $\Gamma=A\ast_{C_1} N\ast_{C_2} B$, where $C_1=\langle C,a\rangle$ et $C_2=\langle c,b\rangle$. First, we decompose $\rho_n(x_i)$ as a reduced word whose first and last letters belong to $N$. This word ends with $y_{i,n}n_{i,n}$ where $n_{i,n}\in N$ and $y_{i,n}\in A\setminus C_1$ or $y_{i,n}\in B\setminus C_2$. The element $n_{i,n}$ can be written as $n_{i,n}=z^{k_{i,n}}a^{\varepsilon_{i,n}}\mod C$ with $k_{i,n}\in\mathbb{Z}$ and $\varepsilon_{i,n}\in\lbrace 0,1\rbrace$. We decompose $\rho_n(x_{j})$ in the same way. The corresponding reduced word begin with $n_{j,n}y_{j,n}$, where $n_{j,n}\in N$ and $y_{j,n}\in A\setminus C_1$ ou $B\setminus C_2$. Again, $n_{j,n}$ can be written as $n_{j,n}=z^{k_{j,n}}a^{\varepsilon_{j,n}}\mod C$ where $k_{j,n}\in\mathbb{Z}$ and $\varepsilon_{j,n}\in\lbrace 0,1\rbrace$. So we have \[\tau_{f(n)}(n_{i,n})=z^{k_{i,n}(1+4K!f(n))}a^{\varepsilon_{i,n}} \mod C \ \ \ \text{ and } \ \ \ \tau_{f(n)}(n_{j,n})=z^{k_{j,n}(1+4K!f(n))}a^{\varepsilon_{j,n}} \mod C.\]


\vspace{2mm}

\emph{Subcase i.} Suppose that $y_{i,n}$ and $y_{j,n}$ are both in $A$. At the junction of the concatenation of the reduced forms of $r_n(x_i)$, $r_n(s_i)$ and $r_n(x_j)$, we see the subword \[y_{i,n}\omega_ny_{j,n}, \ \text{with} \  \omega_n = \tau_{f(n)}(n_{i,n})r_n(s_i)\tau_{f(n)}(n_{j,n}).\]
We have to prove that one can choose $f(n)$ so that $\omega_n$ does not belong to $C_1=\langle C,a\rangle$. Any easy calculation shows that \[\omega=z^{R_n}a^{\varepsilon_{i,n}+\varepsilon+\varepsilon_{j,n}},\]where $R_n$ is of the form \[R_n = \pm k + (1+4K!f(n))D\]for some integer $D$. Since the integer $k$ is not zero, one can always choose $f(n)$ in such a way that $R_n$ is non-zero. This concludes the first subcase.


\vspace{2mm}

\emph{Subcase ii.} Suppose that $y_{i,n}$ and $y_{j,n}$ are both in $B$. Then \[\tau_{f(n)}(y_{i,n})=\delta^{f(n)}y_{i,n}\delta^{-f(n)} \ \ \ \text{and} \ \ \ \tau_{f(n)}(y_{j,n})=\delta^{f(n)}y_{j,n}\delta^{-f(n)}.\] 

By concatenating $r_n(x_i)$, $r_n(s_i)$ and $r_n(x_j)$, we see the subword \[y_{i,n}\delta^{-f(n)}z^{R_n}a^{\varepsilon_{i,n}+\varepsilon+\varepsilon_{j,n}}\delta^{f(n)}y_{j,n}=y_{i,n}z^{R'_n}a^{\varepsilon_{i,n}+\varepsilon+\varepsilon_{j,n}}y_{j,n},\]where \[R'_n=R_n-2K!f(n)(1+(-1)^{\varepsilon_{i,n}+\varepsilon+\varepsilon_{j,n}+1}).\]

Let us prove that the subword $z^{R'_n}a^{\varepsilon_{i,n}+\varepsilon+\varepsilon_{j,n}}$ does not belong to $C_2=\langle C,b\rangle$. First, observe that this subword belongs to $C_2$ if and only if one of the following two conditions holds:
\begin{itemize}
\item[$\bullet$]$R'_n=0$ and $\varepsilon_{i,n}+\varepsilon+\varepsilon_{j,n}=0\mod 2$, in which case $z^{R'_n}a^{\varepsilon_{i,n}+\varepsilon+\varepsilon_{j,n}}=1 \mod C$, or
\item[$\bullet$]$R'_n=1$ and $\varepsilon_{i,n}+\varepsilon+\varepsilon_{j,n}=1\mod 2$, in which case $z^{R'_n}a^{\varepsilon_{i,n}+\varepsilon+\varepsilon_{j,n}}=b\mod C$.
\end{itemize}

In the first case $R'_n=R_n$ and we saw in the first subcase that $f(n)$ can be chosen large enough so that $R_n\neq 0$ (because $k\neq 0$). In the second case, \begin{align*}
R'_n &= R_n-4K!f(n)\\
 & = \pm k + (1+4K!f(n))(k_{i,n}\pm \ell_n\pm k_{j,n})-4K!f(n)
\end{align*}
so \[R'_n=1 \Leftrightarrow  \pm k + (1+4K!f(n))(k_{i,n}\pm \ell_n\pm k_{j,n}-1)=0.\]
But $k$ is non-zero, so $f(n)$ can be chosen sufficiently large so that the previous equality does not hold.

\vspace{2mm}

\emph{Subcase iii.} If $y_{i,n}\in A$ and $y_{j,n}\in B$, or $y_{i,n}\in B$ and $y_{j,n}\in A$, there is nothing to be done.

\vspace{2mm}

\textbf{Case II.} Let us consider the subword $x_is_ix_{j}$ of $x$ (viewed as word) with $x_i\in L_1$, $s_i\in S$ and $x_j\in L_2$ (the case where $x_i\in L_2$ and $x_j\in L_1$ can be tackled in exactly the same way), where $j=i+1$. The element $s_i$ can be decomposed as $s_i=z^ks_0a^{\varepsilon}\mod C$ with $k\in\mathbb{Z}$ (note that $k$ may be zero here), $s_0\in S_0$ and $\varepsilon\in\lbrace 0,1\rbrace$. So we have \[r_n(s_i)=z^{k+\ell_n(1+4K!f(n))}a^{\varepsilon}\mod C,\] where $\ell_n$ is the integer such that $\rho_n(s_0)=z^{\ell_n}$. Then, as above, we decompose $r_n(x_i)$ and $r_n(x_j)$ as reduced words in the decomposition $\Gamma=A\ast_{C_1} N\ast_{C_2} B$. With the same notations, we have \[\tau_{f(n)}(n_{i,n})=z^{k_{i,n}(1+4K!f(n))}a^{\varepsilon_{i,n}} \mod C \ \ \ \text{ et } \ \ \ \tau_{f(n)}(n_{j,n})=z^{k_{j,n}(1+4K!f(n))}a^{\varepsilon_{j,n}} \mod C.\]As in Case I, there are three subcases because $y_{i,n}$ and $y_{j,n}$ may be in $A$ or in $B$. We will suppose that $y_{i,n}\in A$ and $y_{j,n}\in A$. The reader can check that the three other cases can be solved in the same manner.

By concatenating the words corresponding to $r_n(x_i)$, $r_n(s_i)$ and $r_n(x_j)$, we see the subword \[y_{i,n}\tau_{f(n)}(n_{i,n})r_n(s_i)z^{-\mu_n}\tau_{f(n)}(n_{j,n})y_{j,n}=y_{i,n}z^{R_n}a^{\varepsilon_{i,n}+\varepsilon+\varepsilon_{j,n}}y_{j,n},\]where{\small \begin{align*}
R_n&=k_{i,n}(1+4K!f(n))+(-1)^{\varepsilon_{i,n}}(k+\ell_n(1+4K!f(n)))+(-1)^{\varepsilon_{i,n}+\varepsilon}(k_{j,n}(1+4K!f(n))-\mu_n)\\
 & = \pm k + (1+4K!f(n))(k_{i,n}\pm \ell_n\pm k_{j,n})\pm\mu_n \\
 & = \pm k + 2K!f(n)(2\alpha_n\pm 4K!\lambda_n\pm 1)+(\alpha_n\pm 2K!\lambda_n).
\end{align*}}%
where $\alpha_n=k_{i,n}\pm \ell_n\pm k_{j,n}$ and $\mu_n=(1+4K!f(n))2K!\lambda_n+2K!f(n)$. The integer $\mu_n$ comes from the fact that $r_n=\mathrm{ad}(z^{-\mu_n})\circ \tau_{f(n)}\circ \rho_n$ on $L_2$). We claim that $f(n)$ can be chosen in such a way that $z^{R_n}a^{\varepsilon_{i,n}+\varepsilon+\varepsilon_{j,n}}\notin C_1=\langle C,a\rangle$. Indeed, for $f(n)$ large enough, $R_n\neq 0$ since $2\alpha_n\pm 4K!\lambda_n\pm 1$ is odd so non-zero.

\vspace{2mm}

\textbf{Case III.} Consider a subword $s_ix_is_{j}$ with $s_i,s_j\in S$ and $x_i\in L_1$ (or $x_i\in L_2$), where $j=i+1$. 

Since the word representing $x$ is reduced, $x_i$ does not belong to $S$, so $\rho_n(x_i)$ does not belong to $N$. Hence, the decomposition of $\rho_n(x)$ as a reduced word in $\Gamma=A\ast_{C_1} N\ast_{C_2} B$ is of the form $n_{1}y_{1}\cdots y_{r}n_{r}$ with $r\geq 1$, $y_{k}\in A\setminus C$ or $B\setminus C$, and $n_{1},n_{r}\in N$ (maybe zero), so there is nothing to be done.\end{proof}

\vspace{2mm}


We will use Lemma \ref{sacerdote3} to prove Theorem \ref{mainsac}.

\subsubsection{Proof of Theorem \ref{mainsac}}Let $G$ be a hyperbolic group. Suppose that $G$ splits over a finite subgroup $C$ whose normalizer $N$ is virtually cyclic. Let $\Gamma=G\ast_N N'$ be a legal small extension of $G$. We shall prove that $\mathrm{Th}_{\forall\exists}(G)= \mathrm{Th}_{\forall\exists}(\Gamma)$. By Lemma \ref{symétrie} below, it suffices to prove that $\mathrm{Th}_{\forall\exists}(G)\subset \mathrm{Th}_{\forall\exists}(\Gamma)$. 

\begin{lemme}\label{symétrie}Let $G$ be a hyperbolic group. Suppose that $G$ splits over a finite subgroup $C$ whose normalizer $N$ is virtually cyclic. Let $\Gamma=G\ast_N N'$ be a legal small extension of $G$. Then $G$ is a legal small extension of $\Gamma$ (viewed as abstract groups).\end{lemme}

\begin{proof}There exists an injective twist $\varphi : \Gamma \rightarrow G \subset \Gamma$ (one can take any homomorphism of the small test sequence $(\varphi_n : \Gamma\rightarrow G)_{n\in\mathbb{N}}$ whose existence was proved above). Let $N''=N_{\varphi(\Gamma)}(C)$. The group $G$ can be decomposed as $G=\varphi(\Gamma)\ast_{N''} N$. The inclusion of $N''$ into $N$ is legal, an there exists a legal embedding of $N$ into $N''$ (for example $\iota'\circ \iota$, with the same notations as above).\end{proof}





Before proving Theorem \ref{mainsac}, we prove that Lemma \ref{sacerdote3} remains true, more generally, if one replaces the system of equations and inequations by a boolean combination of equations and inequations.


\begin{co}\label{sacerdote4}Let $\Gamma$ be a legal small extension of $G$, and let $(\varphi_n : \Gamma \rightarrow G)_{n\in\mathbb{N}}$ be a small test sequence. Let \[\bigvee_{k=1}^N(\Sigma_k(\bm{x},\bm{y})=1 \ \wedge \ \Psi_k(\bm{x},\bm{y})\neq 1)\] be a disjunction of systems of equations and inequations, where $\bm{x}$ is a $p$-tuple of variables and $\bm{y}$ is a $q$-tuple of variables. Let $\bm{\gamma}\in\Gamma^p$. Suppose that $G$ satisfies the following condition: for every integer $n$, there exists $\bm{g}_n\in G^q$ such that \[\bigvee_{k=1}^N(\Sigma_k(\varphi_n(\bm{\gamma}),\bm{g}_n)=1 \ \wedge \ \psi_k(\varphi_n(\bm{\gamma}),\bm{g}_n)\neq 1).\] Then there exists $\bm{\gamma}'\in\Gamma^q$ such that \[\bigvee_{k=1}^N(\Sigma_k(\bm{\gamma},\bm{\gamma'})=1 \ \wedge \ \Psi_k(\bm{\gamma},\bm{\gamma'})\neq 1).\]\end{co}

\begin{proof}Up to extracting a subsequence of $(\varphi_n)_{n\in\mathbb{N}}$ (which is still a test sequence), one can assume that there exists an integer $1\leq k\leq N$ such that, for every integer $n$, there exits $\bm{g}_n\in G^q$ such that $\Sigma_k(\varphi_n(\bm{\gamma}),\bm{g}_n)=1$ and $\Psi_k(\varphi_n(\bm{\gamma}),\bm{g}_n)\neq 1$. Proposition \ref{sacerdote3} applies and asserts the existence of a tuple $\bm{\gamma}'\in \Gamma^q$ satisfying $\Sigma_k(\bm{\gamma},\bm{\gamma'})=1$ and $\Psi_k(\bm{\gamma},\bm{\gamma'})\neq1$, which concludes the proof.\end{proof}

We now prove Theorem \ref{mainsac}.

\begin{proof}
According to Lemma \ref{symétrie}, it suffices to prove that $\mathrm{Th}_{\forall\exists}(G)\subset \mathrm{Th}_{\forall\exists}(\Gamma)$. Let $\theta$ be a $\forall\exists$-sentence such that $G \models\theta$. Let us prove that $\Gamma\models \theta$. The sentence $\theta$ has the following form:\[\theta:\forall \bm{x} \exists \bm{y}  \bigvee_{k=1}^N(\Sigma_k(\bm{x},\bm{y})=1 \ \wedge \ \Psi_k(\bm{x},\bm{y})\neq 1),\]where $\bm{x}$ is a $p$-tuple of variables and $\bm{y}$ is a $q$-tuple of variables. Let $\bm{\gamma}$ a $p$-tuple of elements of $\Gamma$. For every integer $n$, there exists $\bm{g}_n\in G^q$ such that \[G\models \bigvee_{k=1}^N(\Sigma_k(\varphi_n(\bm{\gamma}),\bm{g}_n)=1 \ \wedge \ \Psi_k(\varphi_n(\bm{\gamma}),\bm{g}_n)\neq 1).\] By Lemma \ref{sacerdote4}, there exists $\bm{\gamma'}\in \Gamma^q$ such that \[\Gamma\models\bigvee_{k=1}^N(\Sigma_k(\bm{\gamma},\bm{\gamma'})=1 \ \wedge \ \Psi_k(\bm{\gamma},\bm{\gamma'})\neq 1).\]Hence, $\Gamma\models\theta$.\end{proof}

\section{Proof of $(5)\Rightarrow (1)$}\label{5 implique 1}

We are now ready to prove the implication $(5)\Rightarrow (1)$ of Theorem \ref{principal}. In fact, we prove a stronger result, since we only assume that $G$ and $G'$ are hyperbolic.

\begin{te}[Implication $(5)\Rightarrow (1)$ of Theorem \ref{principal}]Let $G,G'$ be two hyperbolic groups. Suppose that there exist two multiple legal extensions $\Gamma$ and $\Gamma'$ of $G$ and $G'$ respectively, such that $\Gamma\simeq \Gamma'$. Then $\mathrm{Th}_{\forall\exists}(G)=\mathrm{Th}_{\forall\exists}(G')$.\end{te}

\begin{proof}By definition of a multiple legal extension, there exists a finite sequence of groups $G=G_0\subset G_1\subset \cdots \subset G_n\simeq \Gamma$ where $G_{i+1}$ is a legal large or small extension of $G_i$ in the sense of Definitions \ref{legal} or \ref{legal2}, for every integer $0\leq i\leq n-1$. According to Theorems \ref{legalte} and \ref{legal2te}, we have $\mathrm{Th}_{\forall\exists}(G_i)=\mathrm{Th}_{\forall\exists}(G_{i+1})$, for every $0\leq i\leq n-1$. Thus, $G$ and $\Gamma$ have the same $\forall\exists$-theory. Similarly, $G'$ and $\Gamma'$ have the same $\forall\exists$-theory. But $\Gamma$ and $\Gamma'$ are isomorphic, so $\mathrm{Th}_{\forall\exists}(\Gamma)=\mathrm{Th}_{\forall\exists}(\Gamma')$. Hence, $G$ and $G'$ have the same $\forall\exists$-theory.\end{proof}

\section{Proof of $(4)\Rightarrow (5)$}\label{4 implique 5}

This section is dedicated to a proof of the implication $(4)\Rightarrow (5)$ of Theorem \ref{principal}, that is the following result.

\begin{prop}[Implication $(4)\Rightarrow (5)$ of Theorem \ref{principal}]\label{vandernash}Let $G$ and $G'$ be two virtually free groups. Suppose that there exist two strongly special homomorphisms $\varphi : G \rightarrow G'$ and $\varphi' : G' \rightarrow G$. Then, there exist two multiple legal extensions $\Gamma$ and $\Gamma'$ of $G$ and $G'$ respectively, such that $\Gamma\simeq \Gamma'$.\end{prop}

While in the previous sections we stated and proved results in the general context of hyperbolic groups (with torsion), the present section is specific to virtually free groups. However, we believe that the construction of multiple legal extensions should play a role in a classification of hyperbolic groups up to elementary equivalence as well.

There are, in brief, three increasing levels of complexity in the proof of Proposition \ref{vandernash} above.

\begin{enumerate}
\item We assume that all edge groups in reduced Stallings splittings of $G$ and $G'$ are equal. In other words, we suppose that there is only one cylinder in these splittings (see paragraph \ref{tree} for the definition of a cylinder). We refer the reader to Corollary \ref{f-by-f0}.
\item We assume that all edge groups in reduced Stallings splittings $T$ and $T'$ of $G$ and $G'$ have the same cardinality. The importance of the previous point appears when one considers the trees of cylinders $T_c$ and $T'_c$ of $T$ and $T'$ (see paragraph \ref{tree} for the definition of the tree of cylinders). See Proposition \ref{2implique3facile}.
\item In the general case, different cardinalities of edge groups may coexist in reduced Stallings splittings of $G$ and $G'$. The proof of Proposition \ref{vandernash} is by induction on the number of edges in reduced Stallings splittings of $G$ and $G'$. By carefully collapsing certain edges, we can assume that there is only one cardinality of edge groups in the splittings we consider, and we can use the same techniques as in the second point above.
\end{enumerate} 

As usual, we denote by $K_G$ the maximal order of a finite subgroup of $G$. Note that, under the hypotheses of Proposition \ref{vandernash}, the integers $K_G$ and $K_{G'}$ are equal, because strongly special morphisms are injective on finite subgroups. We define $K:=K_G=K_{G'}\geq 1$. Note that this integer is preserved by legal extensions. In the sequel, we assume that the maximal order of a finite subgroup in all virtually free groups we consider is at most equal to $K$.

Before proving Proposition \ref{vandernash}, we need to introduce new definitions and to prove some lemmas.

\subsection{Preliminaries} 

\subsubsection{Strongly special pair of homomorphisms}

\begin{de}\label{equivalencerelnash}Let $G$ and $G'$ be two groups. Given two morphisms $\varphi,\varphi'\in\mathrm{Hom}(G,G')$, we use the notation $\varphi\sim\varphi'$ if, for every finite subgroup $C$ of $G$, there exists an element $g'\in G'$ such that $\varphi'=\mathrm{ad}(g')\circ \varphi$ on $C$.\end{de}

\begin{de}\label{defstrongpair}Let $G$ and $G'$ be two groups. Let $\varphi : G \rightarrow G'$ and $\varphi' : G' \rightarrow G$ be two homomorphisms. The pair $(\varphi,\varphi')$ is said to be strongly special if the following two conditions hold.
\begin{enumerate}
\item $\varphi$ and $\varphi$ are strongly special.
\item $\varphi'\circ \varphi\sim \mathrm{id}_G$ and $\varphi\circ \varphi' \sim \mathrm{id}_{G'}$.
\end{enumerate}
\end{de}



Note that if $\varphi : G \rightarrow G'$ and $\varphi : G' \rightarrow G$ are both strongly special, then $\varphi'\circ \varphi$ is strongly special. As a consequence, taking $\psi=\varphi$ and $\psi'=\varphi'\circ (\varphi\circ \varphi')^k$ for a suitable $k$, one gets the following result.

\begin{lemme}\label{strong}Let $G$ and $G'$ be two groups. Let $\varphi : G \rightarrow G'$ and $\varphi' : G'\rightarrow G$ be two strongly special homomorphisms. Suppose that $G$ and $G'$ have finitely many conjugacy classes of finite subgroups. Then there exists a strongly special pair $(\psi,\psi')$.
\end{lemme}

According to the previous lemma, in order to prove Proposition \ref{vandernash}, it suffices to prove the following result.

\begin{prop}\label{2implique3}Let $G$ and $G'$ be two virtually free groups. Suppose that there exists a strongly special pair of homomorphisms $(\varphi : G \rightarrow G', \varphi' : G' \rightarrow G)$. Then there exist two multiple legal extensions $\Gamma$ and $\Gamma'$ of $G$ and $G'$ respectively, such that $\Gamma\simeq \Gamma'$.\end{prop}

\subsubsection{Smallest order of an edge group in a reduced Stallings splitting}\label{mGG'}

\begin{de}\label{mGG'}Given an infinite virtually free group $G$, we denote by $m(G)$ the smallest order of an edge group in a reduced Stallings splitting of $G$. Note that this integer does not depend on a particular reduced Stallings splitting of $G$, because conjugacy classes of edge groups are the same in all reduced Stallings splittings of $G$, since one can pass from a reduced Stallings splitting to another by a sequence of slide moves. If $G$ and $G'$ are two virtually free groups, we define $m_{G,G'}=\min(m(G),m(G'))$.\end{de}



\subsubsection{$m$-splittings}

\begin{de}Let $m\geq 1$ be an integer. Let $G$ be a virtually free group. A \emph{$m$-splitting} of $G$ as a graph of groups is a non-trivial splitting of $G$ over subgroups of order exactly $m$.\end{de}

\begin{lemme}\label{JE SAIS PAS!!}Let $m\geq 1$ be an integer. Let $G$ be a virtually free group, and let $T$ be a reduced $m$-splitting of $G$. Suppose that $T$ has a vertex group of order exactly $m$. Then, $G$ is finite-by-free.
\end{lemme}

\begin{proof}Suppose that there exists a vertex $v$ of $T$ such that $\vert G_v\vert = m$. Since $T/G$ is a reduced splitting of $G$ over edge groups of order $m$, the existence of a vertex group of order exactly $m$ implies that the underlying graph of the graph of groups $T/G$ has only one vertex, i.e.\ is a bouquet of circles. Hence, all edge groups and vertex groups of $T$ are equal to $G_v$. As a consequence, $G_v$ is the unique maximal finite normal subroup of $G$, and the quotient group $G/G_v$ is the fundamental group of a bouquet of circles, i.e.\ a free group. Hence, the group $G$ is $G_v$-by-free, with $G_v$ finite.\end{proof}

\subsubsection{Strongly $(>m)$-special homomorphisms}

We need to slightly weaken the definitions of a strongly special homomorphism and of a strongly special pair of homomorphisms. The following definitions are suitable for proofs by induction.

\begin{de}\label{besoinreffff}Let $G$ and $G'$ be virtually free groups. Let $m\geq 0$ be an integer. A homomorphism $\varphi : G \rightarrow G'$ is said to be a \emph{strongly $(>m)$-special homomorphism} if it satisfies the following four properties:
\begin{enumerate}
\item $\varphi$ is injective on finite subgroups;
\item if $C_1$ and $C_2$ are two non-conjugate finite subgroups of $G$ of order $>m$, then $\varphi(C_1)$ and $\varphi(C_2)$ are non-conjugate in $G'$;
\item if $C$ is a finite subgroup of $G$ of order $>m$ whose normalizer $N_G(C)$ is non-elementary, then the normalizer $N_{G'}(\varphi(C))$ is non-elementary and \[\varphi(E_G(N_G(C)))=E_{G'}(N_{G'}(\varphi(C))).\] In particular, if the finite group $E_G(N_G(C))$ is equal to $C$, then \[E_{G'}(N_{G'}(\varphi(C)))=\varphi(C);\]
\item if $C$ is a finite subgroup of $G$ of order $>m$ whose normalizer is virtually cyclic infinite maximal, then $N_{G'}(\varphi(C))$ is virtually cyclic infinite maximal, and the restriction $\varphi_{\vert N_G(C)}:N_G(C)\rightarrow N_{G'}(\varphi(C))$ is $K$-nice, with $K$ the maximal order of a finite subgroup of $G$ (see Definition \ref{special0}).
\end{enumerate}
We define \emph{strongly $(\geq m)$-special homomorphisms} in the same way.
\end{de}

\begin{rque}A homomorphism is strongly special (see Definition \ref{strongmor}) if and only if it is strongly $(>0)$-special. \end{rque}

\begin{de}\label{strongmpair}Let $G$ and $G'$ be virtually free groups. Let $\varphi : G \rightarrow G'$ and $\varphi' : G' \rightarrow G$ be two homomorphisms. The pair $(\varphi,\varphi')$ is said to be a \emph{strongly $(>m)$-special pair} if the following three conditions hold:
\begin{enumerate}
\item $\varphi$ and $\varphi'$ are strongly $(>m)$-special (see Definition \ref{besoinreffff});
\item $\varphi'\circ \varphi\sim \mathrm{id}_G$, which means that for every finite subgroup $C\subset G$, there exists an element $g\in G$ such that $\varphi'\circ \varphi$ and $\mathrm{ad}(g)$ coincide on $C$;
\item $\varphi\circ \varphi' \sim \mathrm{id}_{G'}$. 
\end{enumerate}
We define \emph{strongly $(\geq m)$-special pairs} in the same way. 
\end{de}

\subsubsection{Preservation of specialness}

The following lemma shows that the property of being strongly $(\geq m)$-special is preserved by composition.

\begin{lemme}[Composition of strongly $(\geq m)$-special homomorphisms]\label{onpeutcomposer}Let $G,G'$ and $G''$ be virtually free groups. Let $m\geq 1$ be an integer. Let $\varphi : G\rightarrow G'$ and $\varphi' : G'\rightarrow G''$ be homomorphisms. Suppose that $\varphi$ and $\varphi'$ are strongly $(\geq m)$-special. Then $\varphi'\circ \varphi$ is strongly $(\geq m)$-special.
\end{lemme}

\begin{rque}\label{9juillet}We also prove that if $\varphi$ and $\varphi'$ satisfy the first three conditions of Definition \ref{besoinreffff}, then $\varphi'\circ \varphi$ satisfies the first three conditions of Definition \ref{besoinreffff}.
\end{rque}

\begin{proof}There are four conditions that need to be verified.

\smallskip

\emph{Condition 1.} Since $\varphi$ and $\varphi'$ are injective on finite subgroups of $G$ and $G'$, $\varphi'\circ \varphi$ is injective on finite subgroups of $G$ as well.

\smallskip

\emph{Condition 2.} If $C_1$ and $C_2$ are two non-conjugate finite subgroups of $G$ of order $\geq m$, then $\varphi(C_1)$ and $\varphi(C_2)$ are two non-conjugate subgroups of $G'$ of order $\geq m$, hence $\varphi'\circ \varphi(C_1)$ and $\varphi'\circ \varphi(C_2)$ are non-conjugate in $G''$.

\smallskip

\emph{Condition 3.} If $C$ is a finite subgroup of $G$ of order $\geq m$ whose normalizer in $G$ is non-elementary, then $N_{G'}(\varphi(C))$ is non-elementary and $\varphi(E_G(N_G(C)))$ is equal to $E_{G'}(N_{G'}(\varphi(C)))$, since $\varphi$ is strongly $(\geq m)$-special. Note that $\vert \varphi(C)\vert =\vert C\vert \geq m$. Hence, since $\varphi'$ is strongly $(\geq m)$-special, the group $N_{G''}(\varphi'(\varphi(C)))$ is non-elementary and the following equality holds: $\varphi'(E_{G'}(N_{G'}(\varphi(C))))=E_{G''}(N_{G''}(\varphi'(\varphi(C))))$.

\smallskip

\emph{Condition 4.} If $C$ is a finite subgroup of $G$ of order $\geq m$ whose normalizer in $G$ is virtually cyclic infinite maximal, then $N_{G'}(\varphi(C))$ is virtually cyclic infinite maximal, and the restriction of $\varphi$ to $N_G(C)$ is $K$-nice (see Definition \ref{special0}), because $\varphi$ is strongly $(\geq m)$-special. Since $\varphi'$ is strongly $(\geq m)$-special, the group $N_{G''}(\varphi'(\varphi(C)))$ is virtually cyclic maximal, and the restriction of $\varphi'$ to $N_{G'}(\varphi(C))$ is $K$-nice. Hence, the restriction of $\varphi'\circ \varphi$ to $N_G(C)$ is $K$-nice, as the composition of $K$-nice homomorphisms.\end{proof}

We now prove that the first three conditions of Definition \ref{besoinreffff} are preserved by the equivalence relation $\sim$ (see Definition \ref{equivalencerelnash}).

\begin{lemme}\label{onpeutapprox}Let $G$ and $G'$ be two virtually free groups, and let $\varphi,\psi : G \rightarrow G'$ be two homomorphisms. Suppose that $\psi\sim\varphi$. If $\varphi$ satisfies conditions 1 to 3 of Definition \ref{besoinreffff}, then $\psi$ satisfies conditions 1 to 3 of Definition \ref{besoinreffff}.
\end{lemme}


\begin{proof}There are three conditions that need to be verified. 

\smallskip

\emph{Condition 1.} By definition of $\sim$, the homomorphisms $\varphi$ and $\psi$ coincide up to conjugacy on finite subgroups. Since $\varphi$ is injective on finite subgroups of $G$, the homomorphism $\psi$ is injective on finite subgroups as well.

\smallskip

\emph{Condition 2.} If $C_1$ and $C_2$ are non-conjugate finite subgroups of $G$ of order $\geq m$, then $\varphi(C_1)$ and $\varphi(C_2)$ are non-conjugate in $G'$, because $\varphi$ is $(\geq m)$-special. In addition $\psi(C_1)$ is conjugate to $\varphi(C_1)$ and $\psi(C_2)$ is conjugate to $\varphi(C_2)$. Hence, $\psi(C_1)$ and $\psi(C_2)$ are non-conjugate in $G$.

\smallskip

\emph{Condition 3.} If $C$ is a finite subgroup of $G$ of order $\geq m$ whose normalizer in $G$ is non-elementary, then $N_{G'}(\varphi(C))$ is non-elementary and $\varphi(E_G(N_G(C)))$ is equal to $E_{G'}(N_{G'}(\varphi(C)))$, since $\varphi$ is strongly $(\geq m)$-special. The group $\psi(C)$ being conjugate to $\varphi(C)$, its normalizer is non-elementary and the following equality holds: $\psi(E_{G}(N_{G}(C)))=E_{G'}(N_{G'}(\psi(C)))$.\end{proof}



\subsubsection{Preliminary lemmas about legal extensions}

\begin{lemme}\label{lemme2ju}Let $G$ be a non-elementary virtually free group. Let $C$ be a finite subgroup of $G$. Suppose that the group $\Gamma=\langle G,t \ \vert \ \mathrm{ad}(t)_{\vert C}=\mathrm{id}_C\rangle$ is a legal large extension of $G$. Let $i$ denote the inclusion of $G$ into $\Gamma$, and let $r : \Gamma \twoheadrightarrow G$ denote the retraction defined by $r_{\vert G}=\mathrm{id}_G$ and $r(t)=1$. Then $r$ and $i$ are strongly special, and $i\circ r\sim\mathrm{id}_{\Gamma}$. 
\end{lemme}

\begin{proof}Let $T$ be the Bass-Serre tree of the splitting $\Gamma=G\ast_{\mathrm{id}_C}$, and let $v$ denote a vertex of $T$ fixed by $G$. Note that this vertex is unique, because the infinite vertex group $G$ is not equal to the adjacent edge group $C$, which is finite. 

First, let us prove that the retraction $r$ is strongly special. There are four conditions that need to be verified.

\smallskip

\emph{Condition 1.} Since every finite subgroup of $\Gamma$ is conjugate to a subgroup of $G$, the retraction $r$ is injective on finite subgroups of $\Gamma$. 

\smallskip

\emph{Condition 2.} Let $A$ and $B$ be two non-conjugate finite subgroups of $\Gamma$. One can suppose without loss of generality that they are contained in $G$. Therefore, $r(A)=A$ and $r(B)=B$ are non-conjugate in $G$. 

\smallskip

\emph{Condition 3.} Now, let $A$ be a finite subgroup of $\Gamma$ whose normalizer $N:=N_{\Gamma}(A)$ is non-elementary. One can suppose without loss of generality that $A$ is contained in $G$. We distinguish two cases.

\smallskip

\emph{First case.} If $A$ is not contained in a conjugate of $C$ in $G$, then $v$ is the unique fixed point of $A$ in $T$. It follows that $N_{\Gamma}(A)$ fixes the vertex $v$. Hence, $N_{\Gamma}(A)=N_{\Gamma}(A)\cap G=N_{G}(A)$. This shows in particular that $N_G(r(A))$ is non-elementary. Moreover, note that $E_{\Gamma}(N_{\Gamma}(A))\supset A$ only fixes $v$. This implies that $E_{\Gamma}(N_{\Gamma}(A))=E_G(N_{\Gamma}(A))=E_G(N_G(A))$. Hence, we have $r(E_{\Gamma}(N_{\Gamma}(A)))=E_G(N_G(r(A)))$.

\smallskip

\emph{Second case.} Otherwise, one can suppose without loss of generality that $A$ is contained in $C$. Then $N_G(C)$ has a subgroup of finite index that centralizes $A$. Since $N_G(C)$ is non-elementary by definition of a large extension, the normalizer $N_G(A)=N_G(r(A))$ is non-elementary as well. Now, let us prove that $r(E_{\Gamma}(N_{\Gamma}(A)))=E_G(N_G(r(A)))$. First, we prove that $E_{\Gamma}(N_{G}(A))$ is contained in $C$. Let us observe that there exists an integer $n\geq 1$ such that, for every $g\in N_{G}(C)$, the element $g^n$ normalizes (and even centralizes) $A$. Then, recall that $E_{\Gamma}(N_G(A))$ is equal to the intersection of all $M(g)\subset \Gamma$ where $g$ runs through the set $N_{G}(A)^0$ consisting of all elements of $N_{G}(A)$ of infinite order. Hence, we have:\[E_{\Gamma}(N_G(A))=\bigcap_{g\in N_{G}(A)^0} M(g)\subset \bigcap_{g\in N_{G}(C)^0}M(g^n)\overset{(*)}{=}\bigcap_{g\in N_{G}(C)^0}M(g)=E_{\Gamma}(N_{G}(C))=C.\]Note that the equality $(*)$ follows from the fact that $M(g^n)=M(g)$ for any non-trivial integer $n$. We have proved that $E_{\Gamma}(N_G(A))$ is contained in $C$, hence in $G$. This shows that the following equality holds: 
\begin{equation}\label{eq01}
E_{\Gamma}(N_G(A))=E_G(N_G(A)).
\end{equation} 
In addition, recall that the stable letter $t$ centralizes $C$. In particular, $t$ centralizes $E_{\Gamma}(N_G(A))$. Moreover $N_G(A)$ normalizes $E_{\Gamma}(N_G(A))$, by definition of $E_{\Gamma}(N_G(A))$. Thus, the group $N_{\Gamma}(A)=\langle N_G(A),t\rangle$ normalizes $E_{\Gamma}(N_G(A))$, which implies that $E_{\Gamma}(N_G(A))$ is contained in $E_{\Gamma}(N_{\Gamma}(A))$, by definition of $E_{\Gamma}(N_{\Gamma}(A))$. The reverse inclusion is obvious since $N_G(A)$ is contained in $N_{\Gamma}(A)$. Hence, we have 
\begin{equation}\label{eq02}
E_{\Gamma}(N_{\Gamma}(A))=E_{\Gamma}(N_G(A)).
\end{equation}
By combining the two equalities (\ref{eq01}) and (\ref{eq02}), we get the equality $E_{\Gamma}(N_{\Gamma}(A))=E_G(N_G(A))$, i.e.\ $r(E_{\Gamma}(N_{\Gamma}(A))=E_G(N_G(r(A)))$, which concludes.

\smallskip

\emph{Condition 4.} Let $A$ be a finite subgroup of $\Gamma$ whose normalizer is virtually $\mathbb{Z}$ maximal. One can suppose without loss of generality that $A$ is contained in $G$. Note that $A$ is not contained in a conjugate of $C$ in $G$, otherwise $N_G(A)$ would be non-elementary (see above). Thus, the vertex $v$ is the unique fixed point of $A$ in $T$. Therefore, $N_{\Gamma}(A)$ and $N_G(A)$ are equal, which implies that $N_G(A)$ is virtually cyclic maximal in $\Gamma$. In addition, $r$ coincides with the identity on $N_{\Gamma}(A)$; in particular, it is $K$-nice. 

\smallskip

We have proved that the retraction $r$ is strongly special. Now, let us prove that the inclusion $i : G \subset \Gamma$ is strongly special. 

\smallskip

\emph{Condition 1.} The inclusion is injective on finite subgroups.

\smallskip

\emph{Condition 2.} Let $A$ and $B$ be two finite subgroups of $G$. If $B=\gamma A\gamma^{-1}$ for a certain element $\gamma\in \Gamma$, then $B=r(\gamma)Ar(\gamma)^{-1}$. 

\smallskip

\emph{Condition 3.} Let $A$ be a finite subgroup of $G$ whose normalizer $N_{G}(A)$ is non-elementary. Then $N_{\Gamma}(A)$ is non-elementary. We have to prove that $E_{\Gamma}(N_{\Gamma}(A))$ and $E_G(N_G(A))$ are equal. We distinguish two cases. If $A$ is not contained in a conjugate of $C$ in $G$, then $N_{\Gamma}(A)=N_{G}(A)$ and $E_{\Gamma}(N_{\Gamma}(A))=E_G(N_G(A))$ (same proof as above). Otherwise, one can suppose without loss of generality that $A$ is contained in $C$. Then one can prove that $E_{\Gamma}(N_G(A))$ is contained in $C$ and deduce that $E_{\Gamma}(N_{\Gamma}(A))=E_G(N_G(A))$ (same proof as above).

\smallskip

\emph{Condition 4.} Let $A$ be a finite subgroup of $G$ whose normalizer is virtually $\mathbb{Z}$ maximal. Note that $A$ is not contained in a conjugate of $C$ in $G$, otherwise $N_G(A)$ would be non-elementary. Thus, the vertex $v$ is the unique fixed point of $A$ in $T$. Therefore, $N_{\Gamma}(A)$ fixes $v$, so $N_{\Gamma}(A)$ and $N_G(A)$ are equal. Let $M$ be the maximal virtually cyclic subgroup of $\Gamma$ containing $N_{\Gamma}(A)$. Since $N_{\Gamma}(A)$ has finite index in $M$ and fixes the vertex $v$, the group $M$ fixes the vertex $v$ as well. Thus $M$ is contained in $G$. Since $N_G(A)$ is maximal and contained in $M$, it is equal to $M$. This proves that $N_{\Gamma}(A)$ is virtually cyclic infinite maximal. In addition, the restriction of the inclusion $i$ to $N_{G}(A)$ is $K$-nice. 

\smallskip

We have proved that $i$ is strongly special. It remains to prove that $i\circ r\sim \mathrm{id}_{\Gamma}$. If $A$ is a finite subgroup of $\Gamma$, then there exists an element $\gamma\in \Gamma$ such that $\gamma A\gamma^{-1}$ is contained in $G$. Consequently, $i\circ r$ coincides with $\mathrm{ad}(r(\gamma)^{-1}\gamma)$ on $A$.\end{proof}

We need an analogous result for legal small extensions. First, we prove an easy lemma.

\begin{lemme}\label{c'est la fin}Let $G$ be a virtually free group. Suppose that $\Gamma=G\ast_N N'$ is a legal small extension of $G$. By definition, there exists a nice emebedding $j : N'\hookrightarrow N$. This homomorphism extends to a monomorphism $j : \Gamma \hookrightarrow G$.
\end{lemme}

\begin{proof}By definition, $G$ splits as $A\ast_C B$ or $A\ast_C$ over a finite subgroup $C$ whose normalizer is $N$. Moreover, $N$ is assumed to be non-elliptic in the splitting. The corresponding tree of cylinders gives a splitting $\Lambda$ of $G$ as a graph of groups whose vertices are $A$, $N$ and $B$ (only in the first case), and whose edge groups are equal to $C$ or contain $C$ as a finite subgroup of index 2. Since $j$ coincides on each finite subgroup of $N\subset N'$ with an inner automorphism, it extends to a homomorphism $j : \Gamma \rightarrow G$ that coincides with inner automorphisms on $A$ and $B$ (see Lemma \ref{s'étend}). Let $\Lambda'$ be the splitting of $\Gamma$ obtained from $\Lambda$ be replacing $N$ by $N'$. One can see that $j$ maps a non-trivial reduced word in the splitting $\Lambda'$ of $\Gamma$ to a non-trivial reduced word in the splitting $\Lambda$ of $G$. This shows that $j$ is injective.\end{proof}

\begin{lemme}\label{lemme2ju2}Let $G$ be a non-elementary virtually free group. Let $\Gamma$ be a legal small extension of $G$. Let $i$ denote the inclusion of $G$ into $\Gamma$ and let $j : \Gamma \hookrightarrow G$ denote a monomorphism as in Lemma \ref{c'est la fin}. Then $i$ and $j$ are strongly special. Moreover, $i\circ j\sim \mathrm{id}_{\Gamma}$.\end{lemme}



\begin{proof}Let $T$ be the Bass-Serre tree of the splitting $\Gamma=G\ast_{\mathrm{id}_C}$, and let $v$ denote a vertex of $T$ fixed by $G$. Note that this vertex is unique, because the vertex group $G$ is not equal to the adjacent edge group $C$. 

By symmetry, it is enough to prove that $j$ is strongly special. There are four conditions that need to be verified.

\smallskip

\emph{Condition 1.} By definition, $j$ is injective on finite subgroups of $\Gamma$. 

\smallskip

\emph{Condition 2.} Since $i\circ j$ maps every finite subgroup of $\Gamma$ to a conjugate of itself, $j$ maps non-conjugate finite subgroups of $\Gamma$ to non-conjugate finite subgroups of $\Gamma$.

\smallskip

\emph{Condition 3.} Let $A$ be a finite subgroup of $\Gamma$. One can suppose without loss of generality that $i\circ j(A)=A$. Let $N:=N_{\Gamma}(A)$. The inclusions $j(N)\subset N_G(j(A))$ and $i(N_G(j(A)))\subset N$ show that $N$ is non-elementary (respectively virtually $\mathbb{Z}$) if and only if $N_{G}(j(A))$ is non-elementary (respectively virtually $\mathbb{Z}$). 

Suppose that $N$ is non-elementary and let $E:=E_{\Gamma}(N)$. Since $E$ is finite, one can suppose without loss of generality that $i\circ j (E)=E$. We claim that $j(E)$ is equal $E_G(N_G(j(A)))$. Recall that a $\Gamma$-chain is a tuple $(\gamma_1,\ldots,\gamma_c)$ of elements of $\Gamma$ of infinite order such that the inclusions 
\begin{equation}\label{inclu}
M(\gamma_1)\supset (M(\gamma_1)\cap M(\gamma_2))\supset \cdots \supset (M(\gamma_1)\cap \cdots \cap M(\gamma_c))\end{equation} 
are all strict (see Definition \ref{chain}), and that the complexity $c:=c(N)$ of $N$ is the maximal size of a $\Gamma$-chain of elements of $N$ (see Definition \ref{complexity}). Let $(\gamma_1,\ldots ,\gamma_c)$ be a $\Gamma$-chain of elements of $N$. By injectivity of $j$, each $j(\gamma_k)$ has infinite order, for $1\leq k\leq c$. Moreover, the sequence of proper inclusions \ref{inclu} is mapped by $j$ to a sequence of proper inclusions, which proves that the tuple $(j(\gamma_1),\ldots ,j(\gamma_c))$ is a $G$-chain of elements of $N_G(j(A))$. In particular, we have $c(N)\leq c(N_G(j(A)))$. Symmetrically, $c(N_G(j(A)))\leq c(N)$. Therefore, the complexities $c(N)$ and $c(N_G(j(A)))$ are the same. This implies that $E_G(N_G(j(A)))$ is equal to the intersection $\cap_{1\leq k\leq c}M(j(\gamma_k))$ (see Lemma \ref{immlemm0307}). Hence, the following holds: \[j(E)=j(\bigcap_{1\leq k\leq c}M(\gamma_k))\subset \bigcap_{1\leq k\leq c}j(M(\gamma_k))\subset \bigcap_{1\leq k\leq c}M(j(\gamma_k))=E_G(N_G(j(A))).\]Symmetrically, $i(E_G(N_G(j(A)))$ is contained in $E$. Hence, $j(E)=E_G(N_G(j(A)))$.

\smallskip

\emph{Condition 4.} Last, if $N=N_{\Gamma}(A)$ is virtually $\mathbb{Z}$ maximal, then $N_G(j(A))$ is virtually $\mathbb{Z}$ maximal. In addition, $j_{\vert N}$ is $K$-nice.\end{proof}

The following lemma shows that the relation $\sim$ is preserved by left of right composition with any homomorphism.

\begin{lemme}\label{base}Let $G,G'$ and $G''$ be virtually free groups. Let $\varphi : G\rightarrow G'$ and $\psi : G \rightarrow G'$ be two homomorphisms. Suppose that $\varphi\sim\psi$. Then, the following hold:
\begin{itemize}
\item[$\bullet$]if $\rho : G'\rightarrow G''$ is a homomorphism, then $\rho\circ \varphi \sim \rho\circ \psi$;
\item[$\bullet$]if $\rho : G''\rightarrow G$ is a homomorphism, then $\varphi \circ \rho\sim \psi\circ \rho$.
\end{itemize}
\end{lemme}


\begin{proof}Let $\rho : G'\rightarrow G''$ be a homomorphism, and let $C$ be a finite subgroup of $G$. By hypothesis, there exists an element $g'\in G'$ such that $\varphi_{\vert C}=\mathrm{ad}(g')\circ \psi_{\vert C}$. By left composing this equality by $\rho$, one gets $\rho\circ \varphi_{\vert C}=\mathrm{ad}(\rho(g'))\circ \rho\circ \psi_{\vert C}$. 

Let $\rho : G''\rightarrow G$ be a homomorphism, and let $C$ be a finite subgroup of $G''$. By hypothesis, since $\rho(C)$ is finite, there exists an element $g'\in G'$ such that $\varphi_{\vert \rho(C)}=\mathrm{ad}(g')\circ \psi_{\vert \rho(C)}$, that is $\varphi\circ \rho_{\vert C}=\mathrm{ad}(g')\circ \psi\circ \rho_{\vert C}$.\end{proof}

We now prove a lemma that will be crucial in the proof of Proposition \ref{2implique3}.

\begin{lemme}\label{lemmetroiscorrection}Let $G$ and $G'$ be two non-elementary virtually free groups, and let $\Gamma$ be a legal extension of $G$. Let $m\geq 1$ be an integer, let $\varphi : G \rightarrow G'$ be a homomorphism satisfying conditions 1 to 3 of Definition \ref{besoinreffff}, and let $\psi : \Gamma\rightarrow G'$ be a homomorphism. If $\psi_{\vert G}\sim \varphi$, then $\psi$ satisfies conditions 1 to 3 of Definition \ref{besoinreffff}.
\end{lemme}


\begin{proof}First, let us suppose that $\Gamma$ is a legal large extension of $G$. Let $i$ denote the inclusion of $G$ into $\Gamma$, and let $r : \Gamma \twoheadrightarrow G$ be the retraction as in Lemma \ref{lemme2ju}. First, note that $\psi_{\vert G}=\psi \circ i$. Therefore $\psi \circ i\sim \varphi$. By the second point of Lemma \ref{base}, we have $\psi \circ i\circ r\sim \varphi\circ r$. Moreover, $i\circ r$ is equivalent (in the sense of $\sim$) to the identity of $\Gamma$ according to Lemma \ref{lemme2ju}, so $\psi \circ i\circ r$ is equivalent to $\psi$, by the first point of Lemma \ref{base}. As a consequence, $\psi$ is equivalent to $\varphi\circ r$. Recall that $r$ is strongly special thanks to Lemma \ref{lemme2ju}. In particular, $r$ is strongly $(\geq m)$-special, so $r$ satisfies conditions 1 to 3 of Definition \ref{besoinreffff}. In addition, $\varphi$ satisfies conditions 1 to 3 of Definition \ref{besoinreffff} by assumption. It follows from Remark \ref{9juillet} below Lemma \ref{onpeutcomposer} that $\varphi\circ r$ satisfies conditions 1 to 3 of Definition \ref{besoinreffff}. Hence, by Lemma \ref{onpeutapprox}, the morphism $\psi$ satisfies conditions 1 to 3 of Definition \ref{besoinreffff}.

Now, suppose that $\Gamma$ is a legal small extension of $G$. Let $i$ denote the inclusion of $G$ into $\Gamma$, and let  $j : \Gamma \hookrightarrow G$ denote a monomorphism as in Lemma \ref{c'est la fin}. First, note that $\psi_{\vert G}=\psi \circ i$. Therefore $\psi \circ i\sim \varphi$. By the second point of Lemma \ref{base}, we have $\psi \circ i\circ j\sim \varphi\circ j$. Moreover, $i\circ j\sim\mathrm{id}_{\Gamma}$, so $\psi \circ i\circ j\sim \psi$. As a consequence, $\psi\sim \varphi\circ j$. Since $\varphi$ is satisfies conditions 1 to 3 of Definition \ref{besoinreffff}, and since $j$ is strongly special (in particular strongly $(\geq m)$-special) thanks to Lemma \ref{lemme2ju2}, it follows from Remark \ref{9juillet} below Lemma \ref{onpeutcomposer} that $\varphi\circ j$ satisfies conditions 1 to 3 of Definition \ref{besoinreffff}. Hence, by Lemma \ref{onpeutapprox}, the homomorphism $\psi$ satisfies the first three conditions of Definition \ref{besoinreffff}.\end{proof}


The following lemma is an immediate consequence of Lemma \ref{lemmetroiscorrection} above.

\begin{lemme}[Extension of a strongly $(\geq m)$-special homomorphism]\label{lemmetrois}Let $G$ and $G'$ be two non-elementary virtually free groups, and let $\Gamma$ be a legal extension of $G$. Let $m\geq 1$ be an integer, let $\varphi : G \rightarrow G'$ be a strongly $(\geq m)$-special homomorphism and $\psi : \Gamma\rightarrow G'$ a homomorphism. If $\psi_{\vert G}\sim \varphi$ and if $\psi$ satisfies the fourth condition of Definition \ref{besoinreffff}, then $\psi$ is strongly $(\geq m)$-special.\end{lemme}

The following lemma shows that strongly $(\geq m)$-special homomorphisms behave nicely with respect to legal extensions of the target group. 

\begin{lemme}\label{lemmequatre}Let $G$ and $G'$ be two non-elementary virtually free groups, and let $\Gamma'$ be a legal (large or small) extension of $G'$. Let $i$ denote the inclusion of $G$ into $\Gamma$. Let $m\geq 1$ be an integer and let $\varphi : G \rightarrow G'$ be a strongly $(\geq m)$-special homomorphism. Then $i\circ \varphi : G \rightarrow \Gamma'$ is strongly $(\geq m)$-special.\end{lemme}

\begin{proof}
By Lemmas \ref{lemme2ju} and \ref{lemme2ju2}, the inclusion $i$ is strongly special, in particular strongly $(\geq m)$-special. Then, by Lemma \ref{onpeutcomposer}, $i\circ \varphi$ is strongly $(\geq m)$-special.
\end{proof}

\begin{lemme}[Restriction of a strongly $(\geq m)$-special pair]\label{rest2906}Let $G$ and $G'$ be two virtually free groups. Suppose that they are not finite-by-free. Let $m$ denote the integer $m_{G,G'}$. Let $\varphi : G\rightarrow G'$ and $\varphi' : G'\rightarrow G$ be two homomorphisms such that $(\varphi,\varphi')$ is a strongly $(\geq m)$-special pair. Let $H$ be a $m$-factor of $G$. Suppose that $\varphi(H)$ is contained in a $m$-factor $H'$ of $G'$, and that $\varphi'(H')\subset H$. Then the pair $(\varphi_{\vert H},\varphi'_{\vert H'})$ is strongly $(>m)$-special.
\end{lemme}

\begin{proof}Let $T$ be a reduced $m$-JSJ tree of $G$. First, we will prove the following preliminary observations: 
\begin{enumerate}
\item if $C$ is a finite subgroup of $H$ of order $>m$, then $N_G(C)=N_H(C)$ and \[E_G(N_G(C))=E_H(N_H(C));\]
\item if $C_1$ and $C_2=gC_1g^{-1}$ are two subgroups of $H$ of order $>m$ with $g\in G$, then $g$ belongs to $H$.
\end{enumerate}
By definition of a $m$-factor, there exists a vertex $v\in T$ such that $H=G_{v}$. The vertex $v$ is the unique vertex of $T$ fixed by $C$, because $\vert C\vert > m$ and edge groups of $T$ have order $m$. This implies that $N_G(C)$ fixes $v$, i.e.\ that $N_G(C)$ is contained in $H$. Hence, $N_G(C)=N_H(C)$. For the same reason, $E_G(N_G(C))\supset C$ is contained in $H$, which proves that $E_G(N_G(C))=E_H(N_H(C))$. The proof of the second point is similar.


Now, let us prove that $\varphi_{\vert H} : H\rightarrow H'$ is strongly $(>m)$-special. There are four points that need to be satisfied. 

\smallskip

\emph{Condition 1.} The restriction $\varphi_{\vert H}$ is injective on finite subgroups of $H$, because $\varphi$ is injective on finite subgroups of $G$. 

\smallskip

\emph{Condition 2.} Let $C_1$ and $C_2$ be two non-conjugate finite subgroups of $H$ of order $>m$. By the second preliminary observation, these groups are non-conjugate in $G$. Since $\varphi$ is $(\geq m)$-special, $\varphi(C_1)$ and $\varphi(C_2)$ are non-conjugate in $G'$. Thus, $\varphi(C_1)$ and $\varphi(C_2)$ are non-conjugate in $H'$.

\smallskip

\emph{Condition 3.} If $C$ is a finite subgroup of $H$ of order $>m$ whose normalizer $N_{H}(C)$ is non-elementary, then $N_G(C)$ is non-elementary. This implies that $N_{G'}(\varphi(C))$ is non-elementary, because $\varphi$ is strongly $(\geq m)$-special. Moreover, we have $N_{G'}(\varphi(C))=N_{H'}(\varphi(C))$ according to the first preliminary observation. Hence, $N_{H'}(\varphi(C))$ is non-elementary. In addition, we have \[\varphi(E_{G}(N_{G}(C)))=E_{G'}(N_{G'}(\varphi(C))),\] since $\varphi$ is strongly $(\geq m)$-special. It follows that \[\varphi(E_{H}(N_{H}(C)))=E_{H'}(N_{H'}(\varphi(C))),\] because, by the first preliminary observation, \[E_{G}(N_{G}(C))=E_{H}(N_{H}(C)) \ \ \ \text{ and } \ \ \ E_{G'}(N_{G'}(\varphi(C)))=E_{H'}(N_{H'}(\varphi(C))).\] 

\smallskip

\emph{Condition 4.} Let $C$ be a finite subgroup of $H$ of order $>m$ such that $N_{H}(C)$ is virtually cyclic infinite maximal. The morphism $\varphi$ being strongly $(\geq m)$-special, the normalizer of $\varphi(C)$ in $G'$ is virtually cyclic infinite maximal and the restriction of $\varphi$ to $N_G(C)$ is $K$-nice in the sense of Definition \ref{special0}. Since $N_H(C)=N_G(C)$ and $N_{H'}(\varphi(C))=N_{G'}(\varphi(C))$, the restriction of $\varphi$ to $N_{H}(C)$ is $K$-nice.

We have proved that $\varphi_{\vert H}$ is strongly $(>m)$-special. Since the same arguments remain valid with $\varphi'$ instead of $\varphi$, the restriction $\varphi'_{\vert H'}$ is strongly $(>m)$-special as well.

It remains to prove that the pair $(\varphi_{\vert H},\varphi'_{\vert H'})$ is strongly $(>m)$-special. To that end, let us consider a finite subgroup $C$ of $H$ of order $>m$. Since $\varphi'\circ \varphi\sim \mathrm{id}_G$, there exists an element $g\in G$ such that $\varphi'\circ \varphi(C)=gCg^{-1}$. Since $\varphi'\circ \varphi(H)$ is contained in $H$ by assumption, the groups $C$ and $gCg^{-1}$ belong to $H$. By the preliminary observation, $g$ belongs to $H$. Hence, $\varphi'\circ \varphi$ maps every finite subgroup of $H$ of order $>m$ to a conjugate of itself in $H$. Symmetrically, $\varphi\circ \varphi'$ maps every finite subgroup of $H'$ of order $>m$ to a conjugate of itself in $H'$.\end{proof}

\subsubsection{Legal $(\geq m)$-extensions}

\begin{de}Let $m\geq 1$ be an integer. Let $\Gamma$ be a virtually free group, and let $G$ be a subgroup of $\Gamma$. We say that $\Gamma$ is a \emph{multiple legal $(\geq m)$-extension} of $G$ if there exist nested subgroups $G=G_1\subset G_2\subset \cdots \subset G_n=\Gamma$ and integers $(k_i)_{1\leq i\leq n-1}$ such that $k_i\geq m$ and $G_{i+1}$ is a legal large or small $k_i$-extension of $G_i$ (see Definitions \ref{legal} and \ref{legal2}) for every $1\leq i\leq n-1$. If $n=2$, we simply say that $\Gamma$ is a \emph{legal $(\geq m)$-extension} of $G$. In the same way, we define \emph{multiple legal $(> m)$-extensions} and \emph{multiple legal $m$-extensions} of $G$ if $k_i>m$ or $k_i=m$ respectively.\end{de}

\begin{lemme}\label{légalité}Let $m\geq 1$ be an integer. Let $G$ be a virtually free group, and let $\Delta$ be a $m$-splitting of $G$ as a graph of groups. Let $k\geq m $ be an integer, let $v$ be a vertex of $\Delta$ and let $\widehat{G}_v$ be a legal large $k$-extension of $G_v$. Then the group $\widehat{G}$ obtained from $G$ by replacing $G_{v}$ by $\widehat{G}_{v}$ in $\Delta$ is a legal large $k$-extension of $G$.\end{lemme}

\begin{proof}By definition of a legal large $k$-extension, there is a subgroup $C$ of $G_v$ of order $k$ such that $\widehat{G}_v=\langle G_v, t \ \vert \ [t,c]=1, \forall c\in C\rangle$, with $N_{G_v}(C)$ non-elementary and $E_{G_v}(N_{G_v}(C))=C$. Thus, the normalizer of $C$ in $G$ is non-elementary, and we only have to prove that $E:=E_G(N_G(C))=C$. Note that the inclusion $C\subset E$ always holds, because $E$ is the unique maximal finite subgroup of $G$ normalized by $N_G(C)$. It remains to prove that $E\subset C$.

Let $T$ be the Bass-Serre tree of $\Delta$, endowed with the action of $G$. We shall still denote by $v$ a lift in $T$ of the vertex $v$ of $\Delta$. As a first step, we shall prove that $E_{G}(N_{G_v}(C))$ fixes the vertex $v\in T$. Assume towards a contradiction that $E_G(N_{G_v}(C))$ does not fix $v$. Then the inclusion $C\subset E_G(N_{G_v}(C))$ is strict. This shows in particular that $E_G(N_{G_v}(C))$ has order $>\vert C\vert \geq m$. Since $E_G(N_{G_v}(C))$ is finite, it fixes a vertex $w\neq v$ of $T$, and this vertex is unique because $E_G(N_{G_v}(C))$ has order $>m$. It follows that $N_{G_v}(C)$ fixes $w$ as well. Hence, $N_{G_v}(C)$ is contained in the finite group $G_w\cap G_v$, contradicting the fact that $N_{G_v}$ is non-elementary. We have proved that $E_G(N_{G_v}(C))$ fixes $v$. As a consequence, the following equality holds: \begin{equation}\label{eq1}
E_{G_v}(N_{G_v}(C))=E_G(N_{G_v}(C)).
\end{equation}

Now, let us assume towards a contradiction that $C$ is strictly contained in $E$. Since $E$ is finite, it fixes a vertex $w$ of $T$. Moreover, this vertex is unique since $\vert E\vert >\vert C\vert=k \geq m$. It follows that $N_G(E)$ is contained in $G_w$. But $N_G(C)$ is contained in $N_G(E)$ by definition of $E$, hence $N_G(C)$ fixes $w$. Since $N_{G_v}(C)$ is infinite, it fixes only $v$ in $T$, which proves that $w=v$. As a consequence $N_G(C)=N_{G_v}(C)$, and therefore \begin{equation}\label{eq2}
E_G(N_G(C))=E_G(N_{G_v}(C)).
\end{equation}

By combining equations (\ref{eq1}) and (\ref{eq2}), we get $E_G(N_G(C))=E_{G_v}(N_{G_v}(C))$, i.e.\ $E=C$, contradicting the assumption that $C$ is strictly contained in $E$. As a conclusion, we have proved that $E=C$. This proves that $\widehat{G}$ is a legal large $k$-extension of $G$.\end{proof}

We need an analogous result for small extensions.

\begin{lemme}\label{légalité2}Let $m\geq 1$ be an integer. Let $G$ be a virtually free group, and let $\Delta$ be a $m$-splitting of $G$ as a graph of groups. Let $v$ be a vertex of $\Delta$. Let $\Delta_v$ be a one edge splitting of $G_v$ over a finite group $C$ of order $k\geq m$ whose normalizer $N_{G_v}(C)$ is virtually cyclic and non-elliptic in $\Delta_v$. Let $\widehat{G}_v$ be a legal small $k$-extension of $G_v$. If $N_{G_v}(C)=N_{G}(C)$, then the group $\widehat{G}$ obtained from $G$ by replacing $G_{v}$ by $\widehat{G}_{v}$ in $\Delta$ is a legal small $k$-extension of $G$.\end{lemme}

\begin{proof}We only need to verify that $C$ is an edge group in a splitting of $G$ in which $N_G(C)$ is non-elliptic. First, note that for any edge $e$ of $\Delta$ incident to $v$, the edge group $G_e$ is contained in a conjugate of $A$ or $B$, as a finite group. Let $\Delta'$ be the splitting of $G$ obtained from $\Delta$ by replacing the vertex $v$ by the splitting $\Delta_v$ of $G_v$, and let $\varepsilon$ be the new edge of $\Delta'$ coming from $\Delta_v$. By collapsing all edges of $\Delta'$ different from $\varepsilon$, we get a one edge splitting of $G$ over $C$ in which $N_G(C)$ is non-elliptic.\end{proof}

\begin{rque}\label{rem29juin}Let $T$ be the Bass-Serre tree of $\Delta$. We still denote by $v$ a lift in $T$ of the vertex $v$ of $\Delta$. Let us observe that the equality $N_{G_v}(C)=N_{G}(C)$ holds if $C$ is not an edge group of $\Delta$, i.e.\ if $v$ is the unique vertex of $T$ fixed by $C$, because in this case $N_G(C)$ fixes $v$ as well, which implies that $N_{G_v}(C)=N_G(C)\cap G_v=N_G(C)$. For instance, if $k=\vert C\vert > m$, then $v$ is the unique vertex of $T$ fixed by $C$.\end{rque}

The following lemma allows us to iterate Lemmas \ref{légalité} and \ref{légalité2} above.

\begin{lemme}\label{légalité3}Let $m\geq 1$ be an integer. Let $G$ be a virtually free group, and let $\Delta$ be a $m$-splitting of $G$ as a graph of groups. Let $v$ and $w$ be two distinct vertices of $\Delta$. Let $\widehat{G}_v$ and $\widehat{G}_w$ be two legal $(\geq m)$-extensions of $G_v$ and $G_w$. If the extension $\widehat{G}_v$ (respectively $G_w$) is small, suppose that $N_{G_v}(C)=N_{G}(C)$ (respectively $N_{G_w}(C)=N_{G}(C)$), where $C$ is the edge group of the one edge splitting of $G_v$ (respectively $G_w$) associated with the small extension. Then the group $\widehat{G}$ obtained from $G$ by replacing $G_{v}$ by $\widehat{G}_{v}$ and $G_{w}$ by $\widehat{G}_{w}$ in $\Delta$ is a multiple legal $(\geq m)$-extension of $G$.\end{lemme}

\begin{proof}Let $\Gamma$ be the group obtained from $G$ by replacing $G_{v}$ by $\widehat{G}_{v}$. By Lemmas \ref{légalité} and \ref{légalité2}, $\Gamma$ is a legal $(\geq m)$-extension of $G$. We claim that $\widehat{G}$ is a legal $(\geq m)$-extension of $\Gamma$. Note that $\Gamma_w=G_w$. We distinguish two cases. If $\widehat{G}_w$ is a large extension of $G_w$, there is no condition to be checked and Lemma \ref{légalité} claims that $\widehat{G}$ is a legal $(\geq m)$-extension of $\Gamma$. If $\widehat{G}_w$ is a small extension of $G_w$, the group $N_{\Gamma}(C)$ is equal to $N_G(C)$, which is equal to $N_{G_w}(C)$ by assumption. Since $\Gamma_w=G_w$, we have $N_{\Gamma}(C)=N_{\Gamma_w}(C)$. Hence Lemma \ref{légalité2} applies and guarantees that $\widehat{G}$ is a legal $(\geq m)$-extension of $\Gamma$.\end{proof}

By iterating Lemma \ref{légalité3}, we get the following result.

\begin{co}\label{légalité4}Let $m\geq 1$ be an integer. Let $G$ be a virtually free group and let $\Delta$ be a $m$-splitting of $G$ as a graph of groups. For every vertex $v$ of $\Delta$, let $\widehat{G}_v$ be a multiple legal $(\geq m)$-extension of $G_v$. If the extension $\widehat{G}_v$ is small, suppose that the edge group of the one edge splitting of $G_v$ associated with the small extension has order $> m$. Then the group $\widehat{G}$ obtained from $G$ by replacing every $G_{v}$ by $\widehat{G}_{v}$ in $\Delta$ is a multiple legal $(\geq m)$-extension of $G$.\end{co}

\subsubsection{Property $\mathcal{P}_m$}


\begin{de}\label{Pm}Let $G$ and $G'$ be two virtually free groups. Let $m$ denote the integer $m_{G,G'}$. Let $\varphi : G \rightarrow G'$ and $\varphi' : G'\rightarrow G$ be two homomorphisms. We say that the tuple $(G,G',\varphi,\varphi')$ has property $\mathcal{P}_m$ if the following two conditions hold:
\begin{enumerate}
\item the pair $(\varphi,\varphi')$ is strongly $(\geq m)$-special;
\item $\varphi'\circ \varphi$ maps each $m$-factor of $G$ isomorphically to a conjugate of itself, and $\varphi\circ \varphi'$ maps each $m$-factor of $G'$ isomorphically to a conjugate of itself.
\end{enumerate}
\end{de}

\begin{rque}\label{capermute}Let $G_1,\ldots,G_p$ and $G'_1,\ldots,G'_{p'}$ be two sets of representatives of the $m$-factors of $G$ and $G'$ respectively. If $(G,G',\varphi,\varphi')$ has $\mathcal{P}_m$, then $p=p'$ and, up to renumbering $G_1,\ldots,G_p$, the homomorphism $\varphi$ maps each $G_i$ isomorphically to a conjugate of $G'_i$, and $\varphi$ maps each $G'_i$ isomorphically to a conjugate of $G_i$. Indeed, let $T$ be a reduced $m$-JSJ splitting of $G$, and let $H$ be a vertex group of $T$, that is a $m$-factor of $G$. Since $\varphi'\circ \varphi$ is injective on $H$, the morphism $\varphi$ is injective on $H$ as well. Hence $\varphi(H)\simeq H$ is $(\leq m)$-rigid. As a consequence, $\varphi(H)$ is contained in a vertex group $H'$ of $T'$. We claim that $\varphi$ induces an isomorphism from $H$ to $H'$. First, note that $\varphi'(H')$ is contained in a vertex group $K$ of $T$, for the same reason as above. Therefore, $\varphi'\circ \varphi(H)$ is contained in $K$. Moreover, we know that $\varphi'\circ \varphi(H)=H^g$ for some $g\in G$. It follows that $H^g$ is contained in $K$. Since $T$ is reduced and since $H^g$ and $K$ are two vertex groups of $T$, we have $H^g=K$. Likewise, $\varphi'(H^g)=H'^{g'}$ for some $g'\in G'$. Hence, the following series of inclusions holds:\[H\overset{\varphi}{\hookrightarrow} H'\overset{\varphi'}{\hookrightarrow} H^g \overset{\varphi}{\hookrightarrow} H'^{g'}.\]Recall that the homomorphism $\varphi\circ\varphi'$ is surjective from $H'$ onto $H'^{g'}$. Thus, $\varphi$ induces an isomorphism from $H^g$ to $H'^{g'}$, and it follows that $\varphi$ induces an isomorphism from $H$ to $H'$.\end{rque}


\subsubsection{Expansions}

\begin{de}\label{expan}Let $G$ and $G'$ be two virtually free groups. Let $m$ denote the integer $m_{G,G'}$. Let $\varphi : G \rightarrow G'$ and $\varphi' : G'\rightarrow G$ be two homomorphisms. Let $\Gamma$ and $\Gamma$ be two virtually free groups containing $G$ and $G'$ respectively, and let $\psi : \Gamma\rightarrow\Gamma'$ and $\psi' : \Gamma'\rightarrow\Gamma$ be two homomorphisms. We say that $(\Gamma,\Gamma',\psi,\psi')$ is a \emph{$(\geq m)$-expansion} of $(G,G',\varphi,\varphi')$ if the following conditions are satisfied:
\begin{enumerate}
\item $\Gamma$ and $\Gamma'$ are two multiple legal $(\geq m)$-extensions of $G$ and $G'$;
\item $\psi_{\vert G}\sim \varphi$ and $\psi'_{\vert G'}\sim \varphi'$;
\end{enumerate}
In the same way, we define \emph{$(>m)$-expansions}.
\end{de}


\begin{rque}Let $\mathcal{U}=(G,G',\varphi,\varphi')$ be a tuple as in Definition \ref{expan} above. If $\mathcal{U}_1$ is a $(\geq m)$-expansion of $\mathcal{U}$, and if $\mathcal{U}_2$ is a $(\geq m)$-expansion of $\mathcal{U}_1$, then $\mathcal{U}_2$ is a $(\geq m)$-expansion of $\mathcal{U}$.\end{rque}

\subsection{Finite extensions of free products}In this section, we are concerned with the case where there is only one cylinder in the trees we consider, which means that all edge groups are equal. This particular case will play a crucial role in the proof of Proposition \ref{2implique3}, which uses extensively the trees of cylinders.

Given a group $G$ and a subgroup $C\subset G$, we denote by $\mathrm{Aut}_G(C)$ the following subgroup of $\mathrm{Aut}(C)$: \[\mathrm{Aut}_G(C)=\lbrace \sigma\in\mathrm{Aut}(C) \ \vert \ \exists g\in N_G(C), \ \mathrm{ad}(g)_{\vert C}=\sigma\rbrace.\]

\begin{lemme}\label{normalisateur}Let $G$ and $G'$ be two non-elementary hyperbolic groups. Let $T$ and $T'$ be two simplicial trees endowed with actions of $G$ and $G'$ respectively. Suppose that all edge groups of $T$ and $T'$ are equal, and let $C$ and $C'$ denote these edge groups. Suppose in addition that $T$ and $T'$ have the same number of orbits of vertex groups, say $p$. Let $G_1,\ldots ,G_p$ and $G'_1,\ldots,G'_p$ be some representatives of the vertex groups of $T$ and $T'$. Suppose that the following two conditions hold:
\begin{itemize}
\item[$\bullet$]there exists a strongly $(\geq \vert C\vert)$-special homomorphism $\varphi : G\rightarrow G'$ that maps each subgroup $G_i$ isomorphically to a conjugate of $G'_i$,
\item[$\bullet$]and there exists a strongly $(\geq \vert C'\vert)$-special homomorphism $\varphi' : G'\rightarrow G$ that maps each subgroup $G'_i$ isomorphically to a conjugate of $G_i$.
\end{itemize}
Then $C$ and $C'$ have the same order $m$, and there exists a $(\geq m)$-expansion $(\widehat{G}, \widehat{G}',\widehat{\varphi},\widehat{\varphi}')$ of $(G,G',\varphi,\varphi')$ such that $\widehat{\varphi}$ and $\widehat{\varphi}'$ are bijective and satisfy the following two conditions:
\begin{itemize}
\item[$\bullet$]for every $1\leq i\leq p$, there exists an element $g'_i\in \widehat{G}'$ such that \[\varphi_{\vert G_i}=\mathrm{ad}(g'_i)\circ \widehat{\varphi}_{\vert G_i},\] where $\varphi$ is viewed as a homomorphism from $G$ to $\widehat{G}'$;
\item[$\bullet$]for every $1\leq i\leq p$, there exists an element $g_i\in \widehat{G}$ such that \[\varphi'_{\vert G'_i}=\mathrm{ad}(g_i)\circ \widehat{\varphi}'_{\vert G'_i},\] where $\varphi'$ is viewed as a homomorphism from $G'$ to $\widehat{G}$.
\end{itemize}
\end{lemme}

\begin{rque}It is worth noticing that in the particular case where $T$ and $T'$ are $m$-JSJ splittings of $G$ and $G'$ respectively (which is equivalent to say that the vertex groups of $T$ and $T'$ are $m$-rigid), the conclusion of the lemma can be reformulated as follows: there exists a $(\geq m)$-expansion $(\widehat{G}, \widehat{G}',\widehat{\varphi},\widehat{\varphi}')$ of $(G,G',\varphi,\varphi')$ with property $\mathcal{P}_m$. However, the lemma is stated in a more general context because, at some point in the proof (more precisely in the third case of the second step of the proof of Proposition \ref{ledernierlemme}), we will be considering some splittings that are not $m$-JSJ splittings.\end{rque}

\begin{rque}The groups $\widehat{G}$ and $\widehat{G}'$ constructed in the proof below are obtained by performing legal large extensions over $C$ and $C'$ only.\end{rque}

\begin{proof}Note that $C$ and $C'$ are normal in $G$ and $G'$. As a first step, we shall prove that $\varphi(C)= C'$ and $\varphi'(C')= C$. Since $\varphi$ is strongly $(\geq \vert C\vert)$-special, and since $N_G(C)=G$ is non-elementary, we have \[\varphi(E_G(N_G(C)))=E_{G'}(N_{G'}(\varphi(C))).\]The left-hand side of this equation is equal to $\varphi(C)$. Indeed, $N_G(C)=G$ and $E_G(G)=C$ since $C$ is the unique maximal finite normal subgroup of $G$. The right-hand side of the equation contains $C'$. Indeed, since $C'$ is a normal subgroup of $G'$, it is in particular normalized by $N_{G'}(\varphi(C))\subset G'$; hence, $C'$ is contained in the unique maximal finite subgroup of $G'$ normalized by $N_{G'}(\varphi(C))$, namely $E_{G'}(N_{G'}(\varphi(C)))$. We have proved that $C'$ is contained in $\varphi(C)$. Likewise, we have $C\subset \varphi'(C')$.

Since the homomorphisms $\varphi$ and $\varphi'$ are injective on finite subgroups, the inclusions $C'\subset\varphi(C)$ and $C\subset \varphi'(C')$ show that $\varphi(C)= C'$ and $\varphi'(C')= C$. More precisely, $\varphi$ induces an isomorphism from $C$ to $C'$, and $\varphi'$ induces an isomorphism from $C'$ to $C$. In particular, the finite groups $C$ and $C'$ have the same order denoted by $m$.

We will now define the groups $\widehat{G}$ and $\widehat{G}'$. First, let us define a homomorphism $\overline{\varphi}:\mathrm{Aut}_G(C) \rightarrow \mathrm{Aut}_{G'}(C')$ as follows: for every $\theta=\mathrm{ad}(g)_{\vert C}\in\mathrm{Aut}_G(C)$, set \[\overline{\varphi}(\theta)=\varphi_{\vert C}\circ \theta\circ ({\varphi}_{\vert C})^{-1}=\mathrm{ad}(\varphi(g))_{\vert C'} \ \text{independent from the choice of $g$}.\]Note that this homomorphism is injective. Indeed, if there exists an element $c\in C$ such that $gcg^{-1}\neq c$, then $\varphi(g)\varphi(c)\varphi(g)^{-1}\neq \varphi(c)$, because $gcg^{-1}$ belongs to $C$ and $\varphi$ is injective on $C$. Likewise, $\varphi'$ induces a monomorphism $\mathrm{Aut}_{G'}(C')\hookrightarrow \mathrm{Aut}_{G}(C)$. As a consequence, $\mathrm{Aut}_{G}(C)$ and $\mathrm{Aut}_{G'}(C')$ have the same order, say $\ell$. Set $\mathrm{Aut}_{G}(C)=\lbrace \theta_1,\ldots,\theta_{\ell}\rbrace$ and $\mathrm{Aut}_{G'}(C')=\lbrace \theta'_1,\ldots,\theta'_{\ell}\rbrace$ with $\theta'_i=\overline{\varphi}(\theta_i)$ for every $1\leq i\leq \ell$.

Note that the groups $G$ and $G'$ split respectively as \[1\rightarrow C\rightarrow G\rightarrow Q=Q_1\ast\cdots\ast Q_p\ast F_k\rightarrow 1 \ \ \text{   and   }\ \  1\rightarrow C'\rightarrow G'\rightarrow Q'=Q'_1\ast\cdots\ast Q'_p\ast F_{k'}\rightarrow 1\]where $Q_i$ is the image in $Q$ of a conjugate of $G_i$, and $Q'_i$ is the image in $Q'$ of a conjugate of $G'_i$. Up to replacing $G_i$ by a conjugate of itself, one can suppose that $G_i$ is the preimage of $Q_i$ in $G$, for every $1\leq i\leq p$. Likewise, one can suppose that $G'_i$ is the preimage of $Q'_i$ for every $1\leq i\leq p$. Let $\lbrace x_1,\ldots ,x_k\rbrace$ be a generating set of the free group $F_k$, and let $t_i$ be a preimage of $x_i$ in $G$, for every $1\leq i\leq k$. Each element $t_i\in G$ induces by conjugacy an automorphism $\theta_{s(i)}\in\mathrm{Aut}_G(C)$. Let $H$ denote the subgroup $G_1\ast_{C}\cdots\ast_{C} G_p\subset G$, preimage of $Q_1\ast\cdots\ast Q_p$ in $G$. We define $t'_i$, $\theta'_{s'(i)}$ and $H'$ in the same manner. 

The group $G$ admits the following finite presentation:\[G=\langle H,t_1,\ldots ,t_k \ \vert \ \mathrm{ad}(t_i)_{\vert C}=\theta_{s(i)} \ \forall i\in\llbracket 1,k\rrbracket\rangle.\]
Similarly, the group $G'$ has a finite presentation of the form \[G'=\langle H',t'_1,\ldots ,t'_{k'} \ \vert \ \mathrm{ad}(t'_i)_{\vert C'}=\theta'_{s'(i)} \ \forall i\in\llbracket 1,k'\rrbracket\rangle.\]

Let $n=\max(k,k')-k$ and $n'=\max(k,k')-k'$, so that $n+k=n'+k'$. Let us define the overgroups $\widehat{G}$ and $\widehat{G}$ of $G$ and $G'$ as follows:

\[\widehat{G}=\Biggl\langle 
       \begin{array}{l|cl}
                     G,t_{k+1},\ldots,t_{k+\ell},\ldots ,t_{k+\ell+n}   & \mathrm{ad}(t_i)_{\vert C}=\theta_{i} & \forall i\in\llbracket k+1,k+\ell\rrbracket \\
                       &  \mathrm{ad}(t_i)_{\vert C}=\mathrm{id}_C & \forall i\geq k+\ell+1
                                                                   
        \end{array}
     \Biggr\rangle,\]

\[\widehat{G}'=\Biggl\langle 
       \begin{array}{l|cl}
                     G',t'_{k+1},\ldots,t'_{k'+\ell},\ldots ,t'_{k'+\ell+n'}   & \mathrm{ad}(t'_i)_{\vert C'}=\theta'_{i} & \forall i\in\llbracket k'+1,k'+\ell\rrbracket \\
                       &  \mathrm{ad}(t'_i)_{\vert C'}=\mathrm{id}_{C'} & \forall i\geq k'+\ell+1
                                                                   
        \end{array}
     \Biggr\rangle.\]
     
Note that $\widehat{G}$ is a multiple legal large $m$-extension of $G$. We can see that by defining a finite sequence of groups $(\widehat{G}_q)_{0\leq q\leq \ell+n}$ by $\widehat{G}_0=G$ and $\widehat{G}_{q+1}=\langle \widehat{G}_q,t_{k+q+1}\rangle$ for $0\leq q< \ell+n$, and by observing that $\widehat{G}=\widehat{G}_{\ell+n}$, and that $\widehat{G}_{q+1}$ is a legal $\vert C\vert $-extension of $\widehat{G}_q$ for every $q$, because $\mathrm{ad}(t_{k+q+1})_{\vert C}$ belongs to $\mathrm{Aut}_G(C)$. In the same way, the group $\widehat{G}'$ is a legal large $\vert C'\vert$-extension of $G'$. 

We will now construct the isomorphisms $\widehat{\varphi} : \widehat{G}\rightarrow \widehat{G}'$ and $\widehat{\varphi}' : \widehat{G}'\rightarrow \widehat{G}$ satisfying the expected conditions. Let $N:=\ell+k+n=\ell+k'+n'$. Up to renumbering the elements $t_i$, one can assume that $\mathrm{ad}(t_i)_{\vert C}=\theta_i$ for every $1\leq i\leq \ell$. Then, for every $i\geq \ell+1$, there exists an integer $1\leq j\leq \ell$ such that $\mathrm{ad}(t_i)_{\vert C}=\mathrm{ad}(t_j)_{\vert C}$, because $\mathrm{Aut}_G(C)=\lbrace\theta_1,\ldots,\theta_{\ell}\rbrace$ and $\mathrm{ad}(t_i)_{\vert C}$ belongs to $\mathrm{Aut}_G(C)$. Hence, up to replacing $t_i$ by $t_j^{-1}t_i$, one can assume without loss of generality that $\theta_i=\mathrm{id}_{C}$. Now, $\widehat{G}$ has the following presentation: 

\[\widehat{G}=\Biggl\langle 
       \begin{array}{l|cl}
                     H,t_{1},\ldots,t_{N}  & \mathrm{ad}(t_i)_{\vert C}=\theta_{i} & \forall i\in\llbracket 1,\ell\rrbracket \\
                       &  \mathrm{ad}(t_i)_{\vert C}=\mathrm{id}_C & \forall i\geq \ell+1
                                                                   
        \end{array}
     \Biggr\rangle.\]

Likewise, $\widehat{G}'$ has a presentation of the following form: 

\[\widehat{G}'=\Biggl\langle 
       \begin{array}{l|cl}
                     H',t'_{1},\ldots,t'_{N}  & \mathrm{ad}(t'_i)_{\vert C'}=\theta'_{i} & \forall i\in\llbracket 1,\ell\rrbracket \\
                       &  \mathrm{ad}(t'_i)_{\vert C'}=\mathrm{id}_{C'} & \forall i\geq \ell+1
                                                                   
        \end{array}
     \Biggr\rangle.\]

We are now ready to define $\widehat{\varphi}$ and $\widehat{\varphi}'$. By assumption, $\varphi(G_i)=g'_i{G'_i}{g'_i}^{-1}$ for some $g'_i\in G'$. Since $\mathrm{Aut}_{G'}(C')=\lbrace\mathrm{ad}(t'_j)_{\vert C'}, 1\leq j\leq \ell\rbrace$, there exists an integer $\sigma(i)\in\llbracket 1,\ell\rrbracket$ such that $\mathrm{ad}(g'_i)_{\vert C'}=\mathrm{ad}(t'_{\sigma(i)})_{\vert C'}$. Recall that $H=G_1\ast_{C}\cdots\ast_{C} G_p$. First, let us define a homomorphism $\psi : H \rightarrow \widehat{G}'$ by \[\psi_{\vert G_i}=\mathrm{ad}\left(t'_{\sigma(i)}{g'_i}^{-1}\right)\circ \varphi_{\vert G_i}\] for every $1\leq i\leq p$. This homomorphism is well-defined since the element $t'_{\sigma(i)}{g'_i}^{-1}$ of $G'$ centralizes $C'$.

Then, let us define $\widehat{\varphi}:\widehat{G}\rightarrow\widehat{G}'$ by $\widehat{\varphi}_{\vert H}=\psi$ and $\widehat{\varphi}(t_i)=t'_i$ for every $1\leq i\leq N$. This homomorphism is well-defined because $\mathrm{ad}(t_i)_{\vert C}=\theta_i$ and $\mathrm{ad}(t'_i)_{\vert C'}=\varphi_{C}\circ \theta_i\circ (\varphi_C)^{-1}=\theta'_i$.

Last, note that \[\widehat{\varphi}(G_i)=t'_{\sigma(i)}{G'_i}{t'_{\sigma(i)}}^{-1}={\widehat{\varphi}\left(t_{\sigma(i)}\right)}{G'_i}{\widehat{\varphi}\left(t_{\sigma(i)}\right)}^{-1},\] i.e.\ \[G'_i=\widehat{\varphi}\left(t_{\sigma(i)}^{-1}G_it_{\sigma(i)}\right).\]Hence, the image of $\widehat{\varphi}$ contains $G'_i$, for every $1\leq i\leq p$ and $t'_i=\widehat{\varphi}(t_i)$ for every $1\leq i\leq N$. As a consequence, since $\widehat{G}'$ is generated by $G'_1,\ldots,G'_p, t'_1,\ldots ,t'_N$, the homomorphism $\widehat{\varphi}$ is surjective. Likewise, there exists an epimorphism $\widehat{\varphi}' : \widehat{G}' \twoheadrightarrow \widehat{G}$ that coincides with $\varphi'$ on each $G'_i$ up to conjugacy. Since hyperbolic groups are Hopfian, the epimorphism $\widehat{\varphi}'\circ \widehat{\varphi}: \widehat{G}\twoheadrightarrow\widehat{G}$ is an automorphism. Hence, $\widehat{\varphi}$ and $\widehat{\varphi}'$ are two isomorphisms.\end{proof}

\begin{rque}\label{rkalgo}This remark will be useful for proving that there exists an algorithm that takes as input two finite presentations of virtually free groups, and decides whether these groups have the same $\forall\exists$-theory or not. We keep the same notations as in the proof above. Let $r$ be the rank of $G$ (that is the smallest cardinality of a generating set for $G$), and let $r'$ be the rank of $G'$. Note that $k\leq r$ and $k'\leq r'$. We constructed the groups $\widehat{G}$ and $\widehat{G}'$ from $G$ and $G'$ by performing less than $\mathrm{max}(k,k')+\ell\leq \mathrm{max}(r,r')+\vert C\vert !$ legal large extensions.
\end{rque}


We now consider reduced Stallings splittings of virtually free groups $G$ and $G'$. If all edge groups are equal, then Lemma \ref{normalisateur} applies. Here below are two consequences of Lemma \ref{normalisateur} in this context. These results will be useful in the proof of the general case of the implication $(4)\Rightarrow (5)$ of Theorem \ref{principal} (see Proposition \ref{2implique3}). 

\begin{co}\label{f-by-f0}Let $Q=Q_1\ast\cdots\ast Q_p\ast F_k$ and $Q'=Q'_1\ast\cdots\ast Q'_{p'}\ast F_{k'}$ be two free products of finite groups with a free group. Suppose that $Q$ and $Q'$ are non-elementary. Let $G$ and $G'$ be two finite extensions \[1\rightarrow C\rightarrow G\rightarrow Q\rightarrow 1 \ \ \ \text{and} \ \ \ 1\rightarrow C'\rightarrow G'\rightarrow Q'\rightarrow 1.\]Suppose that there exist two homomorphisms $\varphi : G \rightarrow G'$ and $\varphi' : G' \rightarrow G$ such that the pair $(\varphi, \varphi')$ is strongly $(\geq m_{G,G'})$-special. Then $C$ and $C'$ have the same order $m_{G,G'}$, and there exists a $(\geq m_{G,G'})$-expansion $(\Gamma,\Gamma',\psi,\psi')$ of $(G,G',\varphi,\varphi')$ such that $\psi: \Gamma\rightarrow\Gamma'$ and $\psi': \Gamma'\rightarrow\Gamma$ are bijective.\end{co}

\begin{proof}For every $1\leq i\leq p$, let $G_i$ be a preimage of $Q_i$ in $G$ and for every $1\leq i\leq p'$, let $G'_i$ be a preimage of $Q'_i$ in $G'$. In order to establish the existence of $\Gamma$ and $\Gamma'$, we shall use Lemma \ref{normalisateur}. To that end, it is enough to verify that the following three conditions are satisfied (up to renumbering the $G_i$):
\begin{itemize}
\item[$\bullet$]$p=p'$,
\item[$\bullet$]$\varphi$ maps each $G_i$ isomorphically to a conjugate of $G'_{i}$,
\item[$\bullet$]and $\varphi'$ maps each $G'_i$ isomorphically to a conjugate of $G_{i}$.
\end{itemize}

Since every finite subgroup of $G'$ is contained in a conjugate of some $G'_i$, there exists a map $\sigma : \llbracket 1,p\rrbracket \rightarrow \llbracket 1,p'\rrbracket$ such that $\varphi(G_i)$ is contained in a conjugate of $G'_{\sigma(i)}$, for every $1\leq i\leq p$. Likewise, there exists a map $\sigma' : \llbracket 1,p'\rrbracket \rightarrow \llbracket 1,p\rrbracket$ such that $\varphi'(G'_i)$ is contained in a conjugate of $G_{\sigma'(i)}$, for every $1\leq i\leq p'$. 

Since $(\varphi,\varphi')$ is a strongly $(\geq m_{G,G'})$-special pair and $\vert G_i\vert \geq \vert C\vert \geq m_{G,G'}$ for every $1\leq i\leq p$, the subgroup $\varphi'\circ \varphi(G_i)$ of $G$ is conjugate to $G_i$. Hence, $\sigma'\circ \sigma$ is the identity of $\llbracket 1,p\rrbracket$. Likewise, $\sigma\circ \sigma'$ is the identity of $\llbracket 1,p'\rrbracket$. It follows that $p=p'$ and that $\sigma'=\sigma^{-1}$, which concludes the proof.\end{proof}

Note that the previous proposition holds in particular if $G$ and $G'$ are both finite-by-free. We need to prove that this result remains true if only one of these two groups is assumed to be finite-by-free.

\begin{co}\label{f-by-f}Let $G$ and $G'$ be two virtually free groups. Suppose that $G$ is finite-by-free (possibly finite or finite-by-$\mathbb{Z}$). Suppose that there exists a strongly $(\geq m_{G,G'})$-special pair of homomorphisms $(\varphi : G \rightarrow G', \varphi' : G' \rightarrow G)$. Then there exists a $(\geq m_{G,G'})$-expansion $(\Gamma,\Gamma',\psi,\psi')$ of $(G,G',\varphi,\varphi')$ such that $\psi: \Gamma\rightarrow\Gamma'$ and $\psi': \Gamma'\rightarrow\Gamma$ are bijective.\end{co}

\begin{proof}Since $G$ is finite-by-free, there is a unique finite normal subgroup $C\subset G$ such that $G/C\simeq F_n$, with $n\geq 0$. Let $T'$ be a reduced Stallings tree of $G'$. First, note that all vertex groups of $T'$ have order equal to $\vert C\vert$. Indeed, if $v$ is a vertex of $T'$, then $\varphi'(G'_v)$ is a subgroup of $C$. Hence $\varphi\circ \varphi(G'_v)$ is contained in $\varphi(C)$, which is of order $\vert C\vert$. The vertex group $G'_v$ being finite maximal (since $T'$ is reduced), and $\varphi\circ \varphi(G'_v)$ being conjugate to $G'_v$, we have $\vert G'_v\vert=\vert C\vert$.

In order to prove that $G'$ is finite-by-free, it suffices to show that all edge groups of $T'$ have order equal to $\vert C\vert$, as all vertex groups of $T'$ have order equal to $\vert C\vert$. Assume towards a contradiction that there is an edge $e=[v,w]$ of $T'$ such that $\vert G'_e\vert < \vert C\vert$. Since $G'_v$ and $G'_w$ have order $\vert C\vert$, we have $\varphi'(G'_v)=\varphi'(G'_w)=C$. It follows that $G'_v$ and $G'_w$ are conjugate in $G'$ since $\varphi'$, being $(\geq m_{G,G'})$-strongly special, maps non-conjugate finite subgroups of $G'$ of order $\geq m_{G,G'}$ to non-conjugate finite subgroups of $G$. Let $t$ be an element of $G'$ such that $G'_w=tG'_vt^{-1}$, let $E_1:=G'_e$ and $E_2:=tE_1t^{-1}$. 

Let $S$ be the tree obtained from $T'$ by collapsing all edges of $T'$ that are not in the $G'$-orbit of $e$. The segment $[w,tv]$ fixed by $G'_w=tG'_vt^{-1}$ is collapsed to a point in $S$. Hence $t$ has a translation length equal to $1$ in $S$, i.e.\ $t$ is a stable letter in the splitting of $G$ as an HNN extension whose $S$ is the Bass-Serre tree. Let $x$ be the image of $v$ (or $w$) in $S$ and let $H=G'_x$ be its stabilizer. Note that $E_1\subset G'_w\subset H$ and $E_2\subset G'_w\subset H$. We have $G'=\langle H,t \ \vert \ txt^{-1}=\alpha(x), \forall x\in E_1\rangle=H_{\alpha : E_1\rightarrow E_2}$ where $\alpha$ denotes an isomorphism from $E_1$ to $E_2$. Observe that $x$ is the unique vertex of $S$ fixed by $G'_w$, because $\vert G'_w\vert > \vert E_1\vert $. Therefore $N_{G'}(G'_w)$ fixes $x$, i.e.\ $N_{G'}(G'_w)\subset H$. 

Now, let us observe that the homomorphism $\varphi\circ \varphi'$ coincides on the finite subgroup $G'_w$ with an inner automorphism $\mathrm{ad}(g')$, for a certain $g'\in G'$. Up to replacing $\varphi$ by $\mathrm{ad}(g'^{-1})\circ \varphi$, one can assume without loss of generality that $\varphi\circ \varphi'$ coincides with the identity on $G'_w$. In particular, $\varphi\circ \varphi'$ coincides with the identity on $E_1$ and $E_2$. Let $z:=\varphi\circ \varphi'(t)$. We have $zE_1z^{-1}=E_2$ and $tE_1t^{-1}=E_2$. Therefore $z^{-1}t$ normalizes $E_1$. In addition, $z$ normalizes $G'_w$; indeed, $G'_w=tG'_vt^{-1}$ and $G'_w=\varphi\circ \varphi'(G'_w)=\varphi\circ \varphi'(G'_v)$, thus $G'_w=\varphi\circ \varphi'(t)G'_w\varphi'\circ \varphi'(t)^{-1}$. Hence, $z$ belongs to $H$ since $N_{G'}(G'_w)\subset H$. Thus, up to replacing $t$ by $z^{-1}t$ in the previous splitting of $G'$ as an HNN extension, we get a splitting of $G'$ of the form $G'=\langle H,t \ \vert \ txt^{-1}=\alpha(x), \forall x\in E_1\rangle$ where $\alpha$ denotes an automorphism of $E_1$. This shows that $N_{G'}(E_1)$ is infinite and that $E_{G'}(N_{G'}(E_1))=E_1$. By applying the strongly special homomorphism $\varphi'$ to this equality, we get $E_G(N_G(\varphi'(E_1)))=\varphi'(E_1)$. This is a contradiction, because $\vert \varphi'(E_1)\vert < \vert C\vert $ and $C$ is normal in $G$.

As a conclusion, all edge groups of $T'$ have order $\vert C\vert$. As mentionned above, this shows that the group $G'$ is finite-by-free. Let $C'\subset G'$ be the unique finite normal subgroup such that $G'/C'$ is free. Note that $\vert C\vert =\vert C'\vert =m_{G,G'}$ and $G'=N_{G'}(C')$. Since $\varphi'$ is strongly $(\geq m_{G,G'})$-special, $\varphi'(G')$ is non-elementary as soon as $G'$ is non-elementary, and $\varphi'(G')$ is infinite as soon as $G'$ is infinite. By symmetry, $\varphi(G)$ is non-elementary as soon as $G$ is non-elementary, and $\varphi(G)$ is infinite as soon as $G'$ is infinite. Consequently, $G$ and $G'$ are simultaneously finite, virtually $\mathbb{Z}$ or non-elementary. We treat the three cases separately.

\emph{First case.} Suppose that $G$ and $G'$ are finite. Since the homomorphisms $\varphi$ and $\varphi'$ are injective on finite groups, they are bijective, and one can take $\Gamma=G$ and $\Gamma'=G'$.

\emph{Second case.} Suppose that $G$ and $G'$ are virtually $\mathbb{Z}$. Note that $G'$ can be written as $G'=G\ast_G G'$, where the embedding of $G$ into $G$ is the identity, and the embedding of $G$ into $G'$ is the nice embedding $\varphi : G \hookrightarrow G'$. Moreover, $\varphi' : G' \hookrightarrow G$ is a nice embedding. Hence $G'$ is a legal small extension of $G$. One can take $\Gamma=\Gamma'=G'$. 

\emph{Third case.} Suppose that $G$ and $G'$ are non-elementary. Then the existence of $\Gamma$ and $\Gamma'$ is an immediate consequence of Proposition \ref{f-by-f0} above.\end{proof}

\subsection{A property of the tree of cylinders}

Let $G$ be a virtually free group. Let $m\geq 1$ be an integer, and let $T$ be a $m$-splitting of $G$. Recall that the tree of cylinders $T_c$ (see \cite{GL11} and Section \ref{tree}) is the bipartite tree whose set of vertices $V(T_c)$ is the disjoint union of the following two sets:
\begin{itemize}
\item[$\bullet$]the set of vertices $x$ of $T$ which belong to at least two cylinders, denoted by $V_0(T_c)$;
\item[$\bullet$]the set of cylinders of $T$, denoted by $V_1(T_c)$.
\end{itemize}
There is an edge $\varepsilon=(x, Y)$ between $x\in V_0(T_c)$ and $Y\in V_1(T_c)$ in $T_c$ if and only if $x\in Y$. If $Y=\mathrm{Fix}(G_e)$ is the cylinder associated with an edge $e\in T$, then the stabilizer $G_Y$ of $Y$ is $N_G(G_e)$.

\begin{lemme}\label{perin}Let $G$ be a virtually free group. Let $m\geq 1$ be an integer, and let $T$ be a $m$-splitting of $G$. Let $T_c$ denote the tree of cylinders of $T$. Let $\varphi$ be an endomorphism of $G$. Suppose that $\varphi$ maps every vertex group of $T_c$ isomorphically to a conjugate of itself, and every finite subgroup of $G$ isomorphically to a conjugate of itself. Then $\varphi$ is an automorphism.\end{lemme}

\begin{proof}If $T_c$ is a point (i.e.\ if there is a unique cylinder in $T$), then $G$ is a vertex group of $T_c$ and $\varphi$ is an automorphism. From now on, we suppose that $T_c$ is not a point.

As a first step, we build a $\phi$-equivariant map $f:T\rightarrow T$. Let $v_1,\ldots ,v_n$ be some representatives of the orbits of vertices of $T$. For every $1\leq i\leq n$, there exists an element $g_i\in G$ such that $\phi(G_{v_i})=G_{v_i}^{g_i}$. We let $f(v_i)=g_i\cdot v_i$. Then we extend $f$ linearly on edges of $T$.


We claim that the map $f$ induces a $\phi$-equivariant map $f_c : T_c \rightarrow T_c$. Indeed, for each cylinder $Y=\mathrm{Fix}(C)\subset T$, the image $f(Y)$ is contained in $\mathrm{Fix}(\varphi(C))$ of $T$, which is a cylinder (not a point) since $\varphi(C)$ is conjugate to $C$. If $v\in T$ belongs to two cylinders, so does $f(v)$. This allows us to define $f_c$ on vertices of $T_c$, by sending $v\in V_0(T_c)$ to $f(v)\in V_0(T_c)$ and $Y\in V_1(T_c)$ to $f(Y)\in V_1(T_c)$. If $(v,Y)$ is an edge of $T_c$, then $f_c(v)$ and $f_c(Y)$ are adjacent in $T_c$.

We shall prove that $f_c$ does not fold any pair of edges and, therefore, that $f_c$ is injective. Assume towards a contradiction that there exist a vertex $v$ of $T_c$, and two distinct vertices $w$ and $w'$ adjacent to $v$ such that $f_c(w)=f_c(w')$. 

First, assume that $v$ is not a cylinder. Since $T_c$ is bipartite, $w$ and $w'$ are two cylinders, associated with two edges $e$ and $e'$ of $T$. Since $f_c(w)=f_c(w')$, we have $\phi(G_{\varepsilon})=\phi(G_{\varepsilon'})$ by definition of $f_c$. But $\phi$ is injective on $G_v$ by hypothesis, and $G_{\varepsilon},G_{\varepsilon'}$ are two distinct subgroups of $G_v$ (by definition of a cylinder). This is a contradiction.

Now, assume that $v=Y_{\varepsilon}$ is a cylinder. Since $f_c(w)=f_c(w')$, there exists an element $g\in G$ such that $w'=g\cdot w$. As a consequence $\phi(g)$ belongs to $\phi(G_w)$, so one can assume that $\phi(g)=1$ up to multiplying $g$ by an element of $G_w$. In particular, it follows that $g$ does not belong to $N_G(G_{\varepsilon})=G_v$, since $\phi$ is injective in restriction to $N_G(G_{\varepsilon})=G_v$. Then observe that $G_{\varepsilon}\subset G_w$, $G_{\varepsilon}\subset G_{w'}$ and $gG_{\varepsilon}g^{-1}\subset gG_wg^{-1}=G_{w'}$. We have $G_{\varepsilon}\neq gG_{\varepsilon}g^{-1}$ since $g$ does not lie in $N_G(G_{\varepsilon})$, but $\phi(G_{\varepsilon})=\phi(gG_{\varepsilon}g^{-1})$ since $\phi(g)=1$. This contradicts the injectivity of $\phi$ on $G_{w'}$.

Hence, $f_c$ is injective. It follows that $\phi$ is injective. Indeed, let $g$ be an element of $G$ such that $\phi(g)=1$. Then $f_c(g\cdot v)=f_c(v)$ for each vertex $v$ of $T_c$, so $g\cdot v=v$ for each vertex $v$ of $T_c$. But $\phi$ is injective on vertex groups of $T_c$, so $g=1$.

It remains to prove the surjectivity of $ \phi $. We begin by proving the surjectivity of $ f_c $. It suffices to prove the local surjectivity. Let $ v $ be a vertex of $ T $ and $ e $ an edge adjacent to $ v $. Up to conjugacy, we can assume that $ f_c(v) = v $ and $ f_c(e) = e $. We thus have $ f_c(G_v \cdot e) = \phi (G_v) \cdot f_c(e) = G_v \cdot f_c(e) = G_v \cdot e $. Therefore, all the translates of $ e $ by an element of $ G_v $ are in the image of $ f_c$, which proves the surjectivity of $ f_c$. It remains to prove the surjectivity of $ \phi $. Let $ g \in G $ and let $ w $ be a vertex. There are two vertices $ v $ and $ v'$ such that $ f_c(v) = w $ and $ f_c(v') = g w $. Hence there exists $ h \in G $ such that $ v'= hv $, so $ f_c(v') = f_c(hv) = \phi (h) w $, i.e.\ $ gw = \phi (h) w $, so $ g^{- 1} \phi (h) $ belongs to $ G_w = \phi (G_v) $, hence $ g = \phi (h) g'$ with $ g' \in G_v $ . As a consequence, $ \phi $ is surjective.\end{proof}

\subsection{A key proposition}

\begin{de}Let $G$ and $G'$ be two virtually free groups, and let $\varphi,\psi : G \rightarrow G'$ and $\varphi',\psi' : G'\rightarrow G$ be homomorphisms. We say that $(\psi,\psi')$ is a \emph{power} of $(\varphi,\varphi')$ if there exist two integers $n,n'\geq 0$ such that $\psi_{\vert G}=\varphi\circ (\varphi'\circ \varphi)^n$ and $\psi'_{\vert G'}=\varphi'\circ (\varphi\circ \varphi')^{n'}$.\end{de}

\begin{de}Let $G$ and $G'$ be two virtually free groups, and let $\varphi,\psi : G \rightarrow G'$ and $\varphi',\psi' : G'\rightarrow G$ be homomorphisms. The notation $(\psi,\psi')\sim (\varphi,\varphi')$ means that $\psi\sim\varphi$ and $\psi'\sim\varphi'$.\end{de}


\begin{prop}\label{ledernierlemme}Let $G$ and $G'$ be two virtually free groups. Let $m$ denote the integer $m_{G,G'}$. Let $\varphi : G \rightarrow G'$ and $\varphi' : G'\rightarrow G$ be two homomorphisms. If $(G,G',\varphi,\varphi')$ has property $\mathcal{P}_{m}$, then there exists a pair $(\psi : G \rightarrow G',\psi' : G'\rightarrow G)$ which is equivalent (in the sense of $\sim$) to a power of $(\varphi,\varphi')$ and a $m$-expansion $(\widehat{G},\widehat{G}',\widehat{\varphi},\widehat{\varphi}')$ of $(G,G',\psi,\psi')$ such that $\widehat{\varphi}$ and $\widehat{\varphi}'$ are isomorphisms.\end{prop}

\begin{proof}Let $T$ and $T'$ be reduced $m$-JSJ splittings of $G$ and $G'$. First, let us observe that $\varphi$ maps every $m$-factor of $G$ isomorphically to a $m$-factor of $G'$ and that $\varphi'$ maps every $m$-factor of $G'$ isomorphically to a $m$-factor of $G$ (see Remark \ref{capermute}).

If $G$ or $G'$ is finite-by-free, the result is an immediate consequence of Corollary \ref{f-by-f}. From now on, we suppose that $G$ and $G'$ are not finite-by-free. We decompose the proof of the lemma into three steps.




\medskip

\textbf{Step 1.} We claim that there exists a $m$-expansion $(G_1,G'_1,\varphi_1,\varphi'_1)$ of $(G,G',\varphi,\varphi')$ such that $\varphi_1$ and $\varphi'_1$ maps edge groups of $m$-JSJ splittings $T_1$ and $T'_1$ of $G_1$ and $G'_1$ to edge groups of $T'_1$ and $T_1$ respectively.\footnote{Note that this condition is not automatically satisfied, as shown by the following example: take $G$ that does not split non-trivially as a free product, and $G'=G\ast F_n$.}





\medskip

\emph{Proof of Step 1.} If $\varphi(C)$ is an edge group of $T'$ for every edge group $C$ of $T$, and if $\varphi'(C')$ is an edge group of $T$ for every edge group $C'$ of $T'$, then one can take $G_1=G$, $\varphi_1=\varphi$, $G'_1=G'$, and $\varphi'_1=\varphi'$. Now, let us suppose that there is an edge $e$ of $T$ such that $\varphi(G_e)$ is not the stabilizer of an edge of $T'$. Let $C:=G_e$. We shall prove that the group $\widehat{G}'=\langle G',t \ \vert \ \mathrm{ad}(t)_{\vert \varphi(C)}=\mathrm{id}_{\varphi(C)}\rangle$ is a legal $m$-extension of $G'$. Note that $\varphi(C)$ is an edge group in any $m$-JSJ tree $\widehat{T}'$ of $\widehat{G}$. In addition, one easily sees that the homomorphism $\widehat{\varphi}' : \widehat{G}'\rightarrow G$ defined by $\widehat{\varphi}'(t)=1$ and $\widehat{\varphi}'(g')=\varphi(g')$ for every $g'\in G'$ satisfies the following three conditions:
\begin{itemize}
\item[$\bullet$]the pair $(\varphi,\widehat{\varphi}')$ is strongly $(\geq m)$-special;
\item[$\bullet$]${\widehat{\varphi}'}_{\vert G'}\sim \varphi'$;
\item[$\bullet$]$\widehat{\varphi}'$ maps every $m$-factor of $\widehat{G}'$ isomorphically to a $m$-factor of $G$.
\end{itemize}
Hence, one can define the group $G'_1$, and symmetrically the group $G_1$, by iterating the construction described above finitely many times, since $T$ and $T'$ have only finitely many orbits of edges.


It remains to prove that the group $\widehat{G}'=\langle G',t \ \vert \ \mathrm{ad}(t)_{\vert \varphi(C)}=\mathrm{id}_{\varphi(C)}\rangle$ is a legal $m$-extension of $G'$, under the hypothesis that $\varphi(C)$ is not contained in the stabilizer of an edge of $T'$.

Since the group $\varphi(C)$ is not contained in the stabilizer of an edge of $T'$, it fixes a unique vertex $v'$ of $T'$. There exists an edge $e=[v,w]$ of $T$ such that $G_e=C$, $\varphi(G_v)\subset G'_{v'}$ and $\varphi(G_w)\subset G'_{v'}$. Moreover, recall that $\varphi$ maps $G_v$ and $G_w$ isomorphically to $m$-factors of $G'$. Therefore, the following equalities hold: $\varphi(G_v)=\varphi(G_w)=G'_{v'}$ and $\varphi'\circ \varphi(G_v)=\varphi'\circ \varphi(G_w)$. Since $\varphi'\circ \varphi$ maps every $m$-factor of $G$ to a conjugate of itself, there exists an element $g\in G$ such that $G_w=gG_vg^{-1}$. Thus we have $\varphi(G_v)=\varphi(gG_vg^{-1})=\varphi(g)\varphi(G_v)\varphi(g)^{-1}$. This shows that $\varphi(g)$ belongs to $N_{G'}(\varphi(G_v))$. Since $G$ is not finite-by-free, Lemma \ref{JE SAIS PAS!!} asserts that $G_v$ has order $>m$. This implies that $N_{G'}(\varphi(G_v))=\varphi(G_v)$. Hence $\varphi(g)$ belongs to $\varphi(G_v)$. There is an element $h\in G_v$ such that $\varphi(g)=\varphi(h)$. Let $k=gh^{-1}$, so that $\varphi(k)=1$ and $w=gv=kv$. Note that $k\neq 1$, since $w=kv\neq v$. Now, note that $\varphi(C)=\varphi(kCk^{-1})$. Since $C$ and $kCk^{-1}$ are contained in $G_w$, and since $\varphi$ is injective on $G_w$, this proves that $C=kCk^{-1}$, i.e.\ that $k$ belongs to the normalizer $N_{G}(C)$ of $C$ in $G$. In particular, $N_G(C)$ must be non-elementary, otherwise $\varphi$ would be injective on $N_G(C)$, contradicting the fact that $\varphi(k)=1$. In order to prove that the group $\widehat{G}'$ defined above is a legal large $m$-extension of $G'$, we need to prove that $\varphi(C)=E_{G'}(N_{G'}(\varphi(C)))$. First, let us prove that $C=E_G(N_G(C))$. Note that the inclusion $C\subset E_G(N_G(C))$ always holds, because $E_G(N_G(C))$ is the unique maximal finite subgroup of $C$ normalized by $N_G(C)$. Assume towards a contradicton that the inclusion $C\subset E_G(N_G(C))$ is strict. Then $E_G(N_G(C))$ has order $>m$, so it fixes a unique vertex $x\in T$, because edge groups of $T$ have order equal to $m$. This implies that $x$ is fixed by $N_G(C)$. But $N_G(C)$ is not elliptic in $T$ since the element $k\in N_G(C)$ acts hyperbolically on $T$. This is a contradiction. We have proved that $E_G(N_G(C))=C$. It follows that $E_{G'}(N_{G'}(\varphi(C)))=\varphi(C)$, because $\varphi$ is strongly $(\geq m)$-special and $C$ has order $m$. Hence, the group $G'_1=\langle G', t \ \vert \ \mathrm{ad}(t)_{\vert \varphi(C)}=\mathrm{id}_{\varphi(C)}\rangle$ is a legal large $m$-extension of $G$, which concludes the proof of the first step.

Now, up to replacing $(G, G',\varphi, \varphi')$ with $(G_1,G'_1,\varphi_1,\varphi'_1)$, one can suppose that $(G, G',\varphi, \varphi')$ has the property of Step 1.

\medskip

\textbf{Step 2.} We claim that there exists a power $(\rho,\rho')$ of $(\varphi,\varphi')$ and a $m$-expansion $(G_2,G'_2,\varphi_2,\varphi'_2)$ of $(G, G',\rho, \rho')$ such that the following two conditions hold: 
\begin{itemize}
\item[$\bullet$]for every edge group $C$ of $T_2$, the group $\varphi_2(C)$ is an edge group of $T'_2$ and $\varphi_2$ maps $N_{G_2}(C)$ isomorphically to $N_{G'_2}(\varphi_2(C))$;
\item[$\bullet$]for every edge group $C'$ of $T'_2$, the group $\varphi'_2$ is an edge group of $T_2$ and $\varphi'_2$ maps $N_{G'_2}(C')$ isomorphically to $N_{G_2}(\varphi'_2(C'))$.
\end{itemize}

\medskip

\emph{Proof of Step 2.} Let $T$ be a $m$-JSJ splitting of $G$, and let $T'$ be a $m$-JSJ splitting of $G'$. Let $C_1,\ldots ,C_q$ be a set of representatives of the conjugacy classes of edge groups of $T$. Let $C'_1,\ldots ,C'_{q'}$ be a set of representatives of the conjugacy classes of edge groups of $T'$. Thanks to the first step, we know that $q=q'$ and, up to renumbering the edges of $T$, one can assume that $\varphi(C_i)=g'_iC'_i{g'_i}^{-1}$ for a certain element $g'_i\in G'_1$ and that $\varphi'(C'_i)=g_iC_ig_i^{-1}$ for a certain element $g_i\in G$. For every $1\leq i\leq q$, let $N_i=N_{G}(C_i)$ and $N'_i=N_{G'}(C'_i)$. Let $T_{c}$ and $T'_{c}$ be the trees of cylinders of $T$ and $T'$. Recall that $N_1,\ldots, N_q$ are the new vertex groups of $T_{c}$ and that $N'_1,\ldots,N'_q$ are the new vertex groups of $T'_{c}$.

Let $i\in\llbracket 1,q\rrbracket$. Let $\varphi_i=\mathrm{ad}({g'_i}^{-1})\circ \varphi_{\vert N_i}$ and $\varphi'_i=\mathrm{ad}({g_i}^{-1})\circ \varphi'_{\vert N'_i}$. Let $Y_i$ be the cylinder of $C_i$ in $T$. Recall that $Y_i$ is connected (see Section \ref{tree}). Moreover, note that $N_i$ acts cocompactly on $Y_i$, because two edges of $Y_i$ are in the same $G$-orbit if and only if they are in the same $N_i$-orbit. As a consequence, the action of $N_i$ on $Y_i$ gives a decomposition of $N_i$ as a graph of groups, all of whose edges groups are equal to $C_i$. Let $Y'_i$ be the cylinder of $C'_i$ in $T'_1$. 

We claim that, for every $1\leq i\leq q$, there exists a power $(\rho_i,\rho'_i)$ of $(\varphi_i,\varphi'_i)$ and a $m$-expansion $(\widehat{N}_i,\widehat{N}'_i,\widehat{\varphi}_i,\widehat{\varphi}'_i)$ of $(N_i,N'_i,\rho_i,\rho'_i)$ such that $\widehat{\varphi}_i : \widehat{N}_i \rightarrow \widehat{N}'_i$ and $\widehat{\varphi}'_i : \widehat{N}'_i \rightarrow \widehat{N}_i$ are bijective and coincide with $\rho_i$ and $\rho'_i$, up to conjugacy, on the edge groups adjacent to the vertices fixed by $N_i$ and $N'_i$ in $T_c$ and $T'_c$ (this condition on the edge groups will be necessary in order to apply Lemma \ref{s'étend}).

The homomorphisms $\varphi$ and $\varphi'$ being strongly $(\geq m)$-special, and the edge groups of $T$ and $T'$ having order equal to $m$, the groups $N_i$ and $N'_i$ are simultaneously finite, virtually cyclic infinite, or non-elementary. We treat separately the three cases.

\smallskip

\emph{First case.} If $N_i$ and $N'_i$ are finite, ${\varphi}_{\vert N_i} : N_i \rightarrow N'_i$ and ${\varphi}_{\vert N'_i} : N'_i \rightarrow N_i$ are injective. Thus, they are bijective. There is nothing to be done: we simply take $\widehat{N}_i=N_i$, $\widehat{N}'_i=N'_i$, $\widehat{\varphi}_i={\varphi}_{\vert N_i}$ and $\widehat{\varphi}'_i={\varphi'}_{\vert N'_i}$.

\smallskip

\emph{Second case.} Suppose that $N_i$ and $N'_i$ are virtually cyclic infinite. Let us observe that $N_i$ is elliptic in $T$ if and only if $N'_i$ is elliptic in $T$. Indeed, if $N_i$ fixes a vertex $v\in T$, then $\varphi(N_i)$ fixes the vertex $v'$ fixed by $\varphi(G_v)$ in $T'$. Then, note that $\varphi(N_i)$ has finite index in $N'_i$, because $\varphi(N_i)$ is infinite and $N'_i$ is virtually cyclic. It follows that $N'_i$ fixes $v'$.

\emph{First subcase.} If $N_i$ and $N'_i$ are elliptic in $T$ and $T'$, then $\varphi$ and $\varphi'$ induce isomorphisms between $N_i$ and $N'_i$, because $\varphi$ and $\varphi'$ induce isomorphisms between vertex groups of $T$ and $T'$. There is nothing to be done: we simply take $\widehat{N}_i=N_i$, $\widehat{N}'_i=N'_i$, $\widehat{\varphi}_i={\varphi}_{\vert N_i}$ and $\widehat{\varphi}'_i={\varphi'}_{\vert N'_i}$.

\emph{Second subcase.} If $N_i$ and $N'_i$ are not elliptic in $T$ and $T'$, we take $\widehat{N}_i=\widehat{N}'_i=N'_i$, and we take for $\widehat{\varphi}_i$ and $\widehat{\varphi}'_i$ the identity of $N'_i$. The group $\widehat{N}_i$ is a legal small extension of $N_i$, with nice embeddings $\varphi_i : N_i \hookrightarrow \widehat{N}_i$ and $\varphi'_i : \widehat{N}_i \hookrightarrow N_i$, and the group $\widehat{N}'_i$ is a legal small extension of $N'_i$, with nice embeddings $\iota=\varphi_i\circ \varphi'_i : \widehat{N}'_i \hookrightarrow N'_i$ and $\iota'=\varphi_i\circ \varphi'_i : N'_i\hookrightarrow \widehat{N}'_i $.

\begin{center}
\begin{tikzcd}\widehat{N}_i=N'_i\arrow[r, "\widehat{\varphi}_i=\mathrm{id}"]\arrow[d,"\varphi'_i"]& \widehat{N}_i'=N'_i \arrow[d, "\iota" ] \\  N_i\arrow[r, "\varphi_i"]& N'_i\end{tikzcd} \hspace{1cm}\begin{tikzcd}\widehat{N}_i=N'_i& \widehat{N}_i'=N'_i \arrow[l, "\widehat{\varphi}'_i=\mathrm{id}"] \\  N_i\arrow[u, "\varphi_i"] & N'_i\arrow[u, "\iota'"]\arrow[l, "\varphi_i'"]\end{tikzcd}
\end{center}

Note that the edge groups adjacent to the vertices fixed by $N_i$ and $N'_i$ in $T_c$ and $T'_c$ are finite. Indeed, let $e$ be an edge of $T_c$ adjacent to the vertex fixed by $N_i$ in $T_c$. Observe that $G_e$ is elliptic in $T$, by definition of the tree of cylinders $T_c$. As a consequence, if $G_e\subset N_i$ were infinite, then $N_i$ would be ellitptic in $T$, contradicting our hypothesis. As a conclusion, $\widehat{\varphi}_i$ and $\widehat{\varphi}'_i$ coincide with $\varphi_i$ and $\varphi'_i$, up to conjugacy, on the edge groups adjacent to the vertices fixed by $N_i$ and $N'_i$ in $T_c$ and $T'_c$

\smallskip

\emph{Third case.} Last, suppose that $N_i$ and $N'_i$ are non-elementary. Recall that the inclusion $C_i\subset E_{G}(N_i)$ always holds because $E_{G}(N_i)$ is the unique maximal finite subgroup of $G$ normalized by $N_i$. Likewise, $C'_i$ is contained in $E_{G'}(N'_i)$. We distinguish two cases.

\emph{First subcase.} If $E_i:=E_{G}(N_i)$ contains $C_i$ strictly, it has order $> m$. Therefore, it fixes a unique vertex $v_i\in T$, and $N_i$ fixes $v_i$ as well. By definition of a strongly $(\geq m)$-special homomorphism, $\varphi(E_i)=E'_i:=E_{G'}(N_{G'}(\varphi(C_i)))$ and $\varphi(C_i)$ is conjugate to $C'_i$. Consequently, $N'_i$ fixes the unique vertex $v'_i$ of $T'$ fixed by $\varphi(E_i)$. Since $\varphi$ induces an isomorphism from $(G)_{v_i}$ to $(G')_{v'_i}$, it induces an isomorphism from $N_i$ to $N'_i$. There is nothing to be done: we simply take $\widehat{N}_i=N_i$, $\widehat{N}'_i=N'_i$, $\widehat{\varphi}_i={\varphi}_{\vert N_i}$ and $\widehat{\varphi}'_i={\varphi'}_{\vert N'_i}$.

\emph{Second subcase.} If $E_{G}(N_i)=C_i$, then $E_{G'}(N'_i)=C'_i$ since $\varphi$ is strongly $(\geq m)$-special. We will use Proposition \ref{normalisateur} in order to establish the existence of a power $(\rho_i,\rho'_i)$ of $(\varphi_i,\varphi'_i)$ and of a $m$-expansion $(\widehat{N}_i,\widehat{N}'_i,\widehat{\varphi}_i,\widehat{\varphi}'_i)$ of $(N_i,N'_i,\rho_i,\rho'_i)$ satisfying the properties announced above. Before using Proposition \ref{normalisateur}, we will prove that the cylinders $Y_i$ and $Y'_i$ of $C_i$ and $C'_i$, endowed with actions of $N_i$ and $N'_i$ respectively, satisfy the following conditions:
\begin{enumerate}
\item $Y_i$ and $Y'_i$ are $m$-splittings of $N_i$ and $N'_i$ with the same number of orbit of vertices;
\item $\varphi_i$ and $\varphi'_i$ induce bijections between the conjugacy classes of vertices of $Y_i$ and $Y'_i$, and induce isomorphisms between the vertex groups. 
\end{enumerate}

First, note that each vertex group of the $N_i$-tree $Y_i$ is of the form $N_i\cap G_v$ for some $v\in Y_i$. By hypothesis, $\varphi'\circ \varphi(G_v)=gG_vg^{-1}$ for some $g\in G$. Let $\psi'=\mathrm{ad}(g^{-1})\circ \varphi'$ and $\psi=\varphi$, so that $\psi'\circ \psi(G_v)=G_v$. Since $\psi'\circ \psi$ maps non-conjugate finite subgroups to non-conjugate finite subgroups, and since $G_v$ has only finitely many conjugacy classes of finite subgroups, there exists an integer $n\geq 1$ such that $(\psi'\circ \psi)^n(C_i)=g_vC_ig_v^{-1}$ with $g_v\in G_v$. Let $\rho'=\mathrm{ad}(g_v)\circ (\psi'\circ \psi)^{n-1}\circ \psi'$ and $\rho=\psi$, so that $\rho'\circ\rho(G_v)=G_v$ and $\rho'\circ \rho(C_i)=C_i$. As a consequence, $\rho'\circ \rho(N_i)\subset N_i$. 

Let $v_1,\ldots,v_r$ be some representatives of the $G$-orbits of vertices of $Y_i$. We can define iteratively two homomorphisms $\rho : G \rightarrow G'$ and $\rho' : G' \rightarrow G$ such that, for every $1\leq j\leq r$, there exists an element $g_j\in G$ such that $\mathrm{ad}(g_j)\circ\rho'\circ \rho (G_{v_j})=G_{v_j}$ and $\mathrm{ad}(g_j)\circ\rho'\circ \rho(C_i)=C_i$. Up to replacing $\rho'$ by $\mathrm{ad}(g_j^{-1})\circ \rho'$, we can suppose without loss of generality that $g_1=1$. Hence $\rho'\circ \rho(C_i)=C_i$ and $\mathrm{ad}(g_j)\circ\rho'\circ \rho(C_i)=C_i$ for every $j\geq 2$. This shows that $g_j$ normalizes $C_i$, i.e.\ that $g_j$ belongs to $N_i$.

We have proved that the homomorphism $\rho'\circ \rho$ maps every vertex group $N_i\cap G_v$ of $N_i$ isomorphically to a $N_i$-conjugate of itself. Moreover, one can suppose that $\rho\circ \rho'$ maps every vertex group of $N'_i$ isomorphically to a $N'_i$-conjugate of itself (we repeat the same operation described above with $N'_i$ instead of $N_i$, and this does not affect the property satisfied by $\rho'\circ \rho$). Note that the pair $(\rho,\rho')$ is a power of $(\varphi,\varphi')$.
 
Now, the existence of a $m$-expansion $(\widehat{N}_i,\widehat{N}'_i,\widehat{\varphi}_i,\widehat{\varphi}'_i)$ of $(N_i,N'_i,\rho_i,\rho'_i)$ such that $\widehat{\varphi}_i : \widehat{N}_i \rightarrow \widehat{N}'_i$ and $\widehat{\varphi}'_i : \widehat{N}'_i \rightarrow \widehat{N}_i$ are bijective and coincide with $\rho_i$ and $\rho'_i$, up to conjugacy, on the edge groups adjacent to the vertices fixed by $N_i$ and $N'_i$ in $T_c$ and $T'_c$ follows from Lemma \ref{normalisateur}. This concludes the proof of the second subcase.

\smallskip

The group $G_2$ obtained from $G$ by replacing each vertex group $N_i$ of $T_{c}$ by $\widehat{N}_i$ is a multiple legal $m$-extension of $G$. Similarly, the group $G'_2$ obtained from $G'$ by replacing each vertex group $N'_i$ of $T_{c}$ by $\widehat{N}'_i$ is a multiple legal $m$-extension of $G'$.

Last, since the morphisms $\widehat{\varphi}_i : \widehat{N}_i \rightarrow \widehat{N}'_i$ and $\widehat{\varphi}'_i : \widehat{N}'_i \rightarrow \widehat{N}_i$ coincide with $\varphi_i$ and $\varphi'_i$, up to conjugacy, on the vertex groups of $Y_i$ and $Y'_i$, and since every edge group of $T_c$ and $T'_c$ is contained in a vertex group of $Y_i$ or $Y'_i$, Lemma \ref{s'étend} guarantees the existence of the homomorphisms $\varphi_2 : G_2\rightarrow G'_2$ and $\varphi'_2: G'_2\rightarrow G_2$ announced above.


\medskip

\textbf{Step 3.} We now prove that the homomorphisms $\varphi_2 : G_2 \rightarrow G'_2$ and $\varphi'_2 : G'_2 \rightarrow G_2$ constructed previously are bijective. To that end, we use Lemma \ref{perin}. Let $T_2$ and $T'_2$ denote two recuced $m$-JSJ splittings of $G_2$ and $G'_2$, and let $T_{2,c}$ and $T'_{c2,c}$ be their trees of cylinders. Recall that $V(T_{2,c})=V_0(T_{2,c})\sqcup V_1(T_{2,c})$, where $V_0(T_{2,c})$ denotes the set of vertices of $T$ belonging to at least two distinct cylinders, and $V_1(T_{2,c})$ denotes the set of cylinders of $T$. Observe that the following facts hold.
\begin{itemize}
\item[$\bullet$]$\varphi_2$ maps every vertex group of $V_1(T_{2,c})$ isomorphically to a vertex group of $V_1(T'_{2,c})$, and $\varphi'_2$ maps every vertex group of $V_1(T'_{2,c})$ isomorphically to a vertex group of $V_1(T_{2,c})$.
\item[$\bullet$]$\varphi_2$ maps every vertex group of $V_0(T_{2,c})$ isomorphically to a vertex group of $V_0(T'_{2,c})$, and $\varphi'_2$ maps every vertex group of $V_0(T'_{2,c})$ isomorphically to a vertex group of $V_0(T_{2,c})$.
\end{itemize}
The first point is a consequence of Step 2. The second point follows from the fact that $v\in T$ belongs to two cylinders if and only if there exists two distinct edge groups $C_1\subset(G_2)_v$ and $C_2\subset(G_2)_v$. By Step 1, $\varphi_2(C_1)$ and $\varphi_2(C_2)$ are edge groups of $T'$, and they are distinct because $\varphi_1$ is injective on vertex groups of $T$. Consequently, $\varphi_2((G_2)_v)$ is a vertex group of $T'$. Last, Lemma \ref{perin} ensures that $\varphi'_2\circ \varphi_2$ is an automorphism of $G_2$. Thus, $\varphi_2$ and $\varphi'_2$ are two isomorphisms. This concludes the proof.\end{proof}

\subsection{Proof of $(4)\Rightarrow (5)$}

\subsubsection{A particular case}

We first prove the implication $(4)\Rightarrow (5)$ of Theorem \ref{principal} in a particular case.

\begin{prop}\label{2implique3facile}Let $G$ and $G'$ be two virtually free groups. Let $T$ and $T'$ be two reduced JSJ-splittings of $G$ and $G'$ over finite groups. Suppose that all edges groups of $T$ and $T'$ have the same order, say $m$. If there exist two homomorphisms $\varphi : G \rightarrow G'$ and $\varphi' : G' \rightarrow G$ such that the pair $(\varphi, \varphi')$ is strongly special, then there exist two multiple legal extensions $\Gamma$ and $\Gamma'$ of $G$ and $G'$ such that $\Gamma\simeq \Gamma'$.\end{prop}

\begin{proof}
This particular case is an immediate consequence of Proposition \ref{ledernierlemme}.
\end{proof}

\subsubsection{The general case}

We shall now prove the implication $(4)\Rightarrow (5)$ of Theorem \ref{principal} in the general case, that is the following result.

\begin{prop}\label{2implique3}Let $G$ and $G'$ be two virtually free groups. Suppose that there exists a strongly special pair of homomorphisms $(\varphi : G \rightarrow G', \varphi' : G' \rightarrow G)$. Then there exist two multiple legal extensions $\Gamma$ and $\Gamma'$ of $G$ and $G'$ respectively, such that $\Gamma\simeq \Gamma'$.\end{prop}

\begin{proof}We define the complexity of the pair $(G,G')$, denoted by $c(G,G')$, as the sum of the number of edges in reduced Stallings splittings of $G$ and $G'$. We will prove Proposition \ref{2implique3} by induction on the complexity $c(G,G')$. In fact, we will prove a slightly stronger result (see the induction hypothesis below). If $G$ is infinite, recall that we denote by $m(G)$ the smallest order of an edge group in a reduced Stallings splitting of $G$. If $G$ is finite, we set $m(G)=\vert G\vert$. We denote by $m_{G,G'}$ the integer $\min(m(G),m(G'))$.

\medskip

\emph{Induction hypothesis} $H(n)$. For every pair of virtually free groups $(G,G')$ such that $c(G,G')\leq n$, if there exist two homomorphisms $\varphi : G \rightarrow G'$ and $\varphi' : G' \rightarrow G$ such that the pair $(\varphi,\varphi')$ is strongly $(\geq m_{G,G'})$-special, then there exists a power $(\phi,\phi')$ of $(\varphi,\varphi')$ and a $(\geq m_{G,G'})$-expansion $(\Gamma,\Gamma',\psi,\psi')$ of $(G,G',\phi,\phi')$ such that $\psi : \Gamma \rightarrow \Gamma'$ and $\psi' : \Gamma' \rightarrow \Gamma$ are two isomorphisms.



\medskip

The base case $H(0)$ is obvious: if reduced Stallings splittings of $G$ and $G'$ have $0$ edge, then $G$ and $G'$ are finite, and one can take $\Gamma=G$ and $\Gamma'=G'$. 

We now prove the induction step. Let $n\geq 0$ be an integer. Suppose that $H(n)$ holds, and let us prove that $H(n+1)$ holds. To that end, let us consider two virtually free groups $G$ and $G'$ such that $c(G,G')= n+1$, and suppose that there exists a strongly $(\geq m_{G,G'})$-special pair of homomorphisms $(\varphi : G \rightarrow G', \varphi' : G' \rightarrow G)$. 

Let us fix two reduced $m_{G,G'}$-JSJ decompositions of $G$ and $G'$, and let $T$ and $T'$ denote their respective Bass-Serre trees. Note that $T$ or $T'$ can be trivial, but not both at the same time. Let $G_1\ldots ,G_p$ and $G'_1,\ldots,G'_{p'}$ denote the $m_{G,G'}$-factors of $G$ and $G'$, well-defined up to conjugacy. In other words, $G_1\ldots ,G_p$ and $G'_1,\ldots,G'_{p'}$ are some representatives of the conjugacy classes of vertex groups of $T$ and $T'$ respectively. These groups are $m_{G,G'}$-rigid by definition. 


We now prove that there exists a $(\geq m_{G,G'})$-expansion $(\widehat{G},\widehat{G}',\widehat{\varphi},\widehat{\varphi}')$ of $(G,G',\varphi,\varphi')$ with property $\mathcal{P}_{m_{G,G'}}$, which means that the pair $(\widehat{\varphi},\widehat{\varphi}')$ is strongly $(\geq m)$-special, and that $\widehat{\varphi}'\circ \widehat{\varphi}$ and $\widehat{\varphi}\circ \widehat{\varphi}'$ map every $m_{G,G'}$-factor of $G$ and $G'$ isomorphically to a conjugate of itself. 


Note that $m(G)=m(\widehat{G})$, because $\widehat{G}$ is a $(\geq m_{G,G'})$-legal extension of $G$, with $m_{G,G'}\geq m(G)$. Likewise, $m(G')=m(\widehat{G}')$. Therefore, $m_{G,G'}=m_{\widehat{G},\widehat{G}'}$. 


If $G$ or $G'$ is finite-by-free, the result is a consequence of Proposition \ref{f-by-f}. From now on, we suppose that $G$ and $G'$ are not finite-by-free. In particular, according to Lemma \ref{JE SAIS PAS!!}, for every vertex $v$ of $T$ or $T'$ and for every edge $e$ incident to $v$, the edge group $G_e$ is strictly contained in $G_v$.

\textbf{Claim:} the integers $p$ and $p'$ are equal. Moreover, $\varphi(G_i)$ is contained in $g'_i{G'_i}{g'_i}^{-1}$ for some $g'_i\in G'$, and $\varphi'(G'_i)$ is contained in $g_i G_ig_i^{-1}$ for some $g_i\in G$, up to renumbering the subgroups $G_i$.

Let us prove this claim. First, we prove that $\varphi(G_i)$ fixes a unique vertex of $T'$, for every $1\leq i\leq p$. Note that there exists a vertex $v_i\in T$ such that $G_i=G_{v_i}$. Let $T_i$ be a reduced Stallings splitting of $G_i$. 

As a first step, we will prove that each vertex group $(G_i)_w$ of $T_i$ has order $>m_{G,G'}$. If $T_i$ is not reduced to a point, each vertex group $(G_i)_w$ of $T_i$ contains an edge group, and edge groups of $T_i$ have order $>m_{G,G'}$ since $G_i$ is $m_{G,G'}$-rigid. If $T_i$ is reduced to a point, the group $G_i=(G_w)_i$ is finite. By Lemma \ref{JE SAIS PAS!!}, we have $\vert G_i\vert > m_{G,G'}$, because $G$ is not finite-by-free by assumption.

In the previous paragraph, we have proved that each vertex group $(G_i)_w$ of $T_i$ has order $>m_{G,G'}$. Since $\varphi$ is injective on finite subgroups of $G$, the finite group $\varphi((G_i)_w)$ has order $>m_{G,G'}$ as well. But edge groups of $T'$ have order exactly $m_{G,G'}$, so $\varphi((G_i)_w)$ fixes a unique vertex $v'$ of $T'$. We shall prove that $\varphi(G_i)$ fixes this vertex $v'$. Let us consider a vertex $w_2$ adjacent to $w$ in $T_i$. The same argument shows that $\varphi((G_i)_{w_2})$ fixes a unique vertex $v'_2$ of $T'$. But $G_i$ is $m_{G,G'}$-rigid, as a $m_{G,G'}$-factor of $G$, so $(G_i)_{w}\cap (G_i)_{w_2}$ has order $>m_{G,G'}$. As a consequence, $v'_2=v'$. It follows from the connectedness of $T_i$ that, for every vertex $x$ of $T_i$, the group $\varphi((G_i)_x)$ fixes $v'$ and only $v'$. Now, let $g$ be an element of $G_i$. For any vertex $x$ of $T_i$, $\varphi((G_i)_x)$ and $\varphi((G_i)_{gx})$ fixes $v'$ and only $v'$, so $g$ fixes $v'$. Hence, $\varphi(G_i)$ fixes $v'$ and only $v'$.

 
Symmetrically, $\varphi'(G_{v'})$ fixes a unique vertex $v$ of $T$. Since $\varphi'\circ \varphi$ is a conjugacy on finite subgroups, $v$ is a translate of $v_i$ (because $G_i$ has a finite subgroup of order $>m_{G,G'}$, see above), i.e.\ $\varphi'\circ \varphi(G_i)$ is contained in a conjugate of $G_i$. Hence, $\varphi$ and $\varphi'$ induce two inverses bijections of the conjugacy classes of $m_{G,G'}$-factors of $G$ and $G'$. Now, up to renumbering the $m_{G,G'}$-factors, one can assume that $\varphi(G_i)$ is contained in a conjugate $g'_i{G'_i}{g'_i}^{-1}$ of $G'_i$, with $g'_i\in G$, and that $\varphi'(G'_i)$ is contained in a conjugate $g_iG_ig_i^{-1}$ of $G_i$, with $g_i\in G$. This concludes the proof of the claim.

We aim to apply the induction hypothesis $H(n)$ to the pair $(G_i,G'_i)$, for every $1\leq i\leq p$. First, let us observe that the complexity $c(G_i,G'_i)$ is less than $n$. Indeed, at least one of the trees $T$ or $T'$ is not reduced to a point, say $T$, and one gets a Stallings splitting of $G$ by replacing the vertex of $T$ fixed by $G_i$ by the splitting $T_i$ of $G_i$; moreover, the resulting Stallings tree is reduced because $T_i$ is reduced and the vertex groups of $T_i$ have order $>m_{G,G'}$, whereas edge groups of $T$ are of order $m_{G,G'}$, by Lemma \ref{JE SAIS PAS!!}.

In order to apply the induction hypothesis $H(n)$ to $(G_i,G'_i)$, we need a strongly $(\geq m_{G_i,G_i'})$-special pair of homomorphisms $(\varphi_i : G_i\rightarrow G'_i,\varphi'_i: G'_i\rightarrow G_i)$. We define $\varphi_i=\mathrm{ad}({g'_i}^{-1})\circ \varphi_{\vert G_i} : G_i\rightarrow G'_i$ and $\varphi'_i=\mathrm{ad}(g_i^{-1})\circ{\varphi'}_{\vert G'_i}:G'_i\rightarrow G_i$. By Remark \ref{rem29juin}, the pair $(\varphi_i,\varphi'_i)$ is strongly $(>m_{G,G'})$-special. But $m_{G_i,G'_i}=\min(m(G_i),m(G'_i))> m_{G,G'}$. Thus, $(\varphi_i,\varphi'_i)$ is strongly $(\geq m_{G_i,G_i'})$-special. 

Now, for every $1\leq i\leq p$, the induction hypothesis $H(n)$ applied to the vertex groups $G_i$ and $G'_i$ together with the pair of homomorphisms $(\varphi_i,\varphi'_i)$ endows us a $(> m_{G,G'})$-expansion $(\widehat{G}_i,\widehat{G}'_i,\widehat{\varphi}_i,\widehat{\varphi}'_i)$ of $(G_i,G'_i,\varphi_i,\varphi'_i)$ such that $\widehat{\varphi}_i : \widehat{G}_i \rightarrow \widehat{G}'_i$ and $\widehat{\varphi}'_i : \widehat{G}'_i \rightarrow \widehat{G}_i$ are isomorphisms.

One defines $\widehat{G}$ from $G$ by replacing every vertex group $G_i$ by $\widehat{G}_i\supset G_i$ in the $m_{G,G'}$-JSJ splitting $T$ of $G$, and one defines $\widehat{G}'$ symmetrically. Thanks to Corollary \ref{légalité4}, the groups $\widehat{G}$ and $\widehat{G}'$ are multiple legal $(> m_{G,G'})$-extensions of $G$ and $G'$. In particular, they are multiple legal $(\geq m_{G,G'})$-extensions of $G$ and $G'$. 

We define below a strongly $(\geq m_{G,G'})$-special homomorphism $\widehat{\varphi} : \widehat{G}\rightarrow\widehat{G}'$ that coincides up to conjugacy with $\widehat{\varphi}_i$ on each subgroup $G_i$ (in particular, $\widehat{\varphi}_{\vert G}\sim \varphi$). Thanks to Lemma \ref{lemmetrois}, in order to prove that this morphism $\widehat{\varphi} $ is strongly $(\geq m_{G,G'})$-special, it suffices to prove that the fourth condition of Definition \ref{besoinreffff} holds, namely: for every finite subgroup $A$ of $\widehat{G}$ of order $\geq m_{G,G'}$, if $N_{\widehat{G}}(A)$ is virtually $\mathbb{Z}$ maximal, then $N_{\widehat{G}'}(\widehat{\varphi}(A))$ is virtually $\mathbb{Z}$ maximal as well, and the restriction of $\widehat{\varphi}$ to $N_{\widehat{G}}(A)$ is nice. 


\textbf{Construction of $\widehat{\varphi}$.} We proceed by induction on the number of edges of the $m$-JSJ decomposition $T/G$ of $G$. It is enough to construct $\widehat{\varphi}$ in the case where $T/G$ has only one edge.

\smallskip

\emph{First case.} Suppose that $G=G_1\ast_C G_2$. If $N_G(C)$ is virtually $\mathbb{Z}$, then there exists two finite subgroups $C_i\subset G_i$ such that $[C_i:C]=2$ and $N_G(C)=\langle C_1,C_2\rangle=C_1\ast_C C_2$. If $N_G(C)$ is not virtually $\mathbb{Z}$, let $C_1:=C$ and $C_2:=C$. Since $C_1$ and $C_2$ are finite, there exist two elements $g'_1,g'_2\in G'$ such that \[({\widehat{\varphi}_1})_{\vert C_1}=\mathrm{ad}(g'_1)\circ {\varphi}_{\vert C_1} \ \ \text{   and   }\ \  ({\widehat{\varphi}_2})_{\vert C_2}=\mathrm{ad}(g'_2)\circ {\varphi}_{\vert C_2}.\]

The homomorphisms $\mathrm{ad}({g'_1}^{-1})\circ \widehat{\varphi}_1$ and $\mathrm{ad}({g'_2}^{-1})\circ \widehat{\varphi}_2$ coincide on $C\subset C_1,C_2$. Hence, one can define $\widehat{\varphi}: \widehat{G}\rightarrow \widehat{G}'$ by \[{\widehat{\varphi}}_{\vert \widehat{G}_1}=\mathrm{ad}({g'_1}^{-1})\circ {\widehat{\varphi}_1} \ \ \text{   and   } \ \ {\widehat{\varphi}}_{\vert \widehat{G}_2}=\mathrm{ad}({g'_2}^{-1})\circ {\widehat{\varphi}_2}.\]

Note that $\widehat{\varphi}$ coincides with $\varphi$ on $C_1$ and on $C_2$. By Lemma \ref{lemmetrois}, in order to prove that $\widehat{\varphi}$ is strongly $(\geq m_{G,G'})$-special, we only need to prove that $\widehat{\varphi}$ satisfies the fourth condition of Definition \ref{besoinreffff}. Let $A$ be a finite subgroup of $\widehat{G}$ of order $\geq m$ such that $N_{\widehat{G}}(A)$ is virtually $\mathbb{Z}$ maximal. Since every finite subgroup of $\widehat{G}$ is conjugate to a finite subgroup of $G$, one can suppose without loss of generality that $A\subset G$. As a finite group, $A$ is elliptic in the $m$-JSJ tree $T$ of $G$, i.e.\ $A$ is contained in at least one conjugate of $\widehat{G}_1$ or $\widehat{G}_2$. There are two possibilities.

\smallskip

\emph{First possibility.} The group $A$ may be contained in only one conjugate of $\widehat{G}_1$ or $\widehat{G}_2$, which is the case for instance if $\vert A\vert > m$. Then $N_{\widehat{G}'}(\widehat{\varphi}(A))$ is virtually $\mathbb{Z}$ maximal and the restriction of $\widehat{\varphi}$ to $N_{\widehat{G}}(A)$ is nice, because $\widehat{\varphi}$ induces an isomorphism from $\widehat{G}_i$ to its image (for $i\in \lbrace 1,2\rbrace$). 

\smallskip

\emph{Second possibiliy.} The group $A$ may be contained in at least two distinct conjugates of $\widehat{G}_1$ or $\widehat{G}_2$. Then $A$ is contained in an edge group of the $m$-JSJ splitting $T$ of $G$. Since $A$ has order $m$, one can suppose without loss of generality that $A=C$. The normalizer $N_{\widehat{G}}(A)$ is finite-by-$D_{\infty}$, equal to $C_1\ast_C C_2$. By construction, $\widehat{\varphi}$ coincides with $\varphi$ on $C_1$ and $C_2$. Hence, $\widehat{\varphi}$ coincides with $\varphi$ on $N_{\widehat{G}}(A)=C_1\ast_C C_2$. Since $\varphi$ is strongly $(\geq m)$-special, $N_{G'}(\varphi(A))$ is virtually $\mathbb{Z}$ maximal, and the restriction of $\varphi$ to $N_G(A)$ is nice. Let us observe that $\varphi(A)$ is an edge group of the $m$-JSJ splitting $T'$ of $G'$. This implies that $N_{\widehat{G}'}(\varphi(A))=N_{G'}(\varphi(A))$. As a conclusion, $N_{\widehat{G}'}(\varphi(A))$ is virtually $\mathbb{Z}$ maximal and the restriction of $\widehat{\varphi}$ to $N_{\widehat{G}}(A)$ is nice.

\smallskip

\emph{Second case.} Suppose that \[G=G_1\ast_{C}=\langle G_1,t\ \vert \ tct^{-1}=\alpha(c), \ \forall c\in C\rangle.\]If $N_G(C)$ is virtually $\mathbb{Z}$ with finite center (i.e.\ finite-by-$D_{\infty}$), then $C$ and $tCt^{-1}$ are non-conjugate in $G_1$ and $C$ has index 2 in $C_1:=N_{G_1}(C)$ and $C_2:=N_{tG_1t^{-1}}(C)$. If $N_G(C)$ is not virtually $\mathbb{Z}$ with finite center, let $C_1:=C$ and $C_2:=tCt^{-1}$. Since $C_1$ and $C_2$ are finite, there exist two elements $g'_1,g'_2\in G'$ such that such that \[({\widehat{\varphi}_1})_{\vert C_1}=\mathrm{ad}(g'_1)\circ {\varphi}_{\vert C_1} \ \ \text{   and   } \ \ ({\widehat{\varphi}_1})_{\vert C_2}=\mathrm{ad}(g'_2)\circ {\varphi}_{\vert C_2}.\]One can define $\widehat{\varphi}: \widehat{G}\rightarrow \widehat{G}'$ by \[{\widehat{\varphi}}_{\vert \widehat{G}_1}={\widehat{\varphi}_1} \ \ \text{   and   } \ \ {\widehat{\varphi}}(t)=g'_2\varphi(t){g'_1}^{-1}.\]

We need to prove that $\widehat{\varphi}$ satisfies the fourth condition of Definition \ref{besoinreffff}. Let $A$ be a finite subgroup of $\widehat{G}$ of order $\geq m$ such that $N_{\widehat{G}}(A)$ is virtually $\mathbb{Z}$ maximal. One can suppose without loss of generality that $A$ is contained in $G$. As a finite group, $A$ is elliptic in the $m$-JSJ tree $T$ of $G$, i.e.\ $A$ is contained in at least one conjugate of $\widehat{G}_1$. 

\smallskip

\emph{First possibility.} The group $A$ may be contained in only one conjugate of $\widehat{G}_1$, which is the case for instance if $\vert A\vert > m$. Then $N_{\widehat{G}'}(\widehat{\varphi}(A))$ is virtually $\mathbb{Z}$ maximal, and the restriction of $\widehat{\varphi}$ to $N_{\widehat{G}}(A)$ is nice, because $\widehat{\varphi}$ induces an isomorphism from $\widehat{G}_1$ to its image. 

\smallskip

\emph{Second possibility.} The group $A$ may be contained in at least two distinct conjugates of $\widehat{G}_1$. Then $A$ is contained in an edge group of $T$. Since $A$ has order $m$, one can suppose without loss of generality that $A=C$. There are two subcases.

\smallskip

\emph{First subcase.} The groups $C$ and $tCt^{-1}=\alpha(C)$ are conjugate in $G_1$. Up to replacing $t$ with $gt$ for some $g\in G_1$, one can suppose without loss of generality that $tCt^{-1}=C$, i.e.\ $t\in N_G(C)$. Thus we have $C_1=C_2=C$ and one can suppose that $g'_1=g'_2=1$. Hence, $\widehat{\varphi}$ coincides with $\varphi$ on $C$ and on $t$. This implies that $\widehat{\varphi}$ coincides with $\varphi$ on $N_G(C)=\langle C,t\rangle$. Since $\varphi$ is strongly $(\geq m)$-special, $N_{G'}(\varphi(A))$ is virtually $\mathbb{Z}$ maximal, and the restriction of $\varphi$ to $N_G(A)$ is nice. Let us observe that $\varphi(A)$ is an edge group of the $m$-JSJ splitting $T'$ of $G'$ (see the first step of Proposition \ref{ledernierlemme}). This implies that $N_{\widehat{G}'}(\varphi(A))=N_{G'}(\varphi(A))$. As a conclusion, $N_{\widehat{G}'}(\varphi(A))$ is virtually $\mathbb{Z}$ maximal and the restriction of $\widehat{\varphi}$ to $N_{\widehat{G}}(A)$ is nice.

\smallskip

\emph{Second subcase.} The groups $C$ and $tCt^{-1}=\alpha(C)$ are not conjugate in $G_1$. Then $N_G(C)$ is $C$-by-$D_{\infty}$ and we conclude as in the first case.

\smallskip

Then, note that the morphism $\widehat{\varphi}$ constructed above is strongly $(\geq m_{G,G'})$-special, thanks to Lemma \ref{lemmetrois}. Likewise, there exists a strongly $(\geq m_{G,G'})$-special homomorphism $\widehat{\varphi}': \widehat{G}'\rightarrow \widehat{G}$ such that, for every $G'_i$, the restriction $\widehat{\varphi}'_{\vert \widehat{G}'_i}$ coincides with $\widehat{\varphi}'_i$ up to conjugacy.


Let $m=m_{G,G'}$. Let us observe that the tuple $(\widehat{G},\widehat{G}',\widehat{\varphi},\widehat{\varphi}')$ is a $(\geq m)$-expansion of $(G,G',\varphi,\varphi')$ and has property $\mathcal{P}_m$. In addition, recall that $m_{\widehat{G},\widehat{G}'}=m$, since $\widehat{G}$ and $\widehat{G}'$ are $(\geq m)$-legal extensions of $G$ and $G'$. Now, the key proposition \ref{ledernierlemme} claims that there exists a pair $(\psi : G \rightarrow G',\psi' : G'\rightarrow G)$ which is equivalent (in the sense of $\sim$) to a power of $(\widehat{\varphi},\widehat{\varphi}')$ and a $(\geq m)$-expansion $(\Gamma,\Gamma',\psi,\psi')$ of $(\widehat{G},\widehat{G}',\widehat{\varphi},\widehat{\varphi}')$ such that $\psi$ and $\psi'$ are isomorphisms. By transitivity of the relation "to be a $(\geq m)$-expansion of", the tuple $(\Gamma,\Gamma',\psi,\psi')$ is a $(\geq m)$-expansion of $(G,G',\varphi,\varphi')$. This concludes the proof.\end{proof}

\section{Algorithm}

We will prove the following result.

\begin{te}\label{algoalgo}There exists an algorithm that takes as input two finite presentations of virtually free groups, and decides whether these groups have the same $\forall\exists$-theory or not.\end{te}

Our proof relies on the main results of \cite{DG10} and \cite{DG11}.

\begin{te}\label{DG10}There exists an algorithm that takes as input a finite presentation of a hyperbolic group $G$ and a finite system of equations and inequations with constants in $G$, and decides whether there exists a solution or not.
\end{te}

\begin{te}\label{DG11}There exists an algorithm that takes as input two finite presentations of hyperbolic groups, and which decides whether these groups are isomorphic or not.\end{te}

We need some preliminary lemmas.

\subsection{Some useful algorithms}


\begin{lemme}[\cite{DG11}, Lemma 2.5]\label{algon}There is an algorithm that computes a set of generators of the normaliser of any given finite subgroup in a hyperbolic group.\end{lemme}

\begin{lemme}[\cite{DG11}, Lemma 2.8]\label{algon2}There is an algorithm that, given a finite set $S$ in a hyperbolic group, decides whether $\langle S\rangle$ is finite, virtually cyclic infinite, or non-elementary.\end{lemme}

\begin{lemme}\label{algofini}There is an algorithm that takes as input a finite presentation of a hyperbolic group and computes a list of representatives of the conjugacy classes of finite subgroups in this hyperbolic group.\end{lemme}

\begin{proof}There exists an algorithm that computes, given a finite presentation $\langle S \ \vert \ R\rangle$ of a hyperbolic group $G$, a hyperbolicity constant $\delta$ of $G$ (see \cite{Pa96}). In addition, it is well-known that the ball of radius $100\delta$ in $G$ contains at least one representative of each conjugacy class of finite subgroups of $G$ (see \cite{Bra00}). Moreover, two finite subgroups $C_1$ and $C_2$ of $G$ are conjugate if and only if there exists an element $g$ whose length is bounded by a constant depending only on $\delta$ and on the size of the generating set $S$ of $G$, such that $C_2=gC_1g^{-1}$ (see \cite{BH05}).\end{proof}

\begin{lemme}\label{algofini2}There is an algorithm that takes as input a finite presentation of a hyperbolic group $G$ and a finite subgroup $C$ of $G$ such that $N_G(C)$ is non-elementary, and decides if $E_G(N_G(C))=C$.\end{lemme}

\begin{proof}One can compute a finite generating set $S$ for $N_G(C)$ using Lemma \ref{algon}. Using the main algorithm of \cite{DG10} (see Theorem \ref{DG10} above), one can decide if the following existential sentence with constants in $G$ is satisfied by $G$: there exists an element $g\in G$ such that
\begin{enumerate}
\item the element $g$ does not belong to $C$;
\item the subgroup $C':=\langle C,g\rangle$ is finite;
\item for every $s\in S$, we have $sC's^{-1}=C'$.
\end{enumerate} 
Note that such an element $g$ exists if and only if $C$ is strictly contained in $E_G(N_G(C))$. This concludes the proof of the lemma.\end{proof}


\begin{lemme}\label{petitalgo}There is an algorithm that takes as input a finite presentation $\langle S \ \vert \ R \rangle$ of a hyperbolic group $G$, and outputs a finite list of finite presentations of all legal large extensions of $G$.
\end{lemme}

\begin{proof}Using Lemma \ref{algofini}, compute a list of representatives of the conjugacy classes of finite subgroups of $G$. For each finite group $C$ in this list, use Lemma \ref{algon2} in order to decide if $N_G(C)$ is finite, virtually cyclic infinite or non-elementary. In the case where $N_G(C)$ is non-elementary, decide if $E_G(N_G(C))=C$ by means of Lemma \ref{algofini2}. Output the finite list of finite presentations $\langle S,t \ \vert \ [t,c]=1, \forall c\in C\rangle$, for every $C$ in the previous list such that $N_G(C)$ is non-elementary and $E_G(N_G(C))=C$.\end{proof}

\begin{lemme}\label{petitalgo2}There is an algorithm that takes as input a finite presentation of a hyperbolic group $G$, and outputs a finite list of presentations of all legal small extensions of $G$.
\end{lemme}

\begin{proof}Using Lemma \ref{algofini}, compute a list of representatives of the conjugacy classes of finite subgroups of $G$. For each finite group $C$ in this list, decide if $N_G(C)$ is finite, virtually cyclic infinite or non-elementary, by means of Lemma \ref{algon2}. If $N:=N_G(C)$ is virtually cyclic infinite, enumerate the list of all virtually cyclic infinite groups $N'$ such that there exist two $K_G$-nice embeddings $\iota : N \hookrightarrow N'$ and $\iota' : N'\hookrightarrow N$. Output a list of presentations of all legal small extensions of $G$, of the form $\langle G,N' \ \vert \ \iota(n)=n, \forall n\in N\rangle.$
\end{proof}

\subsection{Proof of Theorem \ref{algoalgo}}

We are now ready to prove Theorem \ref{algoalgo}.

\begin{proof}Let $G$ and $G'$ be two virtually free groups. Let $r$ be the maximal rank of the normalizer of a finite subgroup of $G$ or $G'$, let $o$ be the maximal order of a finite subgroup of $G$ or $G'$, let $n,n'$ be the number of conjugacy classes of finite subgroups of $G$ and $G'$ respectively, and let $N=\mathrm{max}(n,n')$. By looking closely at the construction of the two isomorphic multiple legal extensions $\Gamma$ and $\Gamma'$ of $G$ and $G'$ in the proof of Proposition \ref{vandernash}, one can enumerate the number of legal extensions involved in the construction. This number of legal large extensions is bounded from above by $N(r+o!)$ (see Remark \ref{rkalgo}), and the number of legal small extensions is bounded by $N$. 

Here below is an algorithm that takes as input two finite presentations of virtually free groups and decides whether two virtually free groups $G$ and $G'$ have the same $\forall\exists$-theory or not.

\emph{Step 1.} Using Lemma \ref{algofini}, compute a list of representatives of the conjugacy classes of finite subgroups of $G$ and $G'$. Let $C_1,\ldots ,C_n, C'_1,\ldots,C'_{n'}$ denote these finite groups.

\emph{Step 2.} For each finite group $C_i$ or $C'_i$ in the previous list, use Lemma \ref{algon} to compute a set of generators $S_i$ or $S'_i$ of $N_G(C_i)$ or $N_{G'}(C'_i)$. Let $R=\mathrm{max}(\lbrace\vert S_i\vert, \vert S'_i\vert, \ 1\leq i\leq n\rbrace)$.

\emph{Step 3.} Using Lemmas \ref{petitalgo} and \ref{petitalgo2}, compute all multiple legal extensions of $G$ and $G'$ obtained by performing less than $N(R+o!)+N$ legal large of small extensions.

\emph{Step 4.} For every pair $(\Gamma,\Gamma')$ of groups computed in Step 3, use the main algorithm of \cite{DG11} (see Theorem \ref{DG11}) in order to decide if $\Gamma$ and $\Gamma'$ are isomorphic. Output "yes" if there exists such a pair, and "no" otherwise.\end{proof}

\begin{rque}Here is an aternative algorithm that decides whether two virtually free groups have the same $\forall\exists$-theory or not. Let $G$ be a hyperbolic group. By carefully looking at the sentence $\exists\forall$-sentence $\zeta_G$ defined in Section \ref{zeta}, one can see that the number of symbols in $\zeta_G$ is bounded from above by a number $n_G$ computable from a finite presentation of $G$. If $G'$ is another hyperbolic group, let us define $n:=\mathrm{max}(n_G,n_{G'})$ and let $\mathcal{A}$ be a finite set of variables of cardinality $n$. 

If $G$ and $G'$ are virtually free, the following three assertions are equivalent (note that the equivalence $(1)\Leftrightarrow (2)$ is part of Theorem \ref{principal} established previously).
\begin{enumerate}
\item $G$ and $G'$ have the same $\forall\exists$-theory.
\item $G$ satisfies $\zeta_{G'}$ and $G'$ satisfies $\zeta_G$.
\item For every $\exists\forall$-sentence $\phi$ in the language of groups over alphabet $\mathcal{A}$ involving less than $n$ symbols, $G$ satisfies $\phi$ if and only if $G'$ satisfies $\phi$.
\end{enumerate}
Since the set of $\exists\forall$-sentence in the language of groups over alphabet $\mathcal{A}$ involving less than $n$ symbols is finite, the third point above is decidable algorithmically using the main algorithm of \cite{DG10} (see Theorem \ref{DG10}).\end{rque}

\renewcommand{\refname}{References}
\bibliographystyle{alpha}
\bibliography{biblio}

\def\cprime{$'$} \def\cprime{$'$}
\begin{thebibliography}{And18b}

\bibitem[And18a]{And18}
S.~Andr\'e.
\newblock Hyperbolicity and cubulability are preserved under elementary
  equivalence.
\newblock arXiv:1801.09411, 2018.

\bibitem[And18b]{And18b}
S.~Andr\'e.
\newblock Virtually free groups are almost homogeneous.
\newblock arXiv:1810.11200, 2018.

\bibitem[Bes88]{Bes88}
M.~Bestvina.
\newblock Degenerations of the hyperbolic space.
\newblock {\em Duke Math. J.}, 56(1):143--161, 1988.

\bibitem[BF95]{BF95b}
M.~Bestvina and M.~Feighn.
\newblock Stable actions of groups on real trees.
\newblock {\em Invent. Math.}, 121(2):287--321, 1995.

\bibitem[BH05]{BH05}
M.~Bridson and J.~Howie.
\newblock Conjugacy of finite subsets in hyperbolic groups.
\newblock {\em Internat. J. Algebra Comput.}, 15(4):725--756, 2005.

\bibitem[Bra00]{Bra00}
N.~Brady.
\newblock Finite subgroups of hyperbolic groups.
\newblock {\em Internat. J. Algebra Comput.}, 10(4):399--405, 2000.

\bibitem[Cha12]{Cha12}
V.~Chaynikov.
\newblock {\em Properties of hyperbolic groups: {F}ree normal subgroups,
  quasiconvex subgroups and actions of maximal growth}.
\newblock ProQuest LLC, Ann Arbor, MI, 2012.
\newblock Thesis (Ph.D.)--Vanderbilt University.

\bibitem[Cou13]{Cou13}
R.~Coulon.
\newblock Small cancellation theory and burnside problem.
\newblock Internat. J. Algebra Comput. 24 (2014), no. 3, 251-345, 2013.

\bibitem[DG10]{DG10}
F.~Dahmani and V.~Guirardel.
\newblock Foliations for solving equations in groups: free, virtually free, and
  hyperbolic groups.
\newblock {\em J. Topol.}, 3(2):343--404, 2010.

\bibitem[DG11]{DG11}
F.~Dahmani and V.~Guirardel.
\newblock The isomorphism problem for all hyperbolic groups.
\newblock {\em Geom. Funct. Anal.}, 21(2):223--300, 2011.

\bibitem[GL11]{GL11}
V.~Guirardel and G.~Levitt.
\newblock Trees of cylinders and canonical splittings.
\newblock {\em Geom. Topol.}, 15(2):977--1012, 2011.

\bibitem[GLP94]{GLP94}
D.~Gaboriau, G.~Levitt, and F.~Paulin.
\newblock Pseudogroups of isometries of {${\bf R}$} and {R}ips' theorem on free
  actions on {${\bf R}$}-trees.
\newblock {\em Israel J. Math.}, 87(1-3):403--428, 1994.

\bibitem[Gui01]{Gui01}
V.~Guirardel.
\newblock Bounding the complexity of small actions on $\mathbb{R}$-trees.
\newblock 2001.

\bibitem[Gui04]{Gui04}
V.~Guirardel.
\newblock Limit groups and groups acting freely on {$\Bbb R^n$}-trees.
\newblock {\em Geom. Topol.}, 8:1427--1470, 2004.

\bibitem[Gui08]{Gui08}
V.~Guirardel.
\newblock Actions of finitely generated groups on {$\Bbb R$}-trees.
\newblock {\em Ann. Inst. Fourier (Grenoble)}, 58(1):159--211, 2008.

\bibitem[Hei18]{Hei18}
S.~Heil.
\newblock Test sequences and formal solutions over hyperbolic groups.
\newblock arXiv:1811.06430, 2018.

\bibitem[KM06]{KM06}
O.~Kharlampovich and A.~Myasnikov.
\newblock Elementary theory of free non-abelian groups.
\newblock {\em J. Algebra}, 302(2):451--552, 2006.

\bibitem[KPS73]{KPS73}
A.~Karrass, A.~Pietrowski, and D.~Solitar.
\newblock Finite and infinite cyclic extensions of free groups.
\newblock {\em J. Austral. Math. Soc.}, 16:458--466, 1973.
\newblock Collection of articles dedicated to the memory of Hanna Neumann, IV.

\bibitem[Mak82]{Mak82}
G.~Makanin.
\newblock Equations in a free group.
\newblock {\em Izv. Akad. Nauk SSSR Ser. Mat.}, 46(6):1199--1273, 1344, 1982.

\bibitem[Mak84]{Mak84}
G.~Makanin.
\newblock Decidability of the universal and positive theories of a free group.
\newblock {\em Izv. Akad. Nauk SSSR Ser. Mat.}, 48(4):735--749, 1984.

\bibitem[Mar02]{Mar02}
D.~Marker.
\newblock {\em Model theory}, volume 217 of {\em Graduate Texts in
  Mathematics}.
\newblock Springer-Verlag, New York, 2002.
\newblock An introduction.

\bibitem[Mer66]{Mer66}
J.~Merzljakov.
\newblock Positive formulae on free groups.
\newblock {\em Algebra i Logika Sem.}, 5(4):25--42, 1966.

\bibitem[Oge83]{Oge83}
F.~Oger.
\newblock Cancellation and elementary equivalence of groups.
\newblock {\em J. Pure Appl. Algebra}, 30(3):293--299, 1983.

\bibitem[Os93]{Ol93}
A.~Ol\cprime~shanski\u\i.
\newblock On residualing homomorphisms and {$G$}-subgroups of hyperbolic
  groups.
\newblock {\em Internat. J. Algebra Comput.}, 3(4):365--409, 1993.

\bibitem[Pap96]{Pa96}
P.~Papasoglu.
\newblock An algorithm detecting hyperbolicity.
\newblock In {\em Geometric and computational perspectives on infinite groups
  ({M}inneapolis, {MN} and {N}ew {B}runswick, {NJ}, 1994)}, volume~25 of {\em
  DIMACS Ser. Discrete Math. Theoret. Comput. Sci.}, pages 193--200. Amer.
  Math. Soc., Providence, RI, 1996.

\bibitem[Pau88]{Pau88}
F.~Paulin.
\newblock Topologie de {G}romov \'{e}quivariante, structures hyperboliques et
  arbres r\'{e}els.
\newblock {\em Invent. Math.}, 94(1):53--80, 1988.

\bibitem[Per08]{Per08}
C.~Perin.
\newblock {\em {Elementary embeddings in torsion-free hyperbolic groups}}.
\newblock Theses, {Universit{\'e} de Caen}, October 2008.
\newblock Th{\`e}se r{\'e}dig{\'e}e en anglais, avec une introduction
  d{\'e}taill{\'e}e en fran{\c c}ais.

\bibitem[Raz84]{Raz84}
A.~Razborov.
\newblock Systems of equations in a free group.
\newblock {\em Izv. Akad. Nauk SSSR Ser. Mat.}, 48(4):779--832, 1984.

\bibitem[RS94]{RS94}
E.~Rips and Z.~Sela.
\newblock Structure and rigidity in hyperbolic groups. {I}.
\newblock {\em Geom. Funct. Anal.}, 4(3):337--371, 1994.

\bibitem[RW14]{RW14}
C.~Reinfeldt and R.~Weidmann.
\newblock Makanin-razborov diagrams for hyperbolic groups.
\newblock 2014.

\bibitem[Sac73a]{Sac73b}
G.~Sacerdote.
\newblock Almost all free products of groups have the same positive theory.
\newblock {\em J. Algebra}, 27:475--485, 1973.

\bibitem[Sac73b]{Sac73}
G.~Sacerdote.
\newblock Elementary properties of free groups.
\newblock {\em Trans. Amer. Math. Soc.}, 178:127--138, 1973.

\bibitem[Sel97]{Sel97}
Z.~Sela.
\newblock Acylindrical accessibility for groups.
\newblock {\em Invent. Math.}, 129(3):527--565, 1997.

\bibitem[Sel04]{Sel04}
Z.~Sela.
\newblock Diophantine geometry over groups. {IV}. {A}n iterative procedure for
  validation of a sentence.
\newblock {\em Israel J. Math.}, 143:1--130, 2004.

\bibitem[Sel05]{Sel05a}
Z.~Sela.
\newblock Diophantine geometry over groups. {$\rm V_1$}. {Q}uantifier
  elimination. {I}.
\newblock {\em Israel J. Math.}, 150:1--197, 2005.

\bibitem[Sel06a]{Sel05b}
Z.~Sela.
\newblock Diophantine geometry over groups. {${\rm V}_2$}. {Q}uantifier
  elimination. {II}.
\newblock {\em Geom. Funct. Anal.}, 16(3):537--706, 2006.

\bibitem[Sel06b]{Sel06}
Z.~Sela.
\newblock Diophantine geometry over groups. {VI}. {T}he elementary theory of a
  free group.
\newblock {\em Geom. Funct. Anal.}, 16(3):707--730, 2006.

\bibitem[Sel09]{Sel09}
Z.~Sela.
\newblock Diophantine geometry over groups. {VII}. {T}he elementary theory of a
  hyperbolic group.
\newblock {\em Proc. Lond. Math. Soc. (3)}, 99(1):217--273, 2009.

\bibitem[Sho91]{Sho91}
H.~Short.
\newblock Quasiconvexity and a theorem of {H}owson's.
\newblock In {\em Group theory from a geometrical viewpoint ({T}rieste, 1990)},
  pages 168--176. World Sci. Publ., River Edge, NJ, 1991.

\bibitem[SW79]{SW79}
P.~Scott and T.~Wall.
\newblock Topological methods in group theory.
\newblock In {\em Homological group theory ({P}roc. {S}ympos., {D}urham,
  1977)}, volume~36 of {\em London Math. Soc. Lecture Note Ser.}, pages
  137--203. Cambridge Univ. Press, Cambridge-New York, 1979.

\end{thebibliography}

\end{document}